\documentclass[reqno, 11pt]{amsart}
\usepackage[margin=2.5cm]{geometry}
\usepackage{mathrsfs}
\usepackage{amsmath}
\usepackage{amsthm}
\usepackage{amsfonts}
\usepackage{amssymb}
\usepackage{url}
\usepackage{enumerate}
\usepackage[pdftex,bookmarks=true]{hyperref}
  \usepackage[usenames, dvipsnames]{color}
\usepackage{verbatim}
\usepackage{bbm}
\usepackage{pdfsync}
\usepackage{graphicx}
\usepackage{xcolor}
\usepackage{cleveref}

\newcommand\R{{\mathbb{R}}}
\newcommand\C{{\mathbb{C}}}

\renewcommand\P{{\mathbb{P}}}
\newcommand\E{{\mathbf{E}}}

\newcommand\Var{{\operatorname{Var}}}

\newcommand\Z{{\mathbb{Z}}}

\newcommand\I{{\mathbf{I}}}

\newcommand\var{{ \operatorname{Var}}}

\newcommand\ep{\varepsilon}
\newcommand\al{\alpha}

\newcommand\Bt{{\mathbf t}}
\newcommand\Bu{{\mathbf u}}
\newcommand\Bv{{\mathbf v}}

\newcommand\Bx{{\mathbf x}}
\newcommand\By{{\mathbf y}}
\newcommand\Bz{{\mathbf z}}

\newcommand\BD{{\mathbf D}}

\newcommand\BI{{\mathbf I}}

\newcommand\BN{{\mathbf N}}

\newcommand\BW{{\mathbf W}}

\newcommand\CE{{\mathcal E}}

\newcommand\CR{{\mathcal R}}

\newcommand\eps{\varepsilon}

\newcommand\bs{\backslash}

\newcommand\cov{{\operatorname{Cov}}}
\newcommand\Cov{{\operatorname{Cov}}}

\newcommand{\wb}{\overline}
\newcommand{\epsg}{\varepsilon_{\operatorname{g}}}

\newcommand\til{\widetilde}

\newcommand\nc\newcommand
\nc\dmo\DeclareMathOperator
\renewcommand\bs\boldsymbol
\nc\tP{{\til P}}
\nc{\xxi}{\xi}
\nc{\ii}{{\sqrt{-1}}}
\dmo{\Leb}{{Leb}}
\nc{\wt}{\widetilde}
\nc{\mLeb}{m_{\Leb}}
\nc{\tran}{{\mathsf{T}}}
\nc{\mom}{{m}}
\nc{\avec}{{\bs a}}
\nc{\bvec}{{\bs b}}
\nc{\uvec}{{\bs u}}
\nc{\vvec}{{\bs v}}
\nc{\wvec}{{\bs w}}
\nc{\ul}{\underline}
\nc{\us}{{\bs{s}}}
\nc{\bbs}{{\bs{s}}}

\renewcommand\wt{\widetilde}

\newcommand\bxi{\boldsymbol{\xi}}

\theoremstyle{plain}
  \newtheorem{theorem}[subsection]{Theorem}
  \newtheorem{conjecture}[subsection]{Conjecture}

  \newtheorem{prop}[subsection]{Proposition}
    
  \newtheorem{fact}[subsection]{Fact}
  \newtheorem{lemma}[subsection]{Lemma}
  \newtheorem{corollary}[subsection]{Corollary}
  
 \newtheorem{question}[subsection]{Question}
  \newtheorem{example}[subsection]{Example}
  \newtheorem{condition}{Condition}
  \newtheorem{remark}[subsection]{Remark}
  
  \newtheorem{claim}[subsection]{Claim}

\theoremstyle{definition}
  \newtheorem{definition}[subsection]{Definition}

\begin{document}

\title[Real  roots of random Weyl polynomials]{Real roots of random Weyl polynomials with general coefficients: expectation and variance}

\author{Ander Aguirre}
\author{Hoi H. Nguyen}
\author{Jingheng Wang}

\address{Department of Mathematics\\ University of Wisconsin Madison\\ 480 Lincoln Dr.\\ Madison, WI 53706 USA}
\email{aguirrezarat@wisc.edu}
\address{Department of Mathematics\\ The Ohio State University \\ 231 W 18th Ave \\ Columbus, OH 43210 USA}
\email{nguyen.1261@math.osu.edu}
\address{Department of Mathematics\\ The Ohio State University \\ 231 W 18th Ave \\ Columbus, OH 43210 USA}
\email{wang.14053@osu.edu}

\thanks{H. Nguyen is supported by a Simons Travel Grant TSM-00013318.}

\begin{abstract} 
In this paper, we investigate the number of real zeros of random Weyl polynomials
of degree \(n \to \infty\) with general coefficient distributions.
Motivated by the results of~\cite{DHNgV, NgNgV} as well as~\cite{BCP, DNN},
we determine how the expected number of real zeros and their variance,
over various natural intervals, depend on the moments of the common coefficient distribution.
Our main finding is that while the first-order asymptotic of the expectation is universal,
the next-order correction depends on the third and fourth moments of the distribution,
and may grow linearly with \(\log n\), depending on the interval under consideration.
In contrast, for the variance we show that the leading-order term is universal,
which differs from the behavior observed for random trigonometric polynomials
in~\cite{BCP, DNN}. Our approach relies on an \emph{Edgeworth expansion} for random walks
arising from Weyl polynomials, a result of independent interest.

\end{abstract}

\maketitle



\section{Introduction}\label{sect:intro}



Over the past few years, there have been active developments to study various statistics of the number $N_\R(F_{n})$ of real roots of a random polynomial $F_{n}$. In its general setting, the random polynomial $F_{n}$ takes the form
\begin{equation}\label{eqn:f:general}
F_{n}(x)=\sum_{j=0}^n \xi_j p_j(x),
\end{equation}
where $\xi_j$ are iid copies of a random variable $\xi$ of mean zero and variance one, and $p_j(x)$ are deterministic polynomials (of degree $j$) coming from various natural sources. Below we list a few typical examples of $F_{n}$, 
\begin{enumerate}[(i)]
\item {\it Kac polynomials}: $p_j(x)=x^j$; 
\item {\it Elliptic polynomials}: $p_j(x) = \sqrt{\binom{n}{j}}x^j$; 
\item {\it Weyl polynomials}: $p_j(x) = \frac{1}{\sqrt{j!}} x^j$;
\item and {\it trigonometric polynomials}: $p_j(x) =\cos (jx)$ or $\sin(jx)$ or a combination of both, and more generally, {\it orthogonal polynomials}, when $p_j(x)$ is a polynomial of degree $j$ and $\{p_j(x)\}_{j=0}^n$ forms an orthonormal basis with respect to a smooth Borel measure $\mu$ on $\R$. 
\end{enumerate}

The zeros and critical points of random functions have significant practical applications in diverse areas such as algebraic geometry, numerical analysis and computational complexity, probability, signal processing and control theory, and statistical mechanics and disordered systems, to name a few.   One of the most common choices for the coefficients $\xi_j$ is the standard gaussian distribution, in which case we refer to $F_n$ as a gaussian polynomial. In this setting, many interesting and profound results concerning the statistics of the zeros of $F_n(x)$ have been extensively investigated.

The celebrated Kac-Rice formula asserts that for polynomials $F_{n}$ of gaussian coefficients, the number of real zeros in an interval $I$ can be expressed as $N_I = \int_I \rho_{1,G}(x)dx$ where the first intensity function $\rho_{1,G}$ \footnote{Here and later the letter $G$ stands for standard gaussian case.} bears a simple formula
\begin{align*}
    \rho_{1,G}(x)=\frac{1}{\pi}\sqrt{\frac{\partial^2}{\partial s \partial t}\log K(s,t)|_{s=t=x}}, \mbox{ where } K(s,t)=\E F_{n}(s)F_{n}(t).
\end{align*}

For instance we have the following:
\begin{itemize}

    \item For Kac polynomials (see for instance \cite{EK})
     \begin{align}
         \rho_{1,G}(x)=\frac{1}{\pi}\sqrt{\frac{1}{(x^2-1)^2}-\frac{(n+1)^2x^{2n}}{(x^{2n+2}-1)^2}};
    \end{align}
        \vskip 5mm
    \item For Elliptic polynomials (see for instance \cite{EK}) 
    \begin{align} 
           \rho_{1,G}(x)=\frac{\sqrt{n}}{\pi (1+x^2)};
    \end{align}
            \vskip 5mm

   \item For Weyl polynomials (see for instance \cite{MS}) 
    \begin{align}\label{eqn:rho:W}  
           \rho_{1,G}(x) = \frac{1}{\pi} \sqrt{ 1 + \frac{x^{2n} (x^2-n-1)}{e^{x^2} \Gamma(n+1,x^2)}  - \frac{x^{4n+2}}{[e^{x^2} \Gamma(n+1, x^2)]^2}},
     \end{align}
    where $\Gamma(n,x) = \int_x^\infty e^{-t} t^{n-1} dt$, from which we see that $\rho_1(x) \approx \frac{1}{\pi}$ for $|x| \le (1-o(1))\sqrt{n}$ and $\rho_1(x) \approx \sqrt{n}/\pi x^2$ if $x \ge (1+o(1))\sqrt{n}$;
        \vskip 5mm
     \item For (stationary) trigonometric polynomials on $[0,2\pi]$ where $P_n(t) 
        = \frac{1}{\sqrt{n}} \sum_{i=1}^n \big( \xi_{i1} \cos(it) + \xi_{i2} \sin(it) \big)$ with iid standard gaussian $\xi_{ij}$ (see for instance \cite{AL})
   \begin{align}
           \rho_{1,G}(x)=  \frac{1}{\pi}\sqrt{(n+1)(2n+1)/6}.
    \end{align}

\end{itemize}

More generally, for each $1\le k\le n$, let $\rho_{k}(x_1,\dots, x_k)$ (or $\rho_{k,\bxi}(x_1,\dots, x_k)$ to emphasize the case of general coefficients $\xi_{i}$) be the $k$-correlation of the real roots of $F_{n}(x)$, for which (see for instance \cite{HKPV})
$$\E \big[\sum H(\zeta_{i_1}, \dots, \zeta_{i_k})\big]  = \int_{\R^k} H(x_1,\dots, x_k) \rho_{k,\bxi}(x_1,\dots, x_k) dx_1 \dots dx_k,$$ 
for any continuous, compactly supported test function $H : \R^k \to \R$, where the sum runs over all $k$-tuples $(\zeta_{i_1}, \dots, \zeta_{i_k})$ of the real roots of $F_{n}(x)$.

In principle one can use Kac-Rice formula to compute these correlation functions for gaussian polynomials, that $\rho_{k,G}(x_1,\dots, x_k) = \int_{\R^k} |y_1,\dots, y_k|p(\mathbf{0},\By) dy_1\dots dy_k$, where $p(.)$ is the joint density function of the random vectors $(F_{n}(x_1),\dots, F_{n}(x_k))$ and $(F_{n}'(x_1),\dots, F_{n}'(x_k))$, we refer the reader to \cite{BD} for detailed formulas for many ensembles.
 
Another fascinating aspect of the theory of random polynomials is its \emph{universality} with respect to the distribution of random coefficients. 
This phenomenon has been verified through the works of Kabluchko--Zaporozhets \cite{KZ1} at the global scale, and Tao--Vu \cite{TV}, Do--O.~Nguyen--Vu \cite{DONgV}, and O.~Nguyen--Vu \cite{ONgV} at the local scale. 
Roughly speaking, these results establish that the fine-scale statistics of zeros are insensitive to the specific distribution of the coefficients. Allow us to cite here an informal version of these particular results.

\begin{theorem}[Local universality of correlations]\label{theorem:universality}
Let \( \xi_j \) be i.i.d.\ copies of a random variable \( \xi \) with mean zero, variance one, and bounded \((2+\varepsilon)\)-moment. 
Then the local correlation functions \( \rho_{k,\bxi}(\cdot) \) of the zeros of classical ensembles such as Kac, Weyl, and Elliptic polynomials, as well as orthogonal polynomial ensembles with weight \( w(x)=d\mu(x) \) satisfying suitable smoothness conditions, are asymptotically identical to those in the gaussian case. 
More precisely, for any smooth test function \( H:\mathbb{R}^k \to \mathbb{R} \),
\[
\int_{\mathbb{R}^k} H(x_1,\dots,x_k)\, \rho_{k,\boldsymbol{\xi}}(x_1,\dots,x_k)\, dx_1 \cdots dx_k
= 
\int_{\mathbb{R}^k} H(x_1,\dots,x_k)\, \rho_{k,G}(x_1,\dots,x_k)\, dx_1 \cdots dx_k
+ o_H(1).
\]
\end{theorem}
We remark that in many cases the error term $o_H(1)$ can be replaced by $O_H(n^{-c})$ for some small constant \( c>0 \) depending only on the ensemble, on \( \xi \), and on the support of $H$.

Although these results have resolved one of the most fundamental aspects of the theory of random polynomials and the underlying methods are remarkably robust—allowing one to treat almost all natural ensembles in a unified manner—many intriguing and important directions remain open for exploration.

\subsection{More detailed statistics: variances and fluctuations}

It follows from Theorem~\ref{theorem:universality} (except for the Kac polynomial case) that for some \( c>0 \) depending  on the ensemble and \( \xi \).

\[
\mathbb{E}_{\boldsymbol{\xi}} N_{I, \bxi}^k = \big(1 + o(1)\big)\, \mathbb{E}_G N_{I,G}^k.
\]
In particular,
\begin{equation}\label{eqn:exp:0}
\mathbb{E}_{\boldsymbol{\xi}} N_{I, \bxi} = \big(1 + o(1)\big)\, \mathbb{E}_G N_{I, G}.
\end{equation}
We therefore obtain a rather precise expression for \( \mathbb{E}_{\boldsymbol{\xi}} N_I \), up to a multiplicative correction factor of \( 1 + O(n^{-c}) \).

\begin{corollary}[{\cite{DONgV, TV}}]\label{cor:E:gen}
Assume that \( \xi_j \) are i.i.d.\ copies of a random variable \( \xi \) with mean zero, variance one, and bounded \( (2+\varepsilon) \)-moment. Then the following hold for the number of (all) real roots.
\begin{itemize}
    \item {Kac polynomials:}
    \[
    \mathbb{E} N_{\mathbb{R},\bxi} = (1 + o(1))\, \frac{2}{\pi} \log n, 
    \qquad 
    \mathrm{Var}(N_{\mathbb{R},\bxi}) = \Big( \frac{4}{\pi} \Big( 1 - \frac{2}{\pi} \Big) + o(1) \Big) \log n;
    \]
    
    \item {Elliptic polynomials:}
    \[
    \mathbb{E} N_{\mathbb{R},\bxi} = (1+O(n^{-c})) \sqrt{n}, 
    \qquad 
    \mathrm{Var}(N_{\mathbb{R},\bxi}) = O(n^{1-c});
    \]
    
    \item {Weyl polynomials:}
    \[
    \mathbb{E} N_{\mathbb{R},\bxi} = ( \frac{2}{\pi} + O(n^{-c})) \sqrt{n}, 
    \qquad 
    \mathrm{Var}(N_{\mathbb{R},\bxi}) = O(n^{1-c});
    \]
    
    \item {Trigonometric polynomials:}
    \[    \mathbb{E} N_{[-\pi,\pi],\boldsymbol{\xi}} = \Big( \frac{2}{\sqrt{3}} + O(n^{-c}) \Big) n,
    \qquad 
    \mathrm{Var}\big(N_{[-\pi,\pi],\boldsymbol{\xi}}\big) = O(n^{2-c}).
    \]
\end{itemize}
\end{corollary}

While the above results are already quite satisfactory, several important questions remain open.  
\begin{enumerate}
    \item Can one obtain sharper estimates for \( \mathbb{E} N_{\R,\bxi} \)?
    \vskip .1in
    \item Can one provide more accurate asymptotics for \( \mathrm{Var} N_{\R, \bxi} \)?
\end{enumerate}

These problems point toward a finer understanding of the fluctuations of the number of real zeros of random polynomials beyond the leading-order universality regime.

{\it \underline{1. Detailed analysis of the expectation}}. For the Kac model, the expectation \( \mathbb{E} N_{\mathbb{R},\boldsymbol{\xi}} \) was studied in detail in \cite{DHNgV, NgNgV}.  
In the gaussian case, results of Wilkins \cite{Wil}, later rediscovered by Edelman and Kostlan \cite{EK}, provide a refined asymptotic expansion
\[
\mathbb{E} N_{\mathbb{R},G} =  \frac{2}{\pi} \log n  + C_{G}  + o(1),
\]
where \( C_{G} \approx 0.625738072 \) is an explicit constant given by
\[
C_G = \frac{2}{\pi} \left( \log 2 + \int_{0}^\infty 
\left( \sqrt{\frac{1}{x^2} - \frac{4 e^{-2x}}{(1 - e^{-2x})^2}} - \frac{1}{x+1} \right) dx \right).
\]
It turns out that similar refined asymptotics also hold beyond the gaussian case.  
As shown in \cite{DHNgV, NgNgV, ONgP}, such results extend to the Rademacher case (\( \xi = \pm 1 \) with equal probability) and, more generally, to broader coefficient distributions.  
We cite here the most recent development from \cite{ONgP}.

\begin{theorem}\label{thm:expect1}
Assume that \( \mathbb{E}|\xi|^{2+\varepsilon_0} < \infty \) for some fixed \( \varepsilon_0 > 0 \).  
Then there exists a constant \( C = C_{\bxi} \) such that
\begin{equation}\label{Kac:C}
\mathbb{E} N_{\mathbb{R},\boldsymbol{\xi}} = \frac{2}{\pi} \log n + C_{\bxi} + o(1).
\end{equation}
\end{theorem}

The constants \( C_{\boldsymbol{\xi}} \) are observed numerically to differ from \( C_G \) for various choices of the coefficient distribution.  
It is of considerable interest to obtain a more precise formula for \( C_{\boldsymbol{\xi}} \) and to understand how \( C_{\boldsymbol{\xi}} \) depends quantitatively on the law of \( \boldsymbol{\xi} \).

As far as current results are concerned, beyond the Kac model, there has been no comparable development for other ensembles \footnote{Nevertheless, we believe that the techniques developed in \cite{DNN} can be adapted to study the finer asymptotics of \( \mathbb{E} N_{[-\pi,\pi],\boldsymbol{\xi}} \) for the trigonometric polynomial model and related settings.}.

\begin{question}\label{conj:exp}
Is it true that for all of the models discussed above, for most natural intervals $I$ 
\[
\mathbb{E} N_{I,\boldsymbol{\xi}} 
= \mathbb{E} N_{I,G} + C_{\xi} + o(1),
\]
for some constant \( C_{\xi} \) depending on the distribution of \( \xi \)?
\end{question}

{\it \underline{2. Detailed analysis of the variances.}}  
The situation is more delicate for the variances, especially beyond the Kac model.  
For the gaussian ensembles, several sharp asymptotic results are now known.

\begin{theorem}\label{thm:variance:G}
Assume that \( \xi_j \) are i.i.d.\ standard gaussian random variables. Then the following hold.
\begin{itemize}
    \item {Elliptic polynomials:}  
    it was shown in \cite{BD, NhNg} that
    \begin{equation}\label{eqn:Var:E}
        \mathrm{Var}(N_{\mathbb{R},G})
        = \frac{2}{\pi} C_E \sqrt{n},
    \end{equation}
    where \( C_E \approx 0.5717310486 \).
    \vskip .1in
    \item {Weyl polynomials:}  it was shown in \cite{DV,MS} that there exists a positive constant \( C_W \) such that for any interval \( I \subset (-n^{1/2} - o(n^{1/4}),\, n^{1/2} + o(n^{1/4})) \) with length \( |I| \to \infty \) as \( n \to \infty \),
       $$     \Var(N_{I,G})=(C_W+o(1))|I|,$$
    where  \( C_W = 0.18198\ldots \).
        \vskip .1in
    \item {(stationary) trigonometric polynomials:} it was shown in \cite{GW} that
    \[
        \mathrm{Var}\big(N_{[-\pi,\pi],G}\big)
        = (C_{\mathrm{T}} + o(1))\, n,
    \]
    where \( C_{\mathrm{T}} \approx 0.55826 \).
\end{itemize}
\end{theorem}

\begin{remark}
All of the constants appearing above are explicit.

\begin{itemize}
    \item {Elliptic polynomials}: the variance constant was computed in \cite{BD, NhNg}
    \[
        C_{E}= \frac{2}{\pi} \int_0^{\infty} f_0(s)\, ds,
    \]
    where
    \begin{align*}
        \delta_0(s)
        &= \frac{e^{-s^2/2}\bigl(1 - s^2 - e^{-s^2}\bigr)}{1 - e^{-s^2} - s^2 e^{-s^2}}, \\
        \gamma_0(s)
        &= \frac{1 - e^{-s^2} - s^2 e^{-s^2}}{(1 - e^{-s^2})^{3/2}}, \\
        f_0(s)
        &= 1 - \delta_0(s)^2 + \delta_0(s)\, \arcsin \delta_0(s)\, \gamma_0(s)^{-1}.
    \end{align*}
\vskip .1in
    \item {Weyl polynomials}: the constant \( C_{W} \) can be computed explicitly as
    \begin{equation*}\label{eq:K}
        C_{W} = \frac{1}{\pi} \int \left( \rho(0,t) - \frac{1}{\pi^2} \right) dt,
    \end{equation*}
    where \( \rho(s,t) \) denotes the two-point correlation function for the real zeros of \( P_\infty \), which can be obtained via the Kac--Rice formula \cite[Appendix~C]{DV}:
    \[
        \rho(0,t)
        = \frac{(1 - e^{-t^{2}})^{2} - t^{4} e^{-t^{2}}}
               {\pi^{2} (1 - e^{-t^{2}})^{3/2}}
           \left(
              1 + \frac{\delta\, \arcsin \delta}{\sqrt{1 - \delta^{2}}}
           \right),
        \qquad
        \delta
        = \frac{e^{-t^{2}/2} \bigl(e^{-t^{2}/2} + t^{2} - 1\bigr)}
               {1 - e^{-t^{2}} - t^{2} e^{-t^{2}}}.
    \]
   \vskip .1in

    \item {(stationary) trigonometric polynomials}: it is shown in \cite{GW} that
    \[
        C_T = \frac{4}{3\pi} \int_{0}^{\infty}
        \left(
            \frac{1 - g(t)^2 - 3 g'(t)^2}{(1 - g(t)^2)^{3/2}}
            \bigl(\sqrt{1 - {R^\ast}^2} + R^\ast \arcsin R^\ast\bigr)
            - 1
        \right) dt + \frac{2}{\sqrt{3}},
    \]
    where
    \[
        g(t) = \frac{\sin t}{t},
        \qquad
        R^\ast = R^\ast(t)
        = \frac{ g''(t)\bigl(1 - g(t)^2\bigr) + g(t)\, g'(t)^2 }
               { \tfrac{1}{3}\bigl(1 - g(t)^2\bigr) - g'(t)^2 }.
    \]
\end{itemize}
\end{remark}

Comparing the results from Corollary \ref{cor:E:gen} and Theorem \ref{thm:variance:G}, it is therefore natural to make the following conjecture.

\begin{conjecture}[Linearity of the variance]\label{conj:var}
For all models above, as long as \( \xi_i \) are i.i.d.\ copies of a “nice’’ random variable \( \xi \) with mean zero and variance one, there exists a constant \( C_{\boldsymbol{\xi}} \) (depending on the model) such that for most natural intervals $I$
\[
    \mathrm{Var}\, N_{I,\boldsymbol{\xi}}
    = \big(C_{\boldsymbol{\xi}} + o(1)\big)\,
      \mathbb{E} N_{I,\boldsymbol{\xi}}.
\]
\end{conjecture}

In a way, this conjecture suggests that even for general coefficient distributions, the real roots spread out over \(I\) in a sufficiently regular manner so that the numbers of roots in well-separated intervals become asymptotically independent within the CLT regime.

Beyond the Kac model (as seen from Corollary \ref{cor:E:gen}), the above phenomenon has recently been confirmed for the trigonometric ensembles by Bally, Caramellino, and Poly~\cite{BCP}, and by Do, H.~Nguyen, and O.~Nguyen~\cite{DNN}.

\begin{theorem}\label{thm:var:trig}
Assume that \( \xi_{ij} \), \( 1 \le i \le n,\, j = 1,2 \), are i.i.d.\ copies of a random variable \( \xi \) with mean zero, variance one, and \( \mathbb{E}|\xi|^{M_0} < \infty \) for some sufficiently large constant \( M_0 > 0 \).  
Then
\[
    \lim_{n \to \infty} \frac{1}{n} 
    \mathrm{Var}\big(N_{[-\pi, \pi], \boldsymbol{\xi}}\big)
    = C_T + \frac{2}{15}\, \mathbb{E}(\xi^4 - 3),
\]
where \( C_T \) is the constant from Theorem~\ref{thm:variance:G}.
\end{theorem}

We may interpret Theorem~\ref{thm:var:trig} as exhibiting a subtle interplay between universality and non-universality.  
The linear growth of the variance in \( n \) reflects the absence of long-range correlation (repulsion or attraction) among sufficiently distant zeros. However, the multiplicative constant, which depends on the short-range correlation of nearby zeros, is sensitive to the kurtosis of the underlying coefficient distribution, as revealed by the appearance of the term \( \mathbb{E}(\xi^4 - 3) \).

\subsection{Main results for Weyl polynomials}

The goal of this note is to focus on the \emph{Weyl polynomial ensemble}
\[
    P_{n}(x)=\sum_{i=0}^{n} \xi_i \frac{x^i}{\sqrt{i!}},
\]
(or its normalized variant \(e^{-x^{2}/2}\sum_{i=0}^{n} \xi_i \frac{x^i}{\sqrt{i!}}\)).

When the coefficients $\xi_i$ are i.i.d.\ complex Gaussian, the random polynomial $F_n$ can be viewed as a truncation of a \emph{Gaussian Entire Function} (see also Section~\ref{section:discussion}), whose zeros are known to be invariant under plane isometries~\cite{HKPV}.  
Zeros of gaussian and non-gaussian random polynomials have also been extensively investigated in the mathematical physics literature, where they serve as canonical models for chaotic spectra and random wavefunctions. We refer the reader to the influential works of Bogomolny--Bohigas--Leboeuf~\cite{BBL1,BBL2}, Leboeuf~\cite{Leb}, Majumdar--Schehr~\cite{MS}, and Nazarov--Sodin~\cite{NS-complex}, as well as the references therein for further developments.  Moreover, the Weyl polynomial model bears a close resemblance to the characteristic polynomial of random non-symmetric matrices, see~\cite{FG}, and thus plays a bridging role between random functions and random matrix theory.

Our aim in this note is to partially address Question~\ref{conj:exp}
and Conjecture~\ref{conj:var} for this intriguing random polynomial model. For simplicity, we will be focusing on non-negative real zeros within the interval $[0,\sqrt{n}]$, although our approach also extend to negative real zeros.

Let  \(0 < c_{1} < c_{2}\) be fixed constants, and consider the interval
\begin{equation}\label{eqn:I_{W}:gen}
    I_{W} = [c_{1} M,\, c_{2} M] \subset [0, n^{1/2} -M],
\end{equation}
where $M$ is sufficiently large, which might also grow with $n$. 

\begin{theorem}[Main result: expectation for real zeros of Weyl polynomials]\label{thm:expectation:W}
Let \( \xi_i \) be i.i.d.\ copies of a mean-zero, variance-one subgaussian random variable \( \xi \).  
\begin{itemize}

\item (In the bulk) The expected number of real zeros of Weyl polynomials over \( I_W \) from \eqref{eqn:I_{W}:gen} satisfies
\begin{equation}\label{eqn:expmain:1}
    \mathbb{E} N_{I_W,\boldsymbol{\xi}}
    = \mathbb{E} N_{I_W,G}
      + C_{\xi} \log\!\left(\frac{c_{2}}{c_{1}}\right)
      + o(1),
\end{equation}
where \( C_{\xi} \) is a constant depending on the third and fourth moments of \( \xi \), given by
\[
    C_{\xi}
    = C_1\, \mathbb{E}(\xi^4 - 3)
      + C_2\, \big(\mathbb{E}(\xi^3)\big)^{2},
\]
with
\[
    C_1 = -\frac{7}{192\pi\sqrt{\pi}},
    \qquad
    C_2 = \frac{\sqrt{2}}{12\pi\sqrt{\pi}}.
\]
\item (Including the soft edge) Let $0<c<1$ be a given constant. The expected number of real zeros of Weyl polynomials over \([0,(1-c)\sqrt{n}]\) satisfies
\begin{equation}\label{eqn:expmain:2}
    \mathbb{E} N_{[0,(1-c)\sqrt{n}],\boldsymbol{\xi}}
    = \mathbb{E} N_{[0,(1-c)\sqrt{n}],G} + (C_{\xi}/2+o(1)) \log n.
\end{equation}
\end{itemize}
\end{theorem}
In other words, our result reveals an interesting feature that within the bulk, the difference of $\mathbb{E} N_{I_W,\boldsymbol{\xi}} -\mathbb{E} N_{I_W,G}$ is a constant (depending on $\xi$), while over the almost entire interval $[0,(1-c)\sqrt{n}]$ there is a $\log n$ term. This latter behavior indicates that by varying the distribution of $\xi$ from gaussian to non-gaussian, one can induce a substantial change in the expected number of real zeros.

In the case of the variance, we verify Conjecture~\ref{conj:var} for Weyl polynomials across a range of natural intervals. Moreover, our findings show that the corresponding multiplicative constant is universal in the sense that it does not depend on high order moments of~\(\xi\).

\begin{theorem}[Main result: variance for real zeros of Weyl polynomials]\label{thm:var:W} Let $0 <\sigma_\ast <1/2$ be a  given constant. Then as $n\to \infty$, the variance of the number of real zeros of Weyl polynomials over \( I_W \) from \eqref{eqn:I_{W}:gen}, where $M \ge n^{\sigma_{\ast}}$, satisfies
\begin{equation}\label{eqn:varmain:1}
    \mathrm{Var}\, N_{I_W,\boldsymbol{\xi}}
    = \mathrm{Var}\, N_{I_W,G} + o(M).
\end{equation}
Furthermore, for any given $0<c<1$
\begin{equation}\label{eqn:varmain:2}
    \mathrm{Var}\, N_{[0,(1-c)\sqrt{n}],\boldsymbol{\xi}}
    = \mathrm{Var}\, N_{[0, (1-c)\sqrt{n}],G} + o(\sqrt{n}).
\end{equation}
\end{theorem}

Thus, in contrast with Theorem~\ref{thm:var:trig}, the leading constant in the variance for Weyl polynomials is essentially \emph{insensitive} to the higher moments of~\(\xi\).\footnote{Lower-order terms within the \(o(\sqrt{n})\) remainder may still depend on these moments; a precise analysis of this dependence lies beyond the scope of the present paper.}

This observation suggests that, in pursuing a refined understanding of the variance toward Conjecture~\ref{conj:var}, there is unlikely to exist a single, model-independent approach---unlike Theorem~\ref{theorem:universality}---that applies uniformly across all random polynomial models. Instead, each model must be examined individually to uncover its intrinsic structural features and specific sources of non-universality.

In summary, our result provides a substantially more detailed description of the statistics of real zeros of Weyl polynomials (see Corollary~\ref{cor:E:gen}, or more precisely \cite[Theorem~5.3]{TV} and \cite[Section~12.1]{TV}). It would be particularly interesting to extend this analysis to the entire interval~\([0,\sqrt{n}]\), a regime where our current control of the characteristic functions near the hard edge~\(\sqrt{n}\) remains slightly incomplete. Another natural and challenging direction for future work is to establish a central limit theorem describing the fluctuations of~\(N_{I_W,\boldsymbol{\xi}}\). Further discussion on the approach to the CLT fluctuation is given in Section~\ref{section:discussion}.

{\bf Notation.}  We will assume $n \to \infty$ throughout the note. We write $X =
O(Y)$, 
 $X \ll Y$, or $Y \gg X$ if $|X| \leq CY$
for some absolute constant $C$. The constant $C$ may depend on some parameters, in which case we write
e.g. $Y=O_\tau(X)$ if $C=C(\tau)$.  We write $X \asymp Y$ or $X = \Theta(Y)$ if $X \gg Y$ and $Y \gg X$.
In what follows, $\|.\|_{\R/\Z}$ is the distance to the nearest integer, and $m=\mLeb(\cdot)$ is the Lebesgue measure. 

\subsection{Method discussion: Kac-Rice formula and obstacles}\label{sect:kr} Our overall method is similar to \cite{BCP, DNN}. First, we recall from \eqref{eqn:rho:W} by Kac-Rice formula for Weyl polynomial that

     \begin{equation*}    \rho_{1,n}(x) = \frac{1}{\pi} \sqrt{ 1 + \frac{x^{2n} (x^2-n-1)}{e^{x^2} \Gamma(n+1,x^2)}  - \frac{x^{4n+2}}{[e^{x^2} \Gamma(n+1, x^2)]^2}},
     \end{equation*}
    where $\Gamma(n,x) = \int_x^\infty e^{-t} t^{n-1} dt$, from which we see that $\rho_{1,n}(x) \approx \frac{1}{\pi}$ for $|x| \le (1-o(1))\sqrt{n}$ and $\rho_{1,n}(x) \approx \sqrt{n}/(\pi x^2)$ if $x \ge (1+o(1))\sqrt{n}$.
  
We next briefly recall the use of approximated Kac-Rice, needed to deal with non-gaussian coefficients, following the method of \cite{DNN}. Consider a smooth function $f$ on an interval $[a,b]$ where for all $x\in [a, b]$ we have $|f(x)| + |f'(x)|>\delta$ and $|f(a)|, |f(b)| > \delta$. Then the number of roots of $ f $ in $ [a, b] $ is approximated by:
\begin{equation}\label{KR-approx}
    N_{[a,b]} \stackrel{\text{def}}{=}\frac{1}{2\delta} \int_a^b |f'(x)| 1_{|f(x)|<\delta} dx.\end{equation}

Under technical condition that the above holds for some $\delta$ (where we refer the reader to Lemma \ref{lemma:smallball:KR} of Section \ref{sect:sbp} for a justification of this) we can write the expectation and variance as 
$$\E N_{I_{W},\bxi}=  \E\left ( \frac{1}{2\delta} \int_{I_{W}} |P'_n(x,\bxi)| 1_{ |P_n(x)|<\delta }\right )dx.$$
and
\begin{equation}\label{eqn:variance:KR1}
\frac{1}{N} \Var(\E N_{I_{W},\bxi}) = \frac{1}{N} \Var\left ( \frac{1}{2\delta} \int_{I_{W}} |P'_n(x,\bxi)| 1_{ |P_n(x)|<\delta }\right )dx,
\end{equation}
where
$$N:= \sqrt{M},$$ 
which is the order of the variance and expectation in the gaussian case over $I_{W}$ from \eqref{eqn:I_{W}:gen}. 

After expanding out the integrals, we will need to compute 
$$\frac{1}{\delta^2}\int_{I_{W}} \int_{I_{W}} \cov(|P'_n(x,\bxi)| 1_{ |P_n(x)|<\delta }, |P'_n(y,Y)| 1_{ |P_n(y,Y)|<\delta }) dxdy.$$
Let  us introduce some notation to simplify the discussion. First, for $x,y\in I_{W}$, consider the $2$-d and $4$-d random walks 
\begin{equation}\label{Sx}
S_n(x,\bxi):=\sum_{i=1}^n\xi_i(b_i(x), b'_i(x))
\end{equation}
and
\begin{equation}\label{Sxy}
S_n(x,y,\bxi):=\sum_{i=1}^n\xi_i(b_i(x), b'_i(x), b_i(y), b'_i(y)),
\end{equation}
where in our Weyl polynomial setting \footnote{As we will see, $|p_{i}(x)|$ has order $1/\sqrt{N}$ for about $\Theta(N)$ indices $i$.} 
\begin{equation}\label{eqn:b_i}
b_i(x)=\sqrt{N}e^{-x^2/2}\frac{1}{\sqrt{i !}}x^i.
\end{equation}
So 
$$S_{n}(x,\bxi)/\sqrt{N} = (P_{n}(x),P'_{n}(x))$$ 
and
$$S_{n}(x,y,\bxi)/\sqrt{N} = (P_{n}(x),P'_{n}(x), P_{n}(y), P'_{n}(y)).$$
Also define the following even functions
\begin{equation}\label{eqn:F:delta} 
F_\delta(x) = \frac{1}{2 \delta} 1_{|x| <\delta}, x \in \R
\end{equation}
and
\begin{equation}\label{eqn:Phi:delta} 
\Phi_\delta(x) = |x_2| F_\delta(x_1), x=(x_1,x_2)\in \R^2
\end{equation}
and 
\begin{equation}\label{eqn:Psi:delta}  
\Psi_\delta(x) = \Phi_\delta(x_1,x_2) \Phi_\delta(x_3,x_4) = |x_2| F_\delta(x_1) |x_4| F_\delta(x_3), x=(x_1,x_2,x_3,x_4) \in \R^4.
\end{equation}
We have
$$\Phi_\delta\left (S_n(x,\bxi)/\sqrt{N}\right )  = |P'_n(x,\bxi)| \times \frac{1}{2\delta} 1_{|y| \le \delta}(P_n(x))=: \Phi_\delta(x,\bxi)$$
and
$$\Psi_\delta\left (S_n(x, y,\bxi)/\sqrt{N}\right ) = \Phi_\delta(x,\bxi) \Phi_\delta(y,\bxi).$$
In later sections, we will use the shorthand $\Phi_{\delta}(x,\bxi)$ for $\Phi_\delta(S_n(x,\bxi)/\sqrt{N})$ and $\Phi_{\delta}(x,y,\bxi)$ for $\Psi_\delta(S_n(x, y,\bxi)/\sqrt{N})$ respectively.
 
Finally, for short we introduce 
\begin{align}\label{eq:def:vn}
v_n(x,y,\bxi):&=\frac{1}{(2\delta)^2}\cov(|P'_n(x,\bxi)| 1_{ |P_n(x)|<\delta }, |P'_n(y,\bxi)| 1_{ |P_n(y,\bxi)|<\delta })\nonumber \\
&= \E \Psi_{\delta}\left (S_n(x,y,\bxi)/\sqrt{N}\right)  - \E \Phi_\delta(x,\bxi) \E \Phi_\delta(y,\bxi).
\end{align}
Using this notation, one has for instance:
\begin{align}
    \E N_{I_W}= \int_{I_W}\E\Phi_\delta(S_n(x,\bxi)/\sqrt{N})dx.
\end{align}

 Thus, to understand the expectation we need to study $ \E\Phi_\delta(S_n(x,\bxi)/\sqrt{N})$ for most $x\in I_{W}$. Several challenges arise with this approach. First, the function $\Phi_\delta$ lacks smoothness, and the parameter $\delta$ can be extremely small (for example, $\delta = N^{-C}$ for some sufficiently large constant $C$). Second, the random coefficients are non-gaussian. Third, the variance computation is more complicated as one has to deal with $\Psi_{\delta}$.
 
 To overcome these difficulties--particularly the second--we introduce the general framework of the \emph{Edgeworth expansion}, which allows us to transfer results for smooth statistics of \( S_n(x, \boldsymbol{\xi}) \) (or its higher-dimensional analogues) to those of the Gaussian case, while maintaining explicitly traceable error terms. The rough statistics \( \Phi_\delta \) can then be handled via suitable approximation. A central feature of our analysis is that the characteristic functions associated with \( S_n(x, \boldsymbol{\xi}) / \sqrt{N} \) decay rapidly and uniformly for all \( x \in I_W \) (see Section~\ref{section:LCD:W}). We emphasize that the method developed here--specifically designed to accommodate the factorial-type coefficients \( \sqrt{i!} \)--is also potentially applicable to the Elliptic ensemble, a direction we plan to pursue in future work. Finally, the reader is referred to Sections~\ref{sect:exp} and~\ref{sect:var} for detailed treatments of the expectation and variance, respectively.

\section{Edgeworth Expansions}\label{sect:EW}

\subsection{Edgeworth Expansion for linear forms} The Edgeworth expansion has been studied extensively in the literature. In this note we introduce a variant where one usually comes up when working with universality question in random matrices and random polynomials. The results are formulated under very mild assumptions on the coefficient distribution(s), which  hold in discrete settings  (such as the Rademacher distribution). Note that our result has some overlap with \cite{AP} and \cite{DNN}, but our conditions and applications are new.


As our setting is a little bit non-standard, let us explain the details below. Consider a sequence of random vectors $(X_{n,k})_{k=1}^n=C_n(k)Y_k$ (here $C_n(k)$ corresponds to the ensemble pre-factor and $Y_k$ are the random coefficients) in $\R^d$ with average covariance matrix 
\begin{equation}\label{eqn:V_n}
V_n=\frac{1}{N} \E \sum_{k=1}^n X_k X_k^\ast.
\end{equation}
Here $N$ is another parameter that are supposed to be sufficiently large. For instance in later application to random walks arising from Weyl polynomials, $N$ will have order $\sqrt{M}$ from \eqref{eqn:I_{W}:gen}.

When the $X_{i}$ are i.i.d. then $N=n$, but in general $N$ is a parameter so that $V_{n}$ is a matrix comparable to $I_{d}$. This parameter $N$ is especially useful when some of $N$ random variables $X_{i}$ dominate the rest.

Consider the random sum in $\R^d$
\begin{equation}\label{e.S_n}
S_n := X_{1}+\dots+X_{n}. 
\end{equation}

Note that $S_n/\sqrt{N}$ has covariance $V_n$. Let $\widetilde Q_n$ denote its distribution, and  let $\widetilde Q_n(x)$ denote the cumulative distribution function for this distribution. The main result of this section proves a local version of central limit theorem, showing that under some reasonable conditions  the law of $\widetilde Q_n$ is asymptotically $\widetilde{Q}_{n,\ell}$ with polynomial error bound $N^{-(\ell-1)/2}$. Here for each $\ell\ge 2$, the measure $\widetilde{Q}_{n,\ell}$ has the form of 
\begin{equation}\label{eqn:Q_nl}
\widetilde{Q}_{n,\ell}= \sum_{r=0}^{\ell-2} N^{-r/2} P_r(-\Phi_{0,V_n}, \{\overline{\chi}_\nu\}),
\end{equation}
where we will define the signed measure $P_r(-\Phi_{0,V_n}, \{\overline{\chi}_\nu\})$ below after fixing some notation. For convenience, the density of $\widetilde Q_{n,\ell}$ is denoted by $Q_{n,\ell}$. 

First, let $W$ be the standard gaussian vector in $\R^d$, then for any covariance matrix $V_n$,  $V_n^{1/2}W$ will be the gaussian random variable in $\R^d$ with mean zero and covariance $V_n$. Let $\phi_{0,V_n}$ denote the density of its distribution and let $\Phi_{0,V_n}$ denote the cumulative distribution function. If $V_n$ is the identity matrix then we simply write $\phi$ and $\Phi$, respectively.  

Secondly, recall that the cumulants of a random vector $X$ in $\R^d$ are the coefficients in the following multivariate power series expansion 
\begin{equation}\label{e.cumulant}
\log \E [ e^{\Bz \cdot X}] = \sum_{\nu\in \BN^d} \frac{\chi_{\nu} \Bz^\nu}{\nu!}, \ \ \Bz\in \C^d,
\end{equation}
where
\begin{equation}\label{eqn:power}
\Bz^\nu = z_1^{\nu_1}\dots z_d^{\nu_d}
\end{equation} 
for each $d$-tuple $\nu=(\nu_1,\dots, \nu_d)$.

With $S_n=X_1+\dots+X_n$ with independent $X_{1},\dots, X_{n}$, it follows that the cumulants of $S_n$ are the sum of the corresponding cumulants of $X_{1},\dots, X_{n}$.  We set 
$$\overline{\chi}_\nu :=\frac{ \chi_{\nu}(S_n)}{N}.$$ 
Observe that $\overline{\chi}_\nu$ is also the ``average'' cumulant of $X_{1}$, \dots, $X_{n}$, where $\chi_\nu(X_i)$ can be computed from 
\begin{equation}\label{e.cumulant:xi}
\log \E [ e^{\Bz \cdot X_i}] = \sum_{\nu\in \BN^d} \frac{\chi_{\nu} (X_i) \Bz^\nu}{\nu!}, \ \ \Bz\in \C^d,\end{equation}
and so
$$\overline{\chi}_\nu= \frac{\chi_{\nu}(S_n)}{N} =  \frac{\sum_i \chi_{\nu}(X_i)}{N}.$$

Now, note that  cumulants  of $V_n^{1/2}W$ matches with the cumulants of $S_n/\sqrt n$ for any $|\nu|\le 2$, at the same time the higher order cumulants of $V_n^{1/2}W$ vanish thanks to symmetries of centered gaussian. Therefore,
\begin{eqnarray*}
\log \E [ e^{\Bz \cdot  (S_n/\sqrt N)}]  
&=&   \log  \E [e^{\Bz\cdot (V_n^{1/2}W)}]  +   \sum_{\nu\in \BN^d: |\nu|\ge 3}  (n\overline{\chi}_{\nu})  \frac{\Bz^\nu}{\nu!}  N^{-|\nu|/2} \\ 
&=& \log  \E [e^{\Bz\cdot V_n^{1/2}W}] +    \sum_{\ell\ge 1} (\sum_{\nu\in \BN^d: |\nu|=\ell+2}  \overline{\chi}_{\nu} \frac{\Bz^\nu}{\nu!})  N^{-\ell/2},
\end{eqnarray*}
where $|\nu| = \sum_{i=1}^{d} \nu_{i}$.

Set 
\begin{equation}\label{eqn:chipoly}
\overline\chi_\ell(\Bz) := \ell!\sum_{\nu\in \BN^d: |\nu|=\ell}  \overline{\chi}_{\nu} \frac{\Bz^\nu}{\nu!}, \Bz\in \C^d.
\end{equation}
We obtain
\begin{eqnarray}\label{eqn:tildeP}
\E [ e^{\Bz \cdot (S_n/\sqrt N)}]/ \E [ e^{\Bz \cdot   V_n^{1/2}W}] 
&=&  \exp[\sum_{\ell\ge 1}  \frac{\overline{\chi}_{\ell+2}(\Bz)}{(\ell+2)!}   N^{-\ell/2}] \nonumber \\
&=&  \sum_{m\ge 0} \frac 1{m!} \Big(\sum_{\ell\ge 1}  \frac{\overline{\chi}_{\ell+2}(\Bz)}{(\ell+2)!}   N^{-\ell/2}\Big)^m \quad \nonumber \\
&= &\quad\sum_{\ell\ge 0}  \widetilde P_\ell N^{-\ell/2},
\end{eqnarray}
where $\widetilde P_\ell$ is obtained by grouping terms having $n^{-\ell/2}$. 

It is clear  that $\widetilde P_{\ell}$ depends only on $\Bz$ and the average cumulants $\overline{\chi}_{\nu}, \ |\nu|\le \ell+2$. We'll write $\widetilde P_\ell(\Bz, \{\overline{\chi}_\nu\})$ to stress this dependence.  

Replacing $\Bz$ by $i\Bz$, we obtain the following expansion for the characteristic function of $S_n/\sqrt N$
\begin{eqnarray*}
\E [ e^{i\Bz \cdot (S_n/\sqrt N)}] 
&=&   \E [ e^{i\Bz \cdot   V_n^{1/2}W}] \sum_{\ell\ge 0}  \widetilde P_\ell (i\Bz, \{\overline{\chi}_\nu\})N^{-\ell/2}.
\end{eqnarray*}

So, in principle, we have expressed the characteristic function of $S_n/\sqrt{N}$ as a product of that of the gaussian vector with covariance $V_n$ and a power series of $\Bz$. 

We next deduce the distributions by the inversion formula. Let $D=(D_1,\dots, D_d)$ be the partial derivative operator and let $\widetilde{P}_{\ell}(-D, \{\overline{\chi}_\nu\})$ be the differential operator obtained by formally replacing all occurrences of $i\Bz$ by $-D$ inside $\widetilde P_\ell (i\Bz, \{\overline{\chi}_\nu\})$.  The signed measures $P_{l}(-\Phi_{0,V_n}, \{\overline{\chi}_\nu\})$ in the definition \eqref{eqn:Q_nl} of $\widetilde Q_{n,\ell}$ now can be defined to have the following density with respect to the Lebesgue measure:
$$P_{l}(-\phi_{0,V_n},  \{\overline{\chi}_\nu\})(\Bx):=  \Big(\widetilde{P}_{l}(-D, \{\overline{\chi}_\nu\})\phi_{0,V_n}\Big)(\Bx).$$

For convenience of notation, for each $\ell>0$, let 
$$\rho_l := \frac{1}{N} \sum \E\|X_i\|_2^l.$$ 
and for any   measurable function $f$
$$M_\ell(f) := \sup_{\Bx\in \R^d} \frac{|f(\Bx)|}{1+\|\Bx\|_2^\ell}.$$
Let the characteristic function  $\phi(\Bx), \Bx\in \R^d$, of $S_n$ 
be
$$\phi(\Bx) = \E e^{i \Bx \cdot S_n} = \prod_{j=1}^n \E e^{i \Bx \cdot X_{j}} = \prod_{j=1}^n \phi_j(\Bx).$$
We now restate a result from \cite{DNN} under the assumption that $\phi(\Bx)$ decays fast.
\begin{theorem}\label{thm:EW:linear} Let $\ell$ be a fixed positive integer. Let $S_n$ be defined as in \eqref{e.S_n} where we assume that the distribution of $\xi$ satisfies $\E |\xi|^{\ell+d+1} <\infty$. Let $f$ be measurable such that $M_{\ell}(f)<\infty$. Suppose that
\begin{enumerate}
\item  all eigenvalues of $V_n$ are larger than a constant $\sigma>0$ independent of $n$;    
\vskip .05in
\item  for a constant $C_{\ast} \ge 1/2$ the characteristic function satisfies that for all index set $I\subset [n]$ of size at most $O_{\ell,d}(1)$ we have
\begin{equation}\label{cond:char:integral}
\int_{r<\|\mathbf{\eta}\|_2< N^{C_\ast-1/2} \log^2 N} |\prod_{i\not \in I} \phi_i(  \mathbf{\eta})| d\mathbf{\eta}  \ll_{\ell,d,C_\ast} \frac{1}{L_n N^{d/2}}.
\end{equation}
\end{enumerate}
for some parameter $L_n$ and for sufficiently small $r$ (depending on $\sigma$). Then the following estimate holds for  $\eps=N^{-C_\ast}$
\begin{eqnarray*}
&& |\int f(\Bx) d\widetilde Q_{n} - \int f(\Bx) d \widetilde{Q}_{n,\ell} |\\
 &\le& C M_{\ell}(f) (\frac{1}{L_n} + N^{-(\ell-1)/2} + e^{-cn} )+  \overline{\omega}_f(2\eps: \sum_{r=0}^{\ell+d-2} N^{-r/2} P_r(-\phi_{0,V_n}: \{\overline{\chi}_\nu\})
 \end{eqnarray*}
where  
$$\overline{\omega}_f(\eps:\phi) = \int (\sup_{\By\in B(\Bx,\eps)} f(\By) -\inf_{\By\in B(\Bx,\eps)} f(\By)) d\phi(\Bx),$$ 
where $B(\Bx, \eps)$ is the open ball of radius $\eps$ centered at $\Bx$, and the implied constant $C$ depends  on $\{\rho_k, k\le \ell\}$, $\sigma$, $C_*, \ell, d$, but not on $f$.
\end{theorem}

Note that the above result is similar to \cite{DNN}, but the difference here is that here we do not assume $|\prod_{i\not \in I} \phi_i(  \mathbf{\eta})| \ll_{\ell,d,C_\ast} \exp(-N^{c})$ for all $r \le \|\mathbf{\eta}\|_2 \le N^{C_\ast}$, but only a much weaker bound for the entire integral \eqref{cond:char:integral}\footnote{Although this difference is not used in the current paper, it is crucial in a subsequent work for expansion of random quadratic forms.}. As the proof of this result is almost identical to that of \cite[Theorem 4.1]{DNN}, we omit the details.


\subsection{Explicit formulas} In this subsection, we present some explicit examples of the Edgeworth expansions afforded by the previous theorem. 

\subsubsection{Univariate case} We first discuss our explicit formula for $d=1$, where we will write $\Bv_i$ as $v_i$. Assume that $S_n= \xi_1 v_{1}+\dots +\xi_n v_{n}$, where $\sum_i v_i^2=N$ and $\xi_i$ are i.i.d. copies of $\xi$ of mean zero and variance one. Then $S_n/\sqrt{N}$ is asymptotically standard gaussian. In one dimension we have
$$\overline{\chi}_\nu = \frac{1}{N} \frac{\chi_{\nu}(\xi)}{\nu!}  \sum_{i=1}^n v_i^\nu.$$
We have the following explicit form (see for instance \cite{Bobkov-survey}) for the density $\phi_\ell(x) = P_{\ell}(-\phi_{0,1},  \{\overline{\chi}_\nu\})(x)$ of $\widetilde{Q}_{n,\ell}$ from \eqref{eqn:Q_nl} for $\ell \ge 3$
\begin{equation}\label{eqn:rho_ell}
\phi_{\ell }(x) = \phi(x) \Big[\sum_{r=0}^{\ell-2} \sum \frac{1}{k_3! \dots k_{ \ell}!} (\frac{\overline{\chi}_3}{3!})^{k_3} \dots  (\frac{\overline{\chi}_{\ell}}{k_\ell!})^{k_{\ell}} H_r(x) N^{-r/2}\Big]
\end{equation}
where the second sum runs over non-negative integers $k_3,\dots, k_{r+2}$ such that  
$$3k_3 +\dots+ (r+2) k_{r+2}=r+2\mbox{ and } k_{3}+2 k_4 +\dots+ rk_{r+2} \le  r,$$ 
and where $H_k(x)$ are Hermite polynomials $$H_0(x)=1, H_1(x)=x, H_2(x)=x^2-1, H_3(x)= x^3-3x, \dots, H_{n+1}(x) = x H_n(x) - H_n'(x).$$
For instance 
$$\phi_3(x) = \phi(x) (1 +N^{-1/2} \frac{\overline{\chi}_3}{3!} H_3(x))$$
and
$$\phi_4(x) = \phi(x) (1 + N^{-1/2} \frac{\overline{\chi}_3}{3!} H_3(x) + N^{-2/2} (\frac{\overline{\chi}_4}{4!} H_4(x) + \frac{\overline{\chi}_3^2}{2! 3!^2} H_6(x)))$$
and
\begin{eqnarray*}
\phi_{5}(x)= \phi(x) [1 & + N^{-1/2} \frac{\overline{\chi}_3}{3!} H_3(x) + N^{-2/2} (\frac{\overline{\chi}_4}{4!} H_4(x) + \frac{\overline{\chi}_3^2}{2! 3!^2} H_6(x)) \\
 & + N^{-3/2} (\frac{\overline{\chi}_5}{5!} H_5(x) + \frac{\overline{\chi}_3\overline{\chi}_4}{3!4!}H_7(x)+\frac{\overline{\chi}_3^2}{3! (3!)^3} H_9(x))].
\end{eqnarray*}


\begin{remark} If $S_n= v_1 \xi_1+\dots +v_n \xi_n$, where $\sum_i v_i^2=nt^2$ for some given $t>0$, then 
 $$\P(S_n/\sqrt{n} \le x) = \P(S_n/\sqrt{n}t \le x/t).$$ 
 So $\wt Q_{S_n/\sqrt{n}}(x) = \wt Q_{S_n/\sqrt{n}t}(x/t)$. The density with respect to $S_n/\sqrt{n}$ then becomes 
 $$\phi_{\ell,t}(x)= \frac{1}{t}\phi_\ell(\frac{x}{t}).$$
 \end{remark}
 
\subsubsection{Multivariate case}\label{subsection:multivariate} Note that if the random vector $X_{i}\in \R^{d}$ has special form $X=(\xi v_{1},\dots, \xi v_{d}) \in \R^{d}$, where $v_{i}$ are deterministic, then
$$\log \E [e^{\Bz \cdot \Bv}] = \log \E e^{\xi (v_{1}z_{1}+\dots+ v_{d}z_{d})} = \sum_{\ell=0}^{\infty}\frac{\chi_{\ell}(\xi)}{\ell !} (v_{1} z_{1}+\dots+ v_{d}z_{d})^{\ell} = \sum_{\ell=0}^{\infty} \sum_{\nu \in \{1,\dots,d\}^{\ell}}\frac{\chi_{\ell}(\xi)}{\ell !} \Bv^{\nu} \Bz^{\nu}.$$
For each $\nu \in \{1,\dots,d\}^{\ell}$ define:
$$\Bv^{\nu} = v_{1}^{\nu_{1}} \dots v_{d}^{\nu_{d}}$$
and $|\nu|=\sum_{i=1}^d\nu_i$. And thus the average cumulant is now defined as: 
\begin{equation}\label{eqn:i.i.d.:barchi}
\overline{\chi_{\nu}}= \frac{1}{N}\frac{\chi_{|\nu|}(\xi)}{|\nu|!}\sum_{i=1}^n\Bv_i^{\nu}.
\end{equation}
We use boldface to emphasize vector variables. Let us now outline the analogous $d$-dimensional with identity covariance matrix 
\begin{equation}\label{eqn:Identity}
V_n=\I_d.
\end{equation}

Expanding the moment-generating function as before one has:
\begin{equation*}\sum_{m=0}^{\infty}\frac{1}{m!}\left(\sum_{s=3}^{\infty}\sum_{|\nu|=s}\frac{\overline{\chi_{\nu}}}{\nu!}\frac{H_{\nu}(\Bx)}{N^{\frac{s}{2}-1}}\right)^m\end{equation*}
and finally grouping like-order terms in powers of $n$ we obtain the following expansion.
\begin{equation}\phi(\Bx) \left[\sum_{r=0}^{\infty}\left(\sum_{|\mu|=r}\frac{\prod_{i=3}^l\left(\sum_{|\nu|=k_i+2}\frac{\overline{\chi_{\nu}}}{\nu!}H_{\nu}(\Bx)\right)}{(l-2)!}\right)n^{-r/2}\right].\end{equation}

Here $\mu=(k_3,...,k_l)$ is a partition of $r$, $\nu$ is a partition of $k_i+2$ for each $3\leq i\leq l$, $H_{\nu}$ is the partial Hermite polynomial corresponding to $\nu$, and we similarly sum over indices. \\


More specifically, $\widetilde P_0 = 1$ and 
\begin{eqnarray}\label{e.P1P2}
\widetilde{P}_1(\Bx, \{\overline{\chi}_\nu\}) = \sum_{|\nu|=3} \frac{\overline{\chi}_\nu(\Bx)}{\nu!} , \qquad
\widetilde{P}_2(\Bx, \{\overline{\chi}_\nu\}) = \sum_{|\nu|=4}\frac{\overline{\chi}_\nu(\Bx)}{\nu!} +  \frac{(\sum_{|\nu|=3}\frac{\overline{\chi}_\nu(\Bx)}{\nu!})^2}{2},
\end{eqnarray}
where we refer to \eqref{eqn:chipoly} for the $\overline{\chi}_{l}(\Bx)$ polynomials.

In what follows we rewrite the above following \cite{BCP, DNN}. For convenience of notation let $e_j = (\dots, 0, 1, 0,\dots) \in \R^d$ where $1$ is in the $j$th coordinate. Using \eqref{e.P1P2} we obtain
$$P_1(-\phi_{0, \I_{d}}, \{\overline{\chi}_\nu\})  \quad = \quad \sum_{|\nu|=3} \frac{\overline{\chi}_\nu}{\nu!} (-D)^\nu\phi(x) \quad =\quad $$
\begin{eqnarray*}
&=& \Big[ \frac 1 6 \sum_{j=1}^d \overline{\chi}_{3e_j} (x_j^3-3x_j) + \frac 1 2 \sum_{i\ne j} \overline{\chi}_{2e_i+e_j} (x_i^2 x_j - x_j) +  \sum_{i< j < k} \overline{\chi}_{e_i+ e_j+ e_k} x_i x_j x_k\Big]\phi_{0,\I_{d}}(\Bx)\\
&=& \Big[\frac 1 6 \sum_{j=1}^d \overline{\chi}_{3e_j} H_3(x_j)  + \frac 1 2 \sum_{i\ne j} \overline{\chi}_{2e_i+e_j} H_2(x_i)H_1(x_j)  +  \sum_{i, j, k} \overline{\chi}_{e_i+ e_j+ e_k} H_1(x_i)H_1(x_j) H_1 (x_k)\Big]\phi_{0,\I_{d}}(\Bx),
\end{eqnarray*}
where we recall that $H_k(x)=(-1)^k e^{x^2/2} \frac{\partial^k}{\partial x^k} e^{-x^2/2} (k=0,1,2,\dots)$ are the (one dimensional) Hermite polynomials.

We can express more explicitly as follows. For any multi-index $\al=(\al_1,\dots,\al_{\ell}) \in \{1,\dots,d\}^{\ell}$, we have $\dim(\alpha)=\ell$ and $|\alpha|=\sum_{i=1}^l\alpha_i$. Now let $n_j(\al) = |\{i: \al_i =j\}|$ for each $j=1,\dots, d$. Thus $\sum_j n_j(\al) =\ell$. We then define
\begin{equation}\label{H}
H_\al(x_1,\dots, x_d):= \prod H_{n_1}(x_1) \dots H_{n_d}(x_d).
\end{equation}

Note that if $\alpha'$ is a permutation of $\alpha$ then $H_{\alpha'}=H_{\alpha}$, and that 
$$\E \partial^\al f(\BW) = \E f(\BW) H_\al(\BW).$$

Moreover, for cross terms involving products of different partitions we define
$$H_{\{\alpha,\beta\}}(x) = H_\alpha(x) \cdot H_\beta(x).$$

For given multi-index  $\al=(\al_1,\dots,\al_{\ell}) \in \{1,\dots,d\}^{\ell}$, for a random vector $Z=(Z_1,\dots,Z_d)$, let  
\begin{equation}\label{eqn:power'}
Z^\al := \prod_{j=1}^d Z_j^{n_j(\al)}.
\end{equation}
Note that this is slightly different from \eqref{eqn:power}.

With $X=(X_{n,1},\dots,X_{n,n} )$, where each $X_{n,k}$ is a random vector in $\R^{d}$, with $G$ being the corresponding gaussian vector in $\R^{d}$ 
, define
\begin{eqnarray}
\label{eqn:Delta}
\Delta_\al(X_{n,k}) &=& \E X_{n,k}^\al - \E G^\al,\\ 
\label{eq:def:c:n:alpha}
c_n(\al, X) &:= & \frac{1}{N} \sum_{k=1}^n \Delta_\al (X_{n,k})\\
\label{eqn:Gamma1}
\Gamma_{n,1}(X,x) &:=&  \frac{1}{6} \sum_{|\al|=3} c_n(\al,X) H_\al(x).
\end{eqnarray}
Furthermore using \eqref{e.cumulant}  and explicit computations it follows that $\chi_{\nu}(X) = \E [X^\nu]$ for all $\dim(\nu)=1,2$  if $X$ is a random vector in $\R^d$ with mean $\E X=0$ and $\Cov(X)=\BI_{d}$. Using these observations, we obtain
$$P_1(-\phi_{0,\I_{d}},  \{\overline{\chi}_\nu\}) = \Gamma_{n,1} (X,x) \phi_{0,\BI_{d}}(\Bx).$$

We also define
\begin{equation}\label{eqn:Gamma:2}
\Gamma_{n,2}(X,x) = \Gamma_{n,2}' + \Gamma_{n,2}''
\end{equation}
where 
$$\Gamma_{n,2}'(X,x) = \frac{1}{24} \sum_{|\beta|=4} c_n(\beta,X) H_\beta(x)$$
and 
$$\Gamma_{n,2}''(X,x) = \frac{1}{72} \sum_{|\rho|=3}\sum_{|\beta|=3} c_n(\beta,X) c_n(\rho,X)H_{(\beta,\rho)}(x),$$
here $H_{(\beta,\rho)}(x):=H_\beta(x)H_\rho(x)$. Here $|\alpha|=\sum_{i=1}^l\alpha_i$ denotes the weight of the multi-index.

Via explicit computations, it can also be checked that
$$P_2(-\phi_{0,\BI_d}, \{\overline{\chi}_\nu\}) = \Gamma_{n,2}(X,x) \phi_{0,\BI_{d}}(x).$$
Finally, recall the definition of $\widetilde Q_{n,2}$ from \eqref{eqn:Q_nl}, which has density
$$\widetilde Q_{n,2}(X,x) = 1+N^{-1/2} P_1(-\Phi_{0,\BI_{d}}, \{\overline{\chi}_\nu\})+N^{-1} P_2(-\Phi_{0,\BI_{d}}, \{\overline{\chi}_\nu\}) .$$
Using the above notation, we record below the density of $Q_{n,2}(X,W)$ (i.e. the density of $\widetilde Q_{n,2}$). 
\begin{fact}\label{fact:Q_{n,2}}  
$$Q_{n,2}(X,x)  = (1 + \frac{1}{\sqrt{N}} \Gamma_{n,1}(X, x) + \frac{1}{N} \Gamma_{n,2}(X,x))\phi_{0,\BI_{d}}(x).$$
\end{fact}

\subsection{Characteristic function condition via high dimensional Diophantine properties}
 In this subsection, we present a systematic approach to guarantee Condition \ref{cond:char:integral} for the Edgeworth expansion. In the special case that 
$X_i = \eps_i \Bv_i$  where $\eps_i$ are i.i.d. copies of general $\xi$ and $\Bv_i$ is a deterministic ensemble pre-factor, we can exploit the lack of arithmetic structure of the $\Bv_i$ to provide a simplified test for the decay properties of the characteristic function.

Given a real number $w$ and a random variable $\xi$, we define the $\xi$-norm of $w$ by
$$\|w\|_\xi := (\E\|w(\xi_1-\xi_2)\|_{\R/\Z}^2)^{1/2},$$
where $\xi_1,\xi_2$ are two i.i.d. copies of $\xi$. For instance if $\xi$ is Bernoulli with $\P(\xi=\pm 1) =1/2$ (which is our main focus), then $\|w\|_\xi^2 = \| 2w\|_{\R/\Z}^2/2$.

The following works for general $\R^d$: consider the random sum (which will be eventually applied to $S_{n}$ in our case) $\sum_i \xi_i \Bv_i$, where $\Bv_i$ are deterministic vectors in $\R^d$. Then the corresponding characteristic function can be bounded by (see \cite[Section 5]{TVcir})
\begin{align}\label{eqn:cahr:bound}
|\prod \phi_i(\Bx)|  = |\prod\E \exp(\mathbf{i}( \xi_{i} \langle \Bv_i, \Bx\rangle))| &\le \prod_i [|\E \exp(\mathbf{i}( \xi_{i} \langle \Bv_i, \Bx\rangle))|^2/2+1/2]  \nonumber \\
& \le \exp(-(\sum_i \|\langle \Bv_i, \Bx/2\pi \rangle\|_{\xi}^2).
\end{align}

Hence if we have a good lower bound on the exponent $\sum_i \|\langle  \Bv_i, \Bx/2\pi \rangle\|_\xi^2$ then we would have a good control on $|\prod \phi_i(\Bx)|$. Furthermore, by definition
\begin{align}\label{eqn:decoupling}
\sum_i \|\langle  \Bv_i, x/2\pi \rangle\|_\xi^2  &= \sum_i \E\| \langle  \Bv_i, \Bx/2\pi \rangle (\xi_1-\xi_2)\|_{\R/\Z}^2 \nonumber \\
& = \E (\sum_i \| \langle  \Bv_i, \Bx/2\pi \rangle (\xi_1-\xi_2)\|_{\R/\Z}^2 \nonumber \\
&= \E_y (\sum_i \| y \langle  \Bv_i, \Bx/2\pi \rangle\|_{\R/\Z}^2,
\end{align}
where $y=\xi_1-\xi_2$.
As $\xi$ has mean zero, variance one and bounded $(2+\eps_0)$-moment, there exist positive constants $c_1\le c_2,c_3$ such that $\P(c_1 \le |y| \le c_2) \ge c_3$, and so 
\begin{align}\label{eqn:y}
\E_y \sum_i \| y \langle  \Bv_i, \Bx/2\pi \rangle\|_{\R/\Z}^2  &\ge c_3 \inf_{c_1\le |y| \le c_2} \sum_i \| y \langle  \Bv_i, \Bx/2\pi  \rangle\|_{\R/\Z}^2.
\end{align}
Hence for Condition \eqref{cond:char:integral} it suffices to show that for any $\BD\in \R^d$ (which plays the role of $(y/2\pi)\Bx$ such that $c_1r \le \|\BD\|_2\le c_2 n^{C_\ast -1/2} \log^2 n$ we have 
\begin{equation}\label{eqn:dist:Z}
\sum_{i\not \in I} \|\langle \Bv_{i},   \BD \rangle\|_{\R/\Z}^2 \ge \log (M_{n}),
\end{equation}
where $ (n^{C_\ast -1/2} \log^2 n)^{d} M_{n} \le 1/L_{n} N^{d/2}$.
This motivates us to define the following.
\begin{definition}\label{lcd} Let $M_n,N$ be given parameters (that might depend on $n$) and $r>0$ be given. For a given $(\Bv_1,\dots, \Bv_n)$ such that $\sum_{i} \|\Bv_{i}\|_{2}^{2}\asymp N$ we define $D=D_{M_n}(\Bv_1,\dots, \Bv_n)$ to be the smallest $\|\BD\|_2$ so that $\|\BD\|_2 \ge r$ and any $I \subset [n]$ of size $O(1)$ and
$$\sum_{i\not \in I} \|\langle \Bv_i,   \BD \rangle\|_{\R/\Z}^2 \leq \log M_n.$$
\end{definition}

This number $D$ is somewhat similar to the so-called {\it least-common-denominator} notion introduced by Rudelson and Vershynin from \cite{RV-rec} (see also \cite{FS}). In most applications it suffices to choose 
$$M_n=N^C, \mbox{ for some large constant $C$}.$$ 


\begin{example}\cite{FS}\label{example:SK} For $d=1$ and $v_{2i} = 1, v_{2i+1} = \sqrt{2}, 1\le i \le n/2$. For any $M_{n}\ge 1$ such that $\tau_{n} = \log M_{n} \le n^{1/2-o(1)}$ we have 
$$|D_{M_{n}}(v_{1},\dots, v_{n})| = \Omega(\sqrt{n /\tau_{n}}).$$ 
\end{example}

\begin{proof} Fix $D\gg1$ and let
$I\subset[n]$ be any index set with $|I|=O(1)$.  Set
\[
\delta:=\|D\|_{\R/\Z},\qquad \varepsilon:=\|D\sqrt2\|_{\R/\Z}.
\]
Since exactly $n/2$ of the $v_i$ equal $1$ and $n/2$ equal $\sqrt2$, we have
\begin{equation}\label{eq:half-sum}
\sum_{i\notin I}\|Dv_i\|_{\R/\Z}^2
= \frac{n}{2}\,\big(\delta^2+\varepsilon^2\big) + O(1).
\end{equation}
Let $q\in\mathbb Z$ be the nearest integer to $D$ and $p\in\mathbb Z$ the nearest integer to $D\sqrt2$, so
$|D-q|=\delta$ and $|D\sqrt2-p|=\varepsilon$.  By Liouville’s theorem (for quadratic irrationals) there is an
absolute constant $c_1>0$ such that for all integers $p,q\neq0$,
\[
\Big|\sqrt2-\frac{p}{q}\Big| \ge \frac{c_1}{q^2}.
\]
Multiplying by $|q|$ and using the triangle inequality gives
\[
\frac{c_1}{|q|} \le |q\sqrt2-p|
\le |D\sqrt2-p|+\sqrt2\,|D-q|
= \varepsilon+\sqrt2\,\delta .
\]
By Cauchy–Schwarz,
\[
\delta^2+\varepsilon^2 \ \ge\ \frac{(\varepsilon+\sqrt2\,\delta)^2}{3}
\ \ge\ \frac{c_1^2}{3\,q^2}.
\]
Because $|q|\ge |D|-1/2$ and $D\gg1$, we have $|q|\asymp |D|$, hence for some absolute $c_2>0$,
\begin{equation}\label{eq:delta-eps}
\delta^2+\varepsilon^2 \ \ge\ \frac{c_2}{D^2}.
\end{equation}
Insert \eqref{eq:delta-eps} into \eqref{eq:half-sum} to get
\[
\sum_{i\notin I}\|Dv_i\|_{\R/\Z}^2
\ \ge\ \frac{n}{2}\cdot \frac{c_2}{D^2}-O(1)
\ \ge\ \frac{c_3\,n}{D^2},
\]
for all sufficiently large $n$ (absorbing the $O(1)$ term into $c_3$).  By Definition \ref{lcd},
$D=D_{M_n}(v_1,\dots,v_n)$ is the smallest $|D|\gg1$ such that the left-hand side is at most
$\tau_n=\log M_n$ for every $I$ with $|I|=O(1)$.  Therefore
\[
\frac{c_3\,n}{D^2}\ \le\ \tau_n
\qquad\Longrightarrow\qquad
|D_{M_{n}}(v_1,\dots,v_n)| \ \ge\ c\,\sqrt{\frac{n}{\tau_n}} .
\]
\end{proof}
We  conclude the section with a useful remark.
\begin{corollary}\label{thm:maincor} Assume that $S_{n} = \sum_{i} \xi_{i} \Bv_{i}$. 
The conclusion of Theorem \ref{thm:EW:linear} holds for $S_{n} = \sum_{i} \xi_{i} \Bv_{i}$ with $L_{n}=N^{L_{0}}$, where Condition \ref{cond:char:integral} is replaced by the condition 
 \begin{equation}\label{cond:LCD}
D_{N^{dC_{\ast}+L_{0}+1}}(\Bv_1,\dots, \Bv_n) \ge N^{C_\ast -1/2} \log^2 N.
\end{equation}
\end{corollary}

\begin{proof} 
Assume that $D=D_{N^{dC_{\ast}+L_{0}+1}}(\Bv_1,\dots, \Bv_n) \ge N^{C_\ast-1/2} \log^2 N$. Then by definition,
\begin{align*}
\int_r^{N^{C_\ast-1/2} \log^2 N} |\prod_{i\not \in I} \phi_i(  \mathbf{\eta})| d\mathbf{\eta} &\le \int_r^{N^{C_\ast-1/2} \log^2 N} e^{-\sum_{i\not \in I} \|\langle X_i,   \mathbf{\eta} \rangle\|_{\R/\Z}^2} d\mathbf{\eta} \\
&\le (N^{C_\ast-1/2} \log^2 N)^{d} N^{-dC_{\ast}-L_{0}-1} \le 1/L_{n} N^{d/2}.
\end{align*}
\end{proof}

\section{Diophantine properties and random walks from Weyl polynomials: checking Condition \eqref{cond:LCD}  (and Condition \eqref{cond:char:integral})}\label{section:LCD:W}
We will be focusing on roots of $P_{n}(x)=0$ where $x$ belongs to the interval $I_{W}$ from \eqref{eqn:I_{W}:gen}. 

In what follows, we show that 
\(D_{(\cdot)}(\mathbf{v}_{1}, \dots, \mathbf{v}_{n})\) 
are polynomially large, where 
\(\mathbf{v}_{1}, \dots, \mathbf{v}_{n}\) 
are vectors arising from Weyl polynomials. 
Owing to the nature of the present setting, 
we do not employ directly the differencing method of \cite{DNN} 
and \cite{CNg}—our vectors possess less algebraic structure—
instead, we exploit the  properties of the variance coefficients in the Weyl model.
Let 
$$N=M,$$
where $M$ is the order of the end points of $I_{W}$.

We recall that
\begin{equation}\label{eqn:W:b_i}
b_{i} = b_{i}(x)= \sqrt{N} e^{-x^2/2} x^{i}/\sqrt{i!}, \ 1\le i\le n.
\end{equation}
For short, its derivative is denoted by
\begin{equation}\label{eqn:W:c_i}c_i = c_{i}(x) =(b_i(x))'= \sqrt{N} e^{-x^2/2} \left(\frac{i-x^2}{x}\right) \frac{x^{i}}{\sqrt{i!}} =  \sqrt{N} e^{-x^2/2} \frac{i}{x} \frac{x^{i}}{\sqrt{i!}} - \sqrt{N} e^{-x^2/2} x \frac{x^{i}}{\sqrt{i!}}.\end{equation}
Assume that $\xi_{i}, 1\le i\le n$ are i.i.d. copies of a subgaussian random variable $\xi$ of mean zero, variance one, we naturally have the random walk in $\R$, 
\begin{equation}\label{eqn:S}
S_{n}^{0}(x,\bxi) = \sqrt{N} P_{n}(x) =\sum_{i=1}^n \xi_{i} b_{i}(x).
\end{equation} 
For $x \in I_{W}$, we defined the $\R^{2}$-vectors
\begin{equation}\label{u:1}
\Bu_i(x)= \left(b_{i}(x),c_{i}(x)\right ).
\end{equation}
 
We consider the random walk in $\R^2$
\begin{equation}\label{eqn:St}
S_{n}(x, \bxi):=\sum_{i=1}^n \xi_{i} \Bu_i,
\end{equation}
where $\bxi=(\xi_{1},\dots,\xi_{n})$, and $\xi_{i}$ are i.i.d. copies of a subgaussian random variable $\xi$ of mean zero and variance one.
For $\R^4$, we also define
\begin{equation}\label{v:1}
\Bv_i(x,y):= \left(b_{i}(x), c_{i}(x), b_{i}(y), c_{i}(y)\right ),
\end{equation}
and the random walk 
\begin{equation}\label{eqn:Sst}
S_{n}(x,y,\bxi):= \sum_{i=1}^n \xi_{i} \Bv_i.
\end{equation}

We first start with the following estimate from \cite[Claim 6.1]{ANgW}.

\begin{claim}\label{claim:Weyl:ix} Assume that $x>0$ is sufficiently large. There exist absolute constants $c_1, c_2$ such that the followings hold.

\begin{itemize}
\item Let $i$ be a positive integer with $i=x^2+t$, where $L= |t|/x \leq x^{\frac{1}{3}}$. Then
\begin{align*} c_1^{-1} x^{-1/2} \exp(-c_1 L^2)  \le e^{-x^2/2} x^{i}/\sqrt{i!} \leq c_2^{-1}x^{-1/2} \exp(-c_2 L^2).\end{align*}
\item If $i\ge x^2 + Lx$ then we have the one sided bound
\begin{align*} e^{-x^2/2} x^{i}/\sqrt{i!} \leq c_2^{-1}x^{-1/2} \exp(-c_2 L^2).\end{align*}.
\item Furthermore, if $i \le x^2 - Lx $ then 
\begin{align*}e^{-x^2/2} x^i/\sqrt{i!} \leq  c_2^{-1}x^{-1/2} \exp(- c_2 \min\{L^2, x^{2/3}\}).\end{align*}
\end{itemize}
\end{claim}
As a consequence, if $x \asymp N$ then
\begin{itemize}
\item with $i=x^2+t$, where $L= |t|/x \leq x^{\frac{1}{3}}$ we have
\begin{align*} \exp(-c_1 L^2)  \ll \sqrt{N} e^{-x^2/2}  x^{i}/\sqrt{i!} \ll \exp(-c_2 L^2)\end{align*}
\item if $i \ge x^2 + Lx $ then 
\begin{align*} \sqrt{N} e^{-x^2/2} x^i/\sqrt{i!} \ll  \exp(- c_2 L^{2}).
\end{align*}
\item and if $i \le x^2 - Lx $ then 
\begin{align*} \sqrt{N} e^{-x^2/2} x^i/\sqrt{i!} \ll  \exp(- c_2 \min\{L^2, x^{2/3}\}).
\end{align*}

\end{itemize}

 
\begin{proof}(of Claim \ref{claim:Weyl:ix}) We focus on the first claim. By Stirling's approximation, $i! \approx \sqrt{2\pi i}(i/e)^i$, and so 
\begin{align*}\frac{x^{2i}}{i!} \approx \frac{1}{\sqrt{2\pi i}} \left(\frac{ex^2}{i}\right)^i.\end{align*}
Substituting $i=x^2+t$ then 
\begin{align}
(ex^2/i)^i = [ex^2/(x^2+t)]^{x^2+t} &= e^{x^2} e^t (1-\frac{t}{x^2+t})^{x^2+t}
\nonumber\\
& = e^{x^2} e^t  \exp\Big(-[\frac{t}{x^2+t} +(\frac{t}{x^2+t})^2/2 + (\frac{t}{x^2+t})^3/3+\dots] \times ({x^2+t})\Big) \nonumber\\
\label{taylor-exp}&= e^{x^2} \exp\Big(-\frac{1}{2}\frac{t^2}{x^2+t} - \frac{1}{3}\frac{t^3}{(x^2+t)^2}-\dots\Big).
\end{align}
With our choice of $t=Lx$ the $k$-th term in the exponent is of the form
\begin{align*}
 \frac{(Lx)^{k+1}}{(x^2+Lx)^k} = \frac{L^{k+1}x}{(x+L)^k}.
\end{align*}
It can readily be observed that
\begin{align*}\sum_{k=2}^\infty \frac{L^{k+1}x}{(x+L)^k(k+1)} = O(1) \text{ as $L \le x^{1/3}$ and $x \to \infty$}. \end{align*}
Hence
$$(ex^2/i)^i  \asymp e^{x^2}  \exp(-\frac{1}{2}\frac{t^2}{x^2+t} )  \asymp e^{x^2}  \exp(-\Theta(L^2)).$$
The second claim then follows automatically. For the third claim with $i \le x^2 - Lx$, it suffices to show for $L \ge x^{1/3}$. Notice that $x^i/\sqrt{i!}$ is increasing as $i$ increases to $x^2$ (because the ratio is $(x^{i+1}/\sqrt{(i+1)!})/(x^i/\sqrt{i!}) = x/\sqrt{i+1}\ge 1$ if $i \le x^2-1$), so for  $i \le x^2 - Lx  \le i_0 = \lfloor x^2 - x^{1/3}x \rfloor,$  
$$e^{-x^2/2} x^i/\sqrt{i!}  \le e^{-x^2/2} x^{i_0}/\sqrt{i_0!} \le c_2^{-1}x^{-1/2} \exp(- c_2 x^{2/3}).$$
\end{proof}


We next study the covariance matrices of our random walks. First, the covariance of the matrix for $S_{n}(x,\bxi) = \sum_{i} \xi_{i} \Bu_{i}$ is
$$V_{n}(x) = \E (S_n(x,\bxi)/\sqrt{N}  (S_{n}^{T}(x,\bxi)/\sqrt{N})) = 
\left( \begin{matrix} 
\sum_{i} b_{i}^{2}(x)/N & \sum_{i} b_{i}(x) c_{i}(x) /N\\
\sum_{i} b_{i}(x) c_{i}(x) /N& \sum_{i} (c_{i}(x))^{2}/N \\
\end{matrix})\right).$$
We have
\begin{claim}\label{claim:cov:2} Let $0<\epsg<1/2$ and assume that $x$ has order $N$, and $|x| \le n^{1/2} - N^{\epsg}$. Then we have 
$$V_{n}(x) = I_{2}+\exp(-N^{O_{\epsg}(1)}).$$
\end{claim}
Next, the covariance of $\sum_{i} \xi_{i} \Bv_{i}$ is  
\begin{align*}
V_{n}(x,y) &= \E (S_n(x,y,\bxi)/\sqrt{N} (S_{n}^{T}(x,y,\bxi)/\sqrt{N})) = \\
&=\left( \begin{matrix} 
\sum_{i} b_{i}^{2}(x)/N & \sum_{i} b_{i}(x) c_{i}(x) /N&   \sum_{i} b_{i}(x) b_{i}(y) /N& \sum_{i} b_{i}(x) c_{i}(y)/N \\
\sum_{i} c_{i}(x) b_{i}(x)/N & \sum_{i} c_{i}(x)^{2}/N &   \sum_{i} c_{i}(x) b_{i}(y) /N& \sum_{i} c_{i}(x) c_{i}(y) /N\\
 \sum_{i}  b_{i}(y) b_{i}(x)/N & \sum_{i} b_{i}(y) c_{i}(x)/N  &   \sum_{i} (b_{i}(y))^{2} /N& \sum_{i} b_{i}(y) c_{i}(y)/N \\
 \sum_{i} c_{i}(y)  b_{i}(x) /N& \sum_{i} c_{i}(y)  c_{i}(x)/N&   \sum_{i} c_{i}(y) b_{i}(y) /N& \sum_{i} (c_{i}(y))^{2}/N
\end{matrix})\right).
\end{align*}

\begin{claim}\label{claim:cov:4} Assume that positive $x,y$ have order $N$, where $N$ is sufficiently large, $|x-y|\ge N^{\epsg}$, and $x,y \le n^{1/2} - N^{\epsg}$. Then 
$$V_{n}(x,y) = I_{4} +\exp(-N^{O_{\epsg}(1)}).$$
\end{claim}
Clearly this result implies Claim \ref{claim:cov:2}.
\begin{proof}(of Claim \ref{claim:cov:4})  The diagonal terms can be checked directly. For the off-diagonal terms 
$$\sum_{i=1}^{n} b_{i}(x) b_{i}(y) = (\sqrt{N})^{2} e^{{-x^{2}/2 - y^{2}/2}} \sum_{i=1}^{n} (xy)^{i} /i! \le N e^{{-(x-y)^{2}/2}}  \le \exp(-N^{\epsg}).$$ 
Also 
$$\sum_{i=1}^{n} b_{i}(x) c_{i}(x) =  (\sqrt{N})^{2} e^{-x^{2}} \sum_{i=1}^{n} (\frac{i-x^{2}}{x}) \frac{x^{2i}}{i!} = N e^{-x^{2}} (\sum_{i=1}^{n} x \frac{x^{{2i-2}}}{(i-1)!} - x \frac{x^{2i}}{i!}) = N e^{-x^{2}} (x - \frac{x^{2n+1}}{n!}) \to 0,$$

By the assumption $|x - \sqrt{n}| \ge N^{\epsg}$,  it follows that $|n -x^{2}|\ge N^{\epsg} x$, and hence by Claim \ref{claim:Weyl:ix}
$$x e^{-x^{2}}\frac{x^{2n}}{n!}  \le \exp(-c_{2} N^{2\epsg}).$$
 
Similarly, 
\begin{align*}
|\sum_{i=1}^{n} c_{i}(x) b_{i}(y)| &\le N e^{-x^{2}/2-y^{2}/2} \sum_{i=1}^{n}| (\frac{i-x^{2}}{x}) \frac{(xy)^{i}}{i!}| \\
&\le  N e^{-x^{2}/2-y^{2}/2}  [y \sum_{i=1}^{n} \frac{(xy)^{i-1}}{(i-1)!} + x \frac{(xy)^{i}}{i!}]\\
&\le  N  (x+y)e^{-x^{2}/2-y^{2}/2 +xy} \le  \exp(-c_{2} N^{2\epsg})
\end{align*}
and
\begin{align*}
|\sum_{i=1}^{n} c_{i}(x) c_{i}(y)| &\le N e^{-x^{2}/2-y^{2}/2} \sum_{i=1}^{n}| (\frac{i-x^{2}}{x})  (\frac{i-y^{2}}{y}) \frac{(xy)^{i}}{i!}| \\
&\le 2 N (xy+x^{2}+y^{2}+x^{2}y^{2}) e^{-x^{2}/2-y^{2}/2 +xy} \le  \exp(-c_{2} N^{2\epsg}).
\end{align*}

\end{proof}

We next move to the main results of the section. Throughout this section, $r$ is a parameter depending on the least singular value $\sigma$ of $V_{n}$ (which has been be shown to be of order 1).

\subsection{Dimension One}\label{subsection:W:1} 

\begin{theorem}\label{thm:Weyl:LCD:1} For any given $A,C> 0$, for any $x\in I_{W}$, there does not exists $r\le D \le N^{A}$ such that 
$$\sum_{i} \|D b_{i}\|_{\R/\Z}^{2} \le C\log N.$$
\end{theorem}
Our proof method is quite non-standard. Broadly speaking, it relies on the fact that for any sequence $x_{i}$, if the discrete differential operator of degree T annihilates it, then $x_{i}$ is a polynomial of degree at most $T-1$. This method was also applied in \cite{DNN}, but for somewhat  simpler sequences of $x_{i}$. As the statement is stronger for larger $A$, we'll assume $A$ to be a sufficiently large constant. Our method has two steps.

\begin{itemize}
\item (Step 1.) In the first step, as $\sum_{i} \|D b_{i}\|_{\R/\Z}^{2} \le C\log N$ is small, there exists a rather long interval $J \subset [n]$ where for each $i\in J$, $\|Db_{i}\|_{\R/\Z}$ is small. We then show that their nearest integers, $m_{i} = \lfloor Db_{i} \rceil$ form a polynomial sequence in $i$ of small degree.
\vskip .1in
\item (Step 2.) Using the information in Step 1, and by passing to a long arithmetic progression in $J$ where $\|Db_{i}\|_{\R/\Z}$ is of order, say, $N^{-1/4+o(1)}$, we reach a contradiction.
\end{itemize}
 
\begin{proof}(of Theorem \ref{thm:Weyl:LCD:1}) Let $x\in I_W$, so $x$ has order $M$. We will choose $L$ to be a sufficiently large constant given $A$. Recall that for any $i= \lfloor x^2 + Lx \rfloor$, where $ |L| \ll x^{1/3}$ we have 
$$|\frac{x}{\sqrt{i}} -1| = \frac{|x^2 -i|} {\sqrt{i} (x + \sqrt{i})} = \Theta(\frac{|L|}{M}).$$

Assume there exists $D$ such that 
$$\sum_i \|Db_i\|_{\R/\Z}^2  \le C \log N.$$
Then there is an interval $J \subset x^{2}+ [Lx/2,Lx]$ of length 
$$|J| = (|L|/ \log^{3} N)x$$ 
such that for all $j\in J$
$$\|Db_j\|_{\R/\Z} \le 1/\log N.$$
We now let $i_{0}$ be the midpoint of $J$, for which we still have 
$$|\frac{x}{\sqrt{i_{0}}} -1|  = \Theta(\frac{|L|}{M}).$$
Our first lemma is about the integral part $n_{j} = \lfloor  Db_{j}\rfloor, j \in J$.

\begin{lemma}\label{lemma:Weyl:poly} The sequence $n_{j}, j\in J$ is a polynomial in $j$ of degree  $T \le  12 A$.
\end{lemma}
We are going to show that for any $i\in J$ such that $i+T \in J$ 
\begin{equation}\label{eqn:poly}
\sum_{m=0}^T (-1)^m \binom{T}{m} n_{i+m} =0.
\end{equation}
Lemma \ref{lemma:Weyl:poly} then clearly follows.

To justify \eqref{eqn:poly} we will need some preparation. Write
$$\frac{x}{\sqrt{i+l}} =\frac{x}{\sqrt{i}} \frac{1}{\sqrt{1+\frac{l}{i}}}.$$
Note that as long as $|z| <1$,
$$(1+z)^{-1/2} = \sum_{k=0}^\infty \binom{-1/2}{k} z^k.$$
Note that as $n^{\delta_{\ast}} \le M$
$$|\frac{l}{i}| \le \frac{x^{4/3}}{x^{2}} \le \frac{1}{M^{2/3}} \ll 1,$$ 
the above series decay very fast, that is suffices to truncate at some level $k=O(1)$, however we will not truncate here for now. 

In what follows let $1\le q \le (|J|-i)/T$. Although for \eqref{eqn:poly} we will just need $q=1$, let us stay more general to treat with the sequence $D b_{i+qm}, 0\le m\le T$ and their integral parts $n_{i+qm}$. For short let 
$$y =\frac{x}{\sqrt{i}}.$$ 
Then starting from $b_{i} = \sqrt{N}e^{-x^{2}/2} \frac{x^{i}}{\sqrt{i!}}$, for each $0\le d \le T$, and for $1\le q$ so that $i+Tq$ is still in $J$, we will be focusing on the sequence $b_{i},b_{i+q},\dots, b_{i+Tq}$. For each $1\le m \le T$ we write

\begin{align}\label{eqn:b_{i}}
 b_{i+mq} = b_{i} y^{mq} \prod_{l=0}^{mq}\frac{1}{\sqrt{1+\frac{l}{i}}} &= b_{i} y^{mq}\prod_{l=0}^{mq} \sum_{k=0}^\infty \binom{-1/2}{k} (\frac{l}{i})^k \nonumber \\
 &=b_{i} y^{mq} \sum_{j=0}^\infty c_{jm} \frac{1}{i^j},
 \end{align}
where 
$$c_{j} = \sum_{i_1+\dots + i_s = j} \sum_{0\le l_1,\dots, l_s \le mq}  \binom{-1/2}{i_1} l_1^{i_1}\dots  \binom{-1/2}{i_s} l_s^{i_s}.$$
As of this point, we will be focusing only on 
$$j \le T_{0}=6A.$$
The tails can be bounded by the following observation.

\begin{fact} For each $j$, the coefficient $c_{j}$ of $1/i^{j}$ is bounded by 
$$|c_{j}|\le (Cqm)^{2j}$$ 
for some absolute constant $C$.
\end{fact}
\begin{proof} It is clear that $s\le j$. The number of ways to choose $(i_{1},\dots, i_{s})$ is bounded crudely by $O(2^{j})$,  and the contribution of $ \sum_{0\le l_1,\dots, l_s \le mq}  l_1^{i_1}\dots  l_s^{i_s} \le (mq)^{i_{1}+\dots+i_{s} +s} \le (mq)^{2j}$. Finally, $|\binom{-1/2}{i_{j}}| \le 1$.
\end{proof} 
As such, as $qm \le qT \le |J| \le (L/\log^{3} N) \le x/\log^{2} N =o(\sqrt{N}) =o(\sqrt{i})$
$$\sum_{T' \ge T_{0}} \frac{(CqT)^{2T'}}{i^{T'}}  \le \frac{(CqT)^{2T_{0}}}{i^{T_{0}}}.$$
For fix $j$, we have a more precise behavior as follows
\begin{claim} $c_{j}$ is a polynomial of degree $2j$ of $mq$, where the coefficients do not depend on $m,q$, but on $j$. More precisely we can write
$$c_{j} = \sum_{k=0}^{2j} \al_{jk}(mq)^{k},$$
where $\al_{j 2j} =(1/2)^{j}$, and $\al_{jk}$ do not depend on $m,q$.
\end{claim}
\begin{proof} This is because for a fixed tuple $(i_1,\dots, i_s)$ such that $i_1+\dots +i_s =j$
$$ \sum_{0\le l_1,\dots, l_s \le mq} l_1^{i_1}\dots l_s^{i_s}$$
is a polynomial of $mq$ of degree $j+s$.
\end{proof}

Now we consider 
\begin{align*}
\sum_{m=0}^T (-1)^m \binom{T}{m} Db_{i+mq} & = \sum_{m=0}^T (-1)^m \binom{T}{m} D b_{i} y^{mq} \sum_{j=0}^{T_{0}} c_{j} \frac{1}{i^j} + O(N^{A}|b_{i}y^{mq}|/N^{3A})\\
&= D b_{i}  \sum_{j=0}^{T_{0}} \frac{1}{i^j}  \sum_{m=0}^T (-1)^m \binom{T}{m}  y^{mq} c_{j} + O(1/N^{A})
\end{align*}
where we note that 
$$b_{i}y^{mq} = \sqrt{N} e^{-x^{2}/2} \frac{x^{i}}{\sqrt{i!}} (\frac{x}{\sqrt{i}})^{mq} \le   \sqrt{N} e^{-x^{2}/2} \frac{x^{i}}{\sqrt{i!}} \asymp 1.$$
As such, if $|D b_{i+mq} - n_{i+mq}| \le 1/\log N$ for each $0\le m\le T$, then we have 
$$\sum_{m=0}^T (-1)^m \binom{T}{m} Db_{i+mq}  = \sum_{m=0}^T (-1)^m \binom{T}{m} n_{i+mq} +O_{T}(1/\log N).$$
So
\begin{equation}\label{eqn:binom-n}
D b_{i}  \sum_{j=0}^{T_{0}} \frac{1}{i^j}  \sum_{m=0}^T (-1)^m \binom{T}{m}  y^{mq} c_{j} = \sum_{m=0}^T (-1)^m \binom{T}{m} n_{i+mq}  + O_{T}(N^{-A} + 1/\log N).
\end{equation}
Let us simplify the LHS of \eqref{eqn:binom-n}, noting that the $c_{j}$ do not depend on $m,q$.

First, the coefficient of the free-term $1/i^{0}$ is
$$\sum_{m=0}^T \binom{T}{m} (-y)^{mq} = (1-z)^T,$$
where $z=y^{q}$.

Next, the coefficient of $1/i$ is
\begin{align*}
\sum_{m=0}^T (-1)^m \binom{T}{m} y^{mq}c_{1} &=\sum_{m=0}^T (-1)^m \binom{T}{m} z^{m}(\al_{12}(mq)^2+\al_{11}(mq)+\al_{10}) \\
&= \al_{12} q^{2}\sum_{m=0}^T (-1)^m \binom{T}{m} z^{m}m^2+ \al_{11}q \sum_{m=0}^T (-1)^m \binom{T}{m} z^{m}m+\al_{10}\sum_{m=0}^T (-1)^m \binom{T}{m} z^{m}.
\end{align*}
More generally, the coefficients of $1/i^{j}$ is
\begin{align*}
\sum_{m=0}^T (-1)^m \binom{T}{m} y^{mq}c_{j} &=\sum_{m=0}^T (-1)^m \binom{T}{m} z^{m}(\al_{j2j}(mq)^{2j}+\dots+\al_{j1}(mq)+\al_{j0}) \\
&= \al_{j2j} q^{2j}\sum_{m=0}^T (-1)^m \binom{T}{m} z^{m}m^{2j}+ \dots +\al_{j1}q \sum_{m=0}^T (-1)^m \binom{T}{m} z^{m}m+\al_{j0}\sum_{m=0}^T (-1)^m \binom{T}{m} z^{m}.
\end{align*}

To continue, we record another useful fact

\begin{fact} We can write $m^k$ as a linear combination of $(m+1)\dots (m+k),(m+1)\dots (m+k-1),\dots, m+1,1$, where the coefficients are independent of $m$.
\end{fact}
As such, we can write
$$\sum_{m=0}^T (-1)^m \binom{T}{m} z^{m}m^{k} =\sum_{t=0}^{k} \beta_{t} \sum_{m=0}^T (-1)^m \binom{T}{m} (z^{m+t})^{(t)} =\sum_{t=0}^{k} \beta_{t} (z^{t} (1-z)^{T})^{(t)},$$
where $\beta_{k}=1$.

Putting together, we obtain the following simplification of \eqref{eqn:binom-n}.

\begin{lemma}\label{lemma:Weyl:key} We have
$$D\sum_{m=0}^T (-1)^m \binom{T}{m} b_{i+mq} = Db_{i}\sum_{j=0}^{T_{0}} (1/i)^{j} \sum_{j'=0}^{2j}\sum_{j'' \le j'} \al_{jj'j''}q^{j'} (z^{j''}(z-1)^{T})^{(j'')} + O(\frac{1}{\log N}),$$
where $\al_{jj'j''}$ are constants independent of $q,m$, and $\al_{000}=1$.
\end{lemma}
We are now ready to prove the polynomial behavior of $n_{i},n_{i+1}, \dots.$

\begin{proof}(of Eqn. \ref{eqn:poly} and Lemma \ref{lemma:Weyl:poly}) We start from \eqref{eqn:binom-n} with $q=1$ (that is we are considering consecutive terms in the sequence). As $q=1$, we will choose $T_{0} = \lfloor 4A \rfloor$. We also choose $T=\lfloor 12 A \rfloor.$

By Lemma \ref{lemma:Weyl:key}, we have 
$$Db_{i}\sum_{j=0}^{T_{0}} (1/i)^{j} \sum_{j'=0}^{2j}\sum_{j'' \le j'} \al_{jj'j''}q^{j'} (z^{j''}(z-1)^{T})^{(j'')}= o(1).$$
Next, because $q=1$, $A$ is sufficiently large, and $|z-1| = |y-1| = O(\frac{L}{\log^{3}n} \frac{1}{\sqrt{i}}) \le 1/M$, for each $j\in \{0,\dots, T_{0}\}$ 
$$(1/i)^{j}  ((z-1)^{T})^{(2j)} =O( (1/i)^{j} (1/M)^{T-2j}) = O( (1/M)^T) =O(1/N^{4A}).$$
 Thus, by \eqref{eqn:binom-n}, as $n_{i} \in \Z$, we must have
$$ \sum_{m=0}^T (-1)^m \binom{T}{m} n_{i+m}=0.$$  
\end{proof}
Now we are discussing the second step of the plan. Let $i_{0}$ be the midpoint of $J$. While $q$ was chosen to be 1 in the proof of Lemma \ref{lemma:Weyl:poly} so that all $Db_{i}(1/i)^{j}  ((z-1)^{T})^{(2j)}$ are small, here we will choose $q$ as large as possible
$$q = \frac{x}{\log^{4}N}.$$
To reach to contradiction, as the reader will see, the error bound of type $O(1/\log N)$ will not be enough. Our next move is to improve this error bound. From the assumption that $\sum_{i=i_{0}}^{i_{0}+q-1}\sum_{m=0}^{T} \|D_{i+mq}\|^{2} \le C\log n$, by pigeonhole principle there exists $i \in [i_{0},i_{0}+q-1]$ such that
$$\sum_{m=0}^{T} \|D b_{i+mq}\|_{\R/\Z}^{2} \le \frac{C\log N}{q}.$$
In particular, 
\begin{equation}\label{eqn:D:upgrade} 
\|D b_{i+mq}\|_{\R/\Z} =O(\sqrt{\frac{\log N}{q}}),
\end{equation} 
and also by Cauchy-Schwarz, for any $\gamma_{m}, 0\le m\le T$, $|\sum_{m=0}^{T} \gamma_{m} \|D b_{i+mq}\|_{\R/\Z}| \le \sqrt{\sum_{m} \gamma_{m}^{2}} \sqrt{\frac{C\log N}{q}}$.

With this choice of $i$, we can upgrade the estimate in Lemma \ref{lemma:Weyl:key}  to 
\begin{equation}\label{eqn:binom-n''}
Db_{i}\sum_{j=0}^{T_{0}} (1/i)^{j} \sum_{j'=0}^{2j}\sum_{j'' \le j'} \al_{jj'j''}q^{j'} (z^{j''}(z-1)^{T})^{(j'')}= O(\sqrt{\frac{\log N}{q})},
\end{equation}
where we used the fact that by Lemma \ref{lemma:Weyl:poly} and by \eqref{eqn:D:upgrade}
$$\sum_{m=0}^T (-1)^m \binom{T}{m} Db_{i+mq}  = \sum_{m=0}^T (-1)^m \binom{T}{m} (n_{i+mq} + \|D_{i+mq}\|_{\R/\Z}) =  \sum_{m=0}^T (-1)^m \binom{T}{m} \|D_{i+mq}\|_{\R/Z} = O(\sqrt{\frac{\log N}{q}}).$$

Recall that $|D|\gg 1$ and $|b_{i}|\gg 1$ for $i \in [i_{0}, i_{0}+q]$, and also 
$$|z-1|= |y^{q}-1| = |(\frac{x}{\sqrt{i_{0}}})^{q}-1| =  \Theta(\frac{Lq}{M})=  \Theta(\frac{Lx}{ (\log^{3} N)M}) = \Theta(\frac{L}{\log^{3}N}).$$
 
We next observe that, as $T=\lfloor 12 A \rfloor$, the term corresponding to $1/i^{0}$ dominates the rest.

\begin{fact}\label{fact:domination:0} 
The first term $|Db_{i}(z-1)^{T}|$ (corresponding to $j=0$) from \eqref{eqn:binom-n''} dominates all other terms.
\end{fact}
\begin{proof} Note that all $|\al_{jj'j''}|$ are bounded by $O_{A}(1)$. For each $1\le j \le T_{0}$, the dominating term corresponds to $j'=2j$, and within that the leading term is $j''=j'=2j$, that is
$$C_{A}(1+\frac{L}{\log^{3}N})^{T_{0}}|z-1|^{T-2j} (q/\sqrt{i})^{2j}  \le 2C_{A}|z-1|^{T-2j} \frac{1}{L^{2j}}(CLq/\sqrt{N})^{2j} \le  \frac{C_{A}'}{L^{2}} (Lq/\sqrt{N})^{T} \le \frac{1}{L} |z-1|^{T},$$
provided that $N, L$ are sufficiently large given $A$. 

As the number of summands is $O_{A}(1)$, we thus have, say
$$|Db_{i}\sum_{j=0}^{T_{0}} (1/i)^{j} \sum_{j'=0}^{2j}\sum_{j'' \le j'} \al_{jj'j''}q^{j'} (z^{j''}(z-1)^{T})^{(j'')}| \ge \frac{1}{2} |Db_{i}(z-1)^{T}|.$$
\end{proof}
To complete the proof, we see that because of the above
$$|Db_{i}\sum_{j=0}^{T_{0}} (1/i)^{j} \sum_{j'=0}^{2j}\sum_{j''' \le j'} c_{j,j', j''}q^{j'} (z^{j''}(z-1)^{T})^{(j'')}| \ge Db_{i}|z-1|^{T}/2 \ge \frac{1}{\log^{3T}N}.$$
This contradicts with the upper bound $O(\sqrt{\frac{\log N}{q}}) =O(N^{-1/4+o(1)})$ from \eqref{eqn:binom-n''}.
\end{proof}

\subsection{Dimension two}\label{subsection:W:2}

Our goal is to show that there is no $x\in I_{W}$ such that $D_{(.)}(\Bu_{i})$ is small.
\begin{theorem}\label{thm:Weyl:LCD:2} There does not exist $x \in I_{W}$ which obstructs the Edgeworth expansion. In other words, for any given $A>0$ and $C$, there does not exists $ r^{2} \le D_1^{2}+ D_2^{2} \le n^{A}$ such that 
$$\sum_{i} \|D_1 b_{i} + D_2 c_i\|_{\R/\Z}^{2} \le C\log N.$$
\end{theorem}

Our starting point is similar to the $1d$ case. Let $x\in I_W$. We will choose $L$ to be a sufficiently large constant given $A$. Assume there exist $D_{1},D_{2}$ such that 
$$\sum_i\|D_1 b_{i} + D_2 c_i\|_{\R/\Z}^{2}  \le C \log N.$$
Then there is an interval $J \subset x^{2}+ [Lx/2,Lx]$ of length 
$$|J| = (|L|/ \log^{3} N)x$$ 
such that for all $j\in J$
$$\|D_1 b_{i} + D_2 c_i\|_{\R/\Z}^{2} \le 1/\log N.$$
We now let $i_{0}$ be the midpoint of $J$, for which we still have 
$$|\frac{x}{\sqrt{i_{0}}} -1|  = \Theta(\frac{|L|}{M}).$$
Our first lemma is about the nearest integer 
$$n_{i} = \lfloor D_1 b_{i} + D_2 c_i\rceil, i \in J.$$

Our plan will be similar to the 1-d case, where we consider the sequences $m_{i+mq}$ (where $i$ and $m,q$ are chosen as in the 1d case).  For $b_{i+mq}$ we treat as in \eqref{eqn:b_{i}}. For $c_{i+mq}$, the part involving $ \sqrt{N} e^{-x^2/2} x \frac{x^{i+mq}}{\sqrt{(i+mq)!}} = x b_{i+mq}$ can be treated identically, for the second summand $Ne^{-x^2/2} \frac{i+mq}{x} \frac{x^{i+mq}}{\sqrt{(i+mq)!}}$, we write as follows

$$ e^{-x^2/2}\frac{i+mq}{x} \frac{x^{i+mq}}{\sqrt{(i+mq)!}}  =  e^{-x^2/2} i^{1/2} (\frac{i+mq}{i})^{1/2} \frac{x^{i+mq-1}}{\sqrt{(i+mq-1)!}}$$ 
$$=  e^{-x^2/2}  \frac{x^{i}}{\sqrt{i!}}  (i/x)  (x/\sqrt{i})^{mq} (1+\frac{mq}{i})^{1/2} \prod_{l=0}^{mq-1}(1+\frac{l}{i})^{-1/2}$$
$$=  e^{-x^2/2}  \frac{x^{i}}{\sqrt{i!}}   (i/x) (x/\sqrt{i})^{mq} (1+\frac{mq}{i})^{1/2} \prod_{l=0}^{mq-1}(1+\frac{l}{i})^{-1/2}.$$
Thus we can write the above as
\begin{align}\label{eqn:c_{i}}
e^{-x^2/2} \frac{i+mq}{x} \frac{x^{i+mq}}{\sqrt{(i+mq)!}} & = e^{-x^2/2} \frac{x^{i}}{\sqrt{i!}}  \frac{i}{x} y^{mq}  {\sqrt{1+\frac{mq}{i}}}  \prod_{l=0}^{mq-1}\frac{1}{\sqrt{1+\frac{l}{i}}}\\
 &=  e^{-x^2/2} \frac{x^{i}}{\sqrt{i!}}  \frac{i}{x}y^{mq} [\sum_{k=0}^\infty \binom{1/2}{k} (\frac{mq}{i})^k]  \prod_{l=0}^{mq-1} \sum_{k=0}^\infty \binom{-1/2}{k} (\frac{l}{i})^k \nonumber \\
 &= e^{-x^2/2} \frac{x^{i}}{\sqrt{i!}}  \frac{i}{x} y^{mq} \sum_{j=0}^\infty d_{j} \frac{1}{i^j},
 \end{align}
where 
$$d_{j} = \sum_{i_1+\dots + i_s = j} \sum_{0\le l_1,\dots, l_{s-1} \le mq-1}  \binom{-1/2}{i_1} l_1^{i_1}\dots  \binom{-1/2}{i_{s-1}} l_{s-1}^{i_{s-1}}  \binom{1/2}{i_s} (mq)^{i_s}.$$

Now, similarly to the treatment of $c_{j}$ in the 1d treatment, with $j \le T_{0}=6A$, the tails can be bounded by the following observation.

\begin{fact} For each $j$, the coefficient $d_{j}$ of $1/i^{j}$ is bounded by 
$$|d_{j}|\le (Cqm)^{2j}$$ 
for some absolute constant $C$.
\end{fact}
\begin{proof} It is clear that $s\le j$. The number of ways to choose $(i_{1},\dots, i_{s})$ is bounded crudely by $O(2^{j})$,  and the contribution of $ \sum_{0\le l_1,\dots, l_{s-1} \le mq-1}  l_1^{i_1}\dots  l_{s-1}^{i_{s-1}} (mq)^{i_{s}} \le (mq)^{i_{1}+\dots+i_{s} +s} \le (mq)^{2j}$. Finally, $|\binom{-1/2}{i_{j}}| \le 1$ and also $|\binom{1/2}{i_{j}}| \le 1$.
\end{proof} 
As such, as $qm \le qT \le |J| \le (L/\log^{3} N) \le x/\log^{2} N =o(\sqrt{M}) =o(\sqrt{i})$
$$\sum_{T' \ge T_{0}} \frac{(CqT)^{2T'}}{i^{T'}}  \le \frac{(CqT)^{2T_{0}}}{i^{T_{0}}}.$$
For fix $j$, we have a more precise behavior as follows
\begin{claim} $d_{j}$ is a polynomial of degree $2j$ of $mq$, where the coefficients do not depend on $m,q$, but on $j$. More precisely we can write
$$d_{j} = \sum_{k=0}^{2j} \beta_{jk}(mq)^{k},$$
where $\beta_{j 2j} =(1/2)^{j}$, and $\beta_{jk}$ do not depend on $m,q$.
\end{claim}
\begin{proof} This is because for a fixed tuple $(i_1,\dots, i_s)$ such that $i_1+\dots +i_s =j$
$$ \sum_{0\le l_1,\dots, l_{s-1} \le mq-1} l_1^{i_1}\dots l_s^{i_{s-1}}(mq)^{i_{s}}$$
is a polynomial of $mq$ of degree $j+s$.
\end{proof}

Putting together,

$$D_{1}b_{i+mq} +D_{2}c_{i+mq} = D_{1} b_{i} y^{mq} \sum_{j=0}^\infty c_{j} \frac{1}{i^j} +
D_{2} b_{i} y^{mq} [x\sum_{j=0}^\infty c_{j} \frac{1}{i^j}-\frac{i}{x} \sum_{j=0}^\infty d_{j} \frac{1}{i^j}].$$
$$= (D_{1} + x D_{2}) b_{i} y^{mq} \sum_{j=0}^\infty c_{j} \frac{1}{i^j} -
D_{2} \frac{i}{x} b_{i}  y^{mq}  \sum_{j=0}^\infty d_{j} \frac{1}{i^j}.$$

Using an identical machinery as before, we have the following analog of Lemma \ref{lemma:Weyl:key}.

\begin{lemma}\label{lemma:Weyl:key:d=2} We have
\begin{align*}
&D_{1}\sum_{m=0}^T (-1)^m \binom{T}{m} b_{i+mq} + D_{2}\sum_{m=0}^T (-1)^m \binom{T}{m} c_{i+mq}\\
&= (D_{1}+ xD_{2})b_{i}\sum_{j=0}^{T_{0}} (1/i)^{j} \sum_{j'=0}^{2j}\sum_{j''' \le j'} \al_{jj'j''}q^{j'} (z^{j''}(z-1)^{T})^{(j'')}\\
&-D_{2} \frac{i}{x}b_{i}\sum_{j=0}^{T_{0}} (1/i)^{j} \sum_{j'=0}^{2j}\sum_{j''' \le j'} \beta_{jj'j''}q^{j'} (z^{j''}(z-1)^{T})^{(j'')}+ O(\frac{1}{\log N}),
\end{align*}
where $\al_{jj'j''},\beta_{j,j',j''}$ are constants independent of $q,m$, and $\al_{0,0,0}=1$.
\end{lemma}

We can show that $n_{i}$ is a polynomial of $m$, similarly to Lemma \ref{lemma:Weyl:poly}.

\begin{lemma}\label{lemma:Weyl:poly:d=2} The sequence $n_{j}, j\in J$ is a polynomial in $j$ of degree  $T\le 12 A$.
\end{lemma}
The proof of this result is identical to that of Lemma \ref{lemma:Weyl:poly} by choosing $q=1$, we omit the details.

Now we move to the second step of the plan, where we will following the approach as in the 1d case to deduce contradiction. Let $i_{0}$ be the midpoint of $J$. While $q$ was chosen to be 1 in the proof above, here we will choose $q$ as large as possible
$$q = \frac{x}{\log^{4}N}.$$
From the assumption that $\sum_{i=i_{0}}^{i_{0}+q-1}\sum_{m=0}^{T} \|D_{1} b_{i+mq}+D_{2} c_{i+mq}\|^{2} \le C\log n$, by pigeonhole principle there exists $i \in [i_{0},i_{0}+q-1]$ such that
$$\sum_{m=0}^{T} \|D_{1} b_{i+mq}+D_{2} c_{i+mq}\|_{\R/\Z}^{2} \le \frac{C\log N}{q}.$$
In particular, 
\begin{equation}\label{eqn:D:upgrade2} 
\|D_{1} b_{i+mq}+D_{2} c_{i+mq}\|_{\R/\Z} =O(\sqrt{\frac{\log N}{q}}),
\end{equation} 
Another observation, similarly to Fact \ref{fact:domination:0}, is that

\begin{fact}\label{fact:domination:1} 
The first term $|(D_{1}+ xD_{2})b_{i} (z-1)^{T} -D_{2} \frac{i}{x}b_{i} (z-1)^{T}|$ (corresponding to $j=0$) from Lemma \ref{lemma:Weyl:key:d=2} dominates all other terms in the sense that 

\begin{align*}
&= (D_{1}+ xD_{2})b_{i}\sum_{j=0}^{T_{0}} (1/i)^{j} \sum_{j'=0}^{2j}\sum_{j''' \le j'} \al_{jj'j''}q^{j'} (z^{j''}(z-1)^{T})^{(j'')}\\
&-D_{2} \frac{i}{x}b_{i}\sum_{j=0}^{T_{0}} (1/i)^{j} \sum_{j'=0}^{2j}\sum_{j''' \le j'} \beta_{jj'j''}q^{j'} (z^{j''}(z-1)^{T})^{(j'')}\\
& = (D_{1}+ xD_{2})b_{i} (z-1)^{T} -D_{2} \frac{i}{x}b_{i} (z-1)^{T} \pm \frac{1}{L}|z-1|^{q} b_{i}(|D_{1}+ xD_{2}|+ |D_{2} \frac{i}{x}|).
\end{align*}
\end{fact}
\begin{proof} The proof is similar to that of Fact \ref{fact:domination:1}, and hence omitted. 
\end{proof}

In what follows $L_{0}$ is a sufficiently large constant (given $A$).

{\bf Case 1.} If $|D_{1} + x D_{2}| \le |D_{2} (x +L_{0})|$. It then follows that $|D_{1}| \le 3|x D_{2}|$, and so, as  $D_{1}^{2} + D_{2}^{2} \ge r^{2}$, we have $|D_{2} x|\gg 1$. In this case we choose $L=8L_{0}$. As $i$ is near the midpoint of $J$, we have $i/x \ge x+L/2$, and so the term  $|D_{2} \frac{i}{x}| \ge  |D_{2}| (x+L/2)$ is larger than the other term, and finally notice that 
$$|D_{2} (x+L/2)| \gg 1 \gg \sqrt{\frac{\log N}{q}}.$$

{\bf Case 2.} If $|D_{1} + x D_{2}| > |D_{2} (x +L_{0})|$. It then follows that $|D_{1}| \ge |D_{2}| L_{0}$, and so, as  $D_{1}^{2} + D_{2}^{2} \ge r^{2}$, we have $|D_{1}|\gg 1$. In this case we choose $L=L_{0}/2$, the term  $|D_{1} + x D_{2}|$ is much larger than the other terms (including $|D_{2} \frac{i}{x}|$), and notice that $|D_{2} (x+L)| \gg n^{-\eps}$ unless $D_{2} \ll n^{-1/2-\eps}$, in which case $D_{1} \gg 1$, and so $|D_{1} + x D_{2}| \gg 1$. In summary, the dominating term $|D_{1} + x D_{2}|$ is larger than $n^{-\eps}$ in both cases, which contradicts with the upper bound $\sqrt{\frac{\log N}{q}}$. 

\subsection{Dimension four}\label{subsection:W:4} Now we consider the 4d random walk $S_{n}(x,y,\bxi)=\sum_{i}\xi_{i} \Bv_{i}$.
\begin{theorem}\label{thm:Weyl:LCD:4} Let $A,C, \epsg$ be positive constants, where $\epsg<1/2$. Assume that $x<y \in I_{W}$ such that
 \begin{equation}\label{cond:s,t}
|x-y| \ge N^{\epsg}.
\end{equation}
Then there do not exist $ r^{2} \le D_1^{2}+ D_2^{2} + D_3^{2}+ D_4^{2}\le N^{A}$ such that 
$$\sum_{i} \|D_1 b_{i}(x) + D_2 c_i(x) + D_3 b_{i}(y) + D_4 c_i(y)\|_{\R/\Z}^{2} \le C\log N.$$
\end{theorem}
\begin{proof}(of Theorem \ref{thm:Weyl:LCD:4}) We will pass to the 2d case by using the separation condition $y-x>N^\epsg$ and the growth of $b_{i}(x), c_{i}(x), b_{i}(y), c_{i}(y)$ from Claim \ref{claim:Weyl:ix}. More specifically, as $y-x>N^\epsg$, we see that $y^2 -x^2 \ge N^{1/2+\epsg/2}$ provided that $N$ is sufficiently large, and so
$$[x^{2} - N^{\epsg/4} x, x^{2}+N^{\epsg/4}x] \cap [y^{2} - N^{\eps/4}y, y^{2}+N^{\epsg/4}y] = \emptyset.$$
There are two cases to consider. 

{\bf Case 1.} Assume that $D_1^2 +D_2^2 \ge D_3^2 +D_4^2$. We will be working with $i \in [x^{2} - N^{\eps/4} x, x^{2}+N^{\epsg/4}x]$ (more precisely $i\in [x^{2} -Lx, x^{2}+Lx]$). With this choice of $i$, by Claim \ref{claim:Weyl:ix}, the $b_{i}(y)$ and $c_{i}(y)$ are extremely small (of order $\exp(-N^{\epsg/2})$). Thus it suffices to work with $b_{i}(x),c_{i}(x)$, where we can invoke the 2d treatment of Subsection \ref{subsection:W:2}.

{\bf Case 2.} Assume that $D_3^2 +D_4^2 > D_1^2 +D_2^2$. This time we will be working with $i \in [y^{2} - N^{\epsg/4} y, y^{2}+N^{\epsg/4}y]$ (more precisely $i\in [y^{2} -Ly, y^{2}+Ly]$). With this choice of $i$, by Claim \ref{claim:Weyl:ix}, the $b_{i}(x)$ and $c_{i}(x)$ are extremely small, and it suffices to work with $b_{i}(y),c_{i}(y)$, where we can again rely on the 2d treatment of Subsection \ref{subsection:W:2}.
\end{proof}

To conclude,  as Condition \eqref{cond:LCD} is verified, we obtain the following estimate on the characteristic function.

\begin{theorem}\label{thm:Weyl:char:xy}
Assume that $\xi$ has mean zero and variance one. Let $A,C >0$ be arbitrarily chosen. Let $S_n$ be random walks associated to Weyl polynomials in either $1,2,4$ dimension, \eqref{eqn:S},\eqref{eqn:St}, \eqref{eqn:Sst}, respectively. Then the corresponding characteristic functions satisfy that for all index set $I\subset [n]$ of size at most $O(1)$ and $r \leq ||\eta||_2^2 \leq N^A$ we have \footnote{Where recall $\phi_{i}(\eta)$ from \eqref{eqn:cahr:bound}.}
\begin{equation}
|\prod_{i\not \in I} \phi_i(\eta)| \leq  N^{-C},
\end{equation}
provided that $N$ is sufficiently large.
\end{theorem}

\section{Technical Ingredients: Small Ball Probability Estimates}\label{sect:sbp}
First of all, using the bound for characteristic functions from Theorem \ref{thm:Weyl:char:xy} and the method from \cite[Section 3]{DNN} (or from \cite{BCP}), we have the following small ball estimates for the random walks $S_n$ in either $1,2,4$ dimension, \eqref{eqn:S},\eqref{eqn:St}, \eqref{eqn:Sst}, respectively.


 \begin{theorem}\label{thm:smallball:1} Let $C>0$ be a given constant. Assume that $x$ satisfies the condition from \ref{claim:cov:2}. Then for $\delta = N^{-C}$ and any open ball $B(a,\delta)$, we have
$$\P\left ( S_n^{0}(x,\bxi)/\sqrt{N} \in B(a,\delta) \right ) \stackrel{\rm def}{=} \P\left (P_n(x) \in B(a,\delta) \right ) =O(\delta).$$
\end{theorem}

  \begin{theorem}\label{thm:smallball:2} Let $C>0$ be a given constant. Assume that $x$ satisfies the condition from \ref{claim:cov:2}. Then for $\delta = N^{-C}$ and any open ball $B(a,\delta)$, we have
$$\P\left ( S_n(x,\bxi)/\sqrt{N} \in B(a,\delta) \right ) =O(\delta^2).$$
\end{theorem}

\begin{theorem}\label{thm:smallball:4} Let $C>0$ be a given constant. Assume that $x,y$ satisfy Condition \ref{cond:s,t}. Then for $\delta = N^{-C}$ we have
$$\P\left (S_n(x,y,\bxi)/\sqrt{N} \in B(a,\delta) \right ) =O(\delta^4).$$
\end{theorem}

We will showcase the proof of the $\R^d$  case which specializes to the $d=1,2,4$ results above.

\begin{proof}[Proof of Theorem \ref{thm:smallball:4}]  Let 
$$t_0=\delta^{-1}= N^{C}.$$ 
By a standard procedure (see for instance \cite[Eq. 5.4]{AP}), we can bound the small ball probability by characteristic functions as follows (remember $N = \sqrt{n}$):
$$\P\left (\frac{1}{\sqrt{N}} \sum_i  \xi_i \Bv_i \in B(a,\delta) \right ) \le C_d  \left (\frac{N}{t_0^2}\right )^{d/2} \int_{\R^d} \prod_i \phi_i(\Bt) e^{-\frac{N \|\Bt\|_2^2}{2 t_0^2}} d\Bt.$$
Choose $A$ to be sufficiently large compared to $C$. We break the integral into three parts, $J_1$ when $\|u\|_2 \le r$, $J_2$ when  $r \le \|u\|_2 \le N^A$, and $J_3$ for the remaining part.

For $J_1$, recall that
$$ \left |\prod \phi_i( \Bt)\right | \le \exp\left (-\sum_i \|\langle \Bv_i, \Bt \rangle\|_{\R/\Z}^2/2\right ).$$
So if $\|\Bt\|_2 \le c$ for sufficiently small $c$, then we have $\|\langle \Bv_i, \Bt \rangle\|_{\R/\Z} = \|\langle \Bv_i, \Bt \rangle\|_2$. Then for some constant $c'$, by Claim \ref{claim:cov:4},

$$\sum_i \|\langle \Bv_i, \Bt \rangle\|_{\R/\Z}^2 = \sum_i \|\langle \Bv_i, \Bt \rangle \|_2^2 = \Bt^T V_n \Bt \geq \lambda_{min}(V_n)\|\Bt\|_2^2 =  c'N\|\Bt\|_2^2.$$
Thus, 
\begin{align*}
J_1 &=  C_d \left (\frac{N}{t^2}\right )^{d/2} \int_{\|\Bt\|_2 \le r} \prod_i \phi_i(\Bt) e^{-\frac{N \|\Bt\|_2^2}{2 t_0^2}} d\Bt \le  C_d \left(\frac{N}{t_0^2}\right )^{d/2} \int_{\|\Bt\|_2 \le r}  e^{-\frac{N \|\Bt\|_2^2}{2 t_0^2} - c'N \|\Bt\|_2^2} d\Bt,
\end{align*}
and so
\begin{align*}
J_1 &\le
 C_d  \left (\frac{N}{t_0^2}\right )^{d/2} \int_{\|\Bt\|_2 \le r}  e^{-(\frac{N}{2 t_0^2} + c'N) \|\Bt\|_2^2} d\Bt =  O_d\left(\frac{1}{(c'' t_0^2 +1)^{d/2}}\right ) = O_d\left(\delta^{d}\right ).
\end{align*}

For $J_2$, recall by Theorem \ref{thm:Weyl:char:xy} that for $r \le \|\Bt\|_2 \le N^A$ and $C' = 4A+5$, we have
$$ |\prod \phi_i( \Bt)| \le N^{-C'}.$$
Thus,
\begin{align*}
J_2 &=  C_d \left (\frac{N}{t_0^2}\right )^{d/2} \int_{r \le \|\Bt\|_2 \le N^A} \prod_i \phi_i(\Bt) e^{-\frac{N \|\Bt\|_2^2}{2 t_0^2}} d\Bt  \le  C_d  \left (\frac{N}{t_0^2}\right )^{d/2} \int_{r \le \|\Bt\|_2 \le N^A}  N^{-C'} d\Bt \le   O_d\left ( \delta^{d}\right ).
\end{align*}

For $J_3$, we have
\begin{align*}
J_3 &=  C_d\left (\frac{N}{t_0^2}\right )^{d/2} \int_{ \|\Bt\|_2 \ge N^A} \prod_i \phi_i(\Bt) e^{-\frac{N \|\Bt\|_2^2}{2 t_0^2}} d\Bt = O_d\left (e^{-N}\right )
\end{align*}
as we chose $A$ sufficiently large compare to $C$.
 \end{proof}

We propose the following theorem as another corollary:
\begin{theorem}\label{thm:smallball:inf} Let $\theta>0$ and $\eps<1/2$ be given constants. Assume that $\xi_i$ are i.i.d. copies of a subgaussian random variable $\xi$ of mean zero and  variance one.  We have
$$\P\left ( \inf_{|x| \in I_W} \left \|\frac{1}{\sqrt{N}} S_n(x,\bxi)\right \|_{2} \le N^{-\theta}\right ) =O(N^{-\theta + 1/2 +\eps}).$$
\end{theorem}
\begin{proof} First, as $\xi_{i}$ are sub-gaussian, the event $\CE_{b}$ that $|\xi_{i}| \le N^{\eps}$ for $i\le M$ and $|\xi_{i}|\le N^{\eps}  \exp(c(i-M)^{2}/N^{2})$ for all $i\ge M$ is at least $1- \exp(-\Theta(N^{2\eps})) - \sum_{i} \exp(-N^{2\eps}  \exp(2c(i-M)^{2}/N^{2}) )= 1- \exp(-\Theta(N^{2\eps}))$. We will be working mostly on this event, over which by Claim \ref{claim:Weyl:ix}, for any $x\in I_{W}$ 
\begin{equation}\label{eqn:f'}
|\sum_{i\ge M} \xi_{i}  e^{-x^2/2}\frac{x^{i}}{\sqrt{i!}}| \le N^{\eps} N^{-1/2} \sum_{i\ge M}  \exp(-c(i-M)^{2}/N^{2}) =O(N^{\eps+1/2}).
\end{equation}

Notice that we have 
$$\frac{1}{\sqrt{N}} S_n(x,\bxi)  = (\sum_{i=1}^n \xi_i b_{i}(x)/\sqrt{N},\sum_{i=1}^n \xi_i c_{i}(x)/\sqrt{N}).$$

Next, for any fixed $x$ we control the magnitude of 
$$d\left (\frac{1}{\sqrt{N}} S_n(x,\bxi)\right )/dt=(f_1(x,\bxi), f_2(x,\bxi))$$ 
where 
$$f_1(x,\bxi) =\sum_{i=1}^n \xi_{i} c_{i}(x)/\sqrt{N}$$
and 
$$
f_2(x,\bxi) =\sum_{i=1}^n \xi_{i} c_{i}'(x)/\sqrt{N}.$$ 

Notice that $b_{i}(x)^2/N = e^{-x^2}\frac{x^{2i}}{i!}$ is the $x^2$-Poisson p.m.f., so we have 
$$\sum_{i=1}^n (c_{i}(x))^2/N  = \sum_{i=1}^n (b_{i}(x) \frac{i-x^2}{x})^2/N \leq 1$$ 
and 
$$\sum_{i=1}^n (c'_{i}(x))^2/N= \sum_{i=1}^n (b_{i}(x) \frac{i^2-i-2ix^2-x^2+x^4}{x^2})^2/N \leq 3$$ 
and similarly 
$$\sum_{i=1}^n (c''_{i}(x))^2/N= O(1)$$ 
based on moment computation. 

Therefore for any fixed $x$, by subgaussianity of the coefficients $\xi_{i}$ we have
\begin{equation}\label{eqn:infinity}
\P\left (|f_1(x,\bxi)| \ge N^{\eps/2}\right ) =O\left (\exp(-N^{\eps})\right ) \mbox{ and }  \P\left (|f_2(x,\bxi)| \ge N^{\eps/2}\right )=O\left (\exp(-N^{\eps})\right ).
\end{equation} 
Notice that on $\CE_b$, similarly to \eqref{eqn:f'} we have have $\sup_{x\in I_{W}}|f_i'(x,\bxi)|=O(N^{1/2+\ep})$. By a standard net argument that considers $I_{W}$ as a union of $N^{2}$ equal intervals, we obtain from \eqref{eqn:infinity} and the union bound that
\begin{equation}\label{eqn:derivative}
\P\left (\sup_{x\in I_{W}}\left \|d\left (\frac{1}{\sqrt{N}} S_n(x,\bxi)\right )/dt\right \|_2 \ge N^{\eps/2}\right ) =O\left (\exp(-N^{\eps})\right ).
\end{equation}
We will condition the complement of this event. Decompose $I_{W}$ into $O(N^{1/2+\theta})$ intervals of length $N^{-\theta}$ each. For each such interval $I$, we estimate the probability that $\inf_{t \in I}|S_n(x,\bxi)|/\sqrt{N} \le N^{-\theta}$. By \eqref{eqn:derivative}, this implies that for the midpoint $x_I$ we have 
$$\frac{1}{\sqrt{N}} S_n(x_I,\bxi) \le N^{-\theta} + N^{\eps/2} N^{-\theta} = O(N^{\eps/2 - \theta}).$$ 
However, by using Theorem \ref{thm:smallball:2}, we can control this event by
$$\P\left (\left \|\frac{1}{\sqrt{N}} S_n(x_I,\bxi)\right \|_{2} \le N^{-\theta+ \eps/2}\right ) =O(N^{-2\theta+ \eps}).$$
Taking union bounds over the midpoints of the $O(N^{1/2+\theta})$ intervals we obtain the bound $O(N^{-\theta+1/2 +\eps})$ as claimed. 
\end{proof}

Notice that the above small ball probability theorems can also be extended to random variables with mean $0$, variance $1$ and bounded high moments.

Now recall the  approximated Kac-Rice (\ref{KR-approx}) presented in the introduction. The sense in which this formula turns exact is in the limit $\delta\rightarrow 0$. For our purposes, we must relate this parameter to the main ``thermodynamic" limit, that is the degree of the polynomial $n\rightarrow \infty$. Indeed set, $\delta=N^{-5}$, then we shall need the following technical result to justify the application of Kac-Rice rigorously, that we can do so with probability at least $1-O(N^{-4})$.

\begin{lemma}\label{lemma:smallball:KR} For $I=[a,b]$ with the form of $I_W$, with probability at least $1-N^{-\theta+1}$ the following holds for all $x\in I$ 
$$|P_{n}(a)|, |P_{n}(b)| > N^{-\theta} \mbox{ and } |P_{n}(x)| + |P_{n}'(x)|>N^{-\theta}.$$ 
\end{lemma}
\begin{proof} It follows from Theorem \ref{thm:smallball:1} that
$$\P(|P_{n}(a)|>N^{-\delta} \vee |P_{n}(b)| > N^{-\theta}) = O(N^{-\theta}),$$
while trivially from  Theorem \ref{thm:smallball:inf} that 
$$\P( \inf_{x\in I_{W}}|P_{n}(x)| + |P_{n}'(x)|>2N^{-\theta}) = O(N^{-\theta+1}).$$
\end{proof}

\section{Edgeworth Expansion Redux}\label{sect:Redux}

We will show the following comparison result for the functions involving in the variance formula.
Let $X_{n}(x,\bxi)=(X_{n,k}(x,\bxi))_{k=1}^{n}$ be the $2 \times n$ matrix, where 
$$X_{n,k}(x,\bxi) = \xi_k\Bu_{k}^{T}(x).$$ 
Similarly,
Let $X_{n}(x,y,\bxi)=(X_{n,k}(x,y,\bxi))_{k=1}^{n}$ be the $4 \times n$ matrix, where 
$$X_{n,k}(x,y,\bxi) = \xi_k\Bv_{k}^{T}(x,y).$$ 

The following result follows \cite[Proposition~5.1]{DNN} to pass from Theorem~\ref{thm:EW:linear} to the rougher statistics of \( F_{\delta} \) and \( \Psi_\delta \). In this step, we employ a more detailed form of the covariance matrices obtained in Claim~\ref{claim:cov:2} and Claim~\ref{claim:cov:4}. We also refer the reader to Subsection~\ref{subsection:multivariate} for the precise definitions of \( \Gamma_{n,1} \) and \( \Gamma_{n,2} \).

\begin{prop}\label{prop:EW:delta} Assume that $\xi_{i}$ are i.i.d. copies of subgaussian $\xi$ of mean zero, variance one. Assume that $x,y\in I_{W}$ satisfying Condition \ref{cond:s,t}. Then 
\begin{equation}\label{eqn:P}
|\E F_\delta(P_n(x))- \E F_\delta(P_n(x,G))| = O(\frac{1}{N^{1/2}})
\end{equation}
and
\begin{align}\label{eqn:t,Y}
\Big|  \E \Phi_\delta\left (S_n(x,\bxi)\right )  &- \E \Phi_\delta(S_n(x,G)) - \frac{1}{N} \E\big[\Phi_\delta(W_2) \Gamma_{n,2} \big(X_n(x,\bxi), W_2)\big)\big] \Big| \nonumber \\
&=O(  \frac{1}{N^{3/2}}+  \frac{1}{N} r_n(x)) 
\end{align}
and
\begin{align}\label{eqn:s,t,Y}
\Big|  \E \Psi_\delta \left(S_n(x,y,\bxi)\right)&- \E \Psi_\delta (S_n(x,y,G)) - \frac{1}{N} \E \big [\Psi_\delta ( W_4) \Gamma_{n,2} \big(X_n(x,y,\bxi),W_4\big )\big]\Big| \nonumber \\
& =O(  \frac{1}{N^{3/2}} + \frac{1}{N} r_n(x,y)),
\end{align}
where $W, W_2,W_4$ are standard gaussian vectors in $\R, \R^2$ and $\R^4$ respectively, and where the implied constants are allowed to depend on the $M_0$-moment of $\eta$, and on a lower bound of the least singular values of $V_n(x), V_n(x,y)$ and $I_2,I_4$. Furthermore, we have the following bounds 
\begin{equation}\label{eqn:r_{n}}
r_n(x)=O(\| V_n(x) - I_2\|_{2})  \mbox{ and }  r_n(x,y) = O(\| V_n(x,y) - I_4\|_{2}).
\end{equation}
\end{prop}
We note that with $x,y\in I_{W}$ satisfying Condition \ref{cond:s,t}, $r_{n}(x), r_{n}(x,y) \le e^{-N^{\epsg}}$, and hence in the estimates above the error bounds are indeed $O(1/N^{3/2})$.

We also note that if $f$ is an even function then using the fact that the standard
gaussian distribution is symmetric and the fact that Hermite polynomials of odd degrees are odd functions
we obtain
$$\E[f(W) \Gamma_{n,1}(X,W)]=0.$$

\section{Completing the proof of Theorem \ref{thm:expectation:W}}\label{sect:exp} 

\subsection{Proof of \eqref{eqn:expmain:1}} 
We focus on $I_W=[c_{1}M,c_{2}M]$ case. Recall that $\Phi_\delta(t_1,t_2)=|t_2|\mathbbm{1}_{|t_1|<\delta}$, and by Kac-Rice formula:
\begin{align}
    \E N_{I_W}=\frac{1}{2\delta}\int_{I_W}\E|P_n'(x)|\mathbbm{1}_{|P_n(x)|<\delta} =\frac{1}{2\delta}\int_{I_W}\E\Phi_\delta(S_n(x,\bxi))dx.
\end{align}

Let us start by computing for fixed $x\in I_W$ the expectation $\E\Phi_\delta((S_n(x,\bxi)))$. We apply Proposition \ref{prop:EW:delta} to obtain that 
\begin{align}\label{prop411}
\Big|  \E \Phi_\delta\left (S_n(x,\bxi)\right )  &- \E \Phi_\delta(S_n(x,G)) - \frac{1}{N} \E\big[\Phi_\delta(W_2) \Gamma_{n,2} \big(X_n(x,\bxi), W_2)\big] \Big| \nonumber \\
&\le  O(\frac{1}{N^{3/2}})+  \frac{1}{N} r_n(x, \Phi_\delta).
\end{align}
By Claim \ref{claim:cov:2} and \eqref{eqn:r_{n}} we have 
$$ r_n(x, \Phi_\delta) =o(1).$$ 

It suffices to work with the gaussian terms 
$$\frac{1}{2\delta}\int_{I_W}  \E \Phi_\delta(S_n(x,G))  +\frac{1}{N} \E\big[\Phi_\delta(W_2) \Gamma_{n,2} \big(X_n(x,\bxi), W_2 \big)\big]dx.$$
We will express $ \E\big[\Phi_\delta(S_n(x,G))  + \frac{1}{N}\Phi_\delta(W_2) \Gamma_{n,2} \big(X_n(x,\bxi), W_2 \big)\big]$ as an integral of $\int_{\R^{2}} \Phi_\delta(t) f(t)$, where $t=(t_{1},t_{2})\in \R^{2}$, and $f(t)$ is a product the (standard) gaussian density $\phi(t)$, with other polynomials (of $t$) of coefficients depending on $x$ basing on the explicit form of $\Gamma_{n,2}$.

For $X=(X_{n,1},\dots,X_{n,n})$, recall we also defined:
\begin{equation}\label{eqn:Gamma:22}
\Gamma_{n,2}(X,t) = \Gamma_{n,2}'(X,t) + \Gamma_{n,2}''(X,t)
\end{equation}
where expanding the formula
\begin{equation}
\label{exp-gamma}
\begin{split}
\Gamma_{n,2}'(X,t) 
  &= \frac{1}{24}\sum_{|\gamma|=4}c_n(\gamma,X)\,H_\gamma(t)\\
  &= \frac{1}{24}c_n((4,0),X)\,H_{(4,0)}(t)
   + \frac{1}{24}c_n((0,4),X)\,H_{(0,4)}(t)\\
  &\quad+ \frac{1}{6}c_n((3,1),X)\,H_{(3,1)}(t)
   + \frac{1}{6}c_n((1,3),X)\,H_{(1,3)}(t)
   + \frac{1}{4}c_n((2,2),X)\,H_{(2,2)}(t).
\end{split}
\end{equation}

and 
\begin{equation}
\label{exp-gammaprime}
\begin{split}
\Gamma_{n,2}''(X,t) 
  &= \frac{1}{72} \sum_{|\alpha|=3}\sum_{|\beta|=3} c_n(\alpha,X) c_n(\beta,X)H_{\{\alpha,\beta\}}(t)\\
  &= \frac{(\frac{1}{6}c_n((3,0),X)\,H_{(3,0)}(t)
   + \frac{1}{2}c_n((2,1),X)\,H_{(2,1)}(t)+ \frac{1}{2}c_n((1,2),X)\,H_{(1,2)}(t)
   + \frac{1}{6}c_n((0,3),X)\,H_{(0,3)}(t))^2}{2}.
\end{split}
\end{equation}
Next, recall that
$$\Bu_{i}(x) = (b_{i}(x), c_{i}(x))=\Big(\sqrt{N} e^{-x^2/2} \frac{x^{i}}{\sqrt{i!}}, \sqrt{N} e^{-x^2/2} (\frac{i}{x} \frac{x^{i}}{\sqrt{i!}} -x \frac{x^{i}}{\sqrt{i!}})\Big) =\sqrt{N} e^{-x^2/2} \frac{x^{i}}{\sqrt{i!}} (1, \frac{i}{x} -x).$$
So $d=2$, and $X_{n,i} = \xi_{i} \Bu_{i}$. From Subsection \ref{subsection:multivariate} we have
$$\Delta_\al(X_{n,i})= (\chi_{|\al|}(\xi) -\chi_{|\al|}(G)) \Bu_{i}^{\al} \mbox{ and } c_{n}(\al, X_{n}) =\frac{1}{N} \sum_{i=1}^{n} \Delta_{\al}(X_{n,i}).$$
For example,
\begin{itemize}
\item if $\al =(3,0)$ or $(0,3)$, then (as $\E \xi =0, \E \xi^{2}=1$)
$$\Delta_\al(X_{n,i}) = (\E\xi^{3} - \E G^{3}) b_{i}^{3}(x) \mbox{ or }  (\E \xi^{3} - \E G^{3}) c_{i}^{3}(x),$$
and so
$$ c_{n}(\al, X_{n}) = \frac{1}{N}  (\E\xi^{3} - \E G^{3})  \sum_{i=1}^{n} b_{i}^{3}(x) \mbox{ or }   \frac{1}{N}  (\E \xi^{3} - \E G^{3}) \sum_{i=1}^{n} c_{i}^{3}(x).$$
\item if $\al =(3,1)$ or (1,3), then 
$$\Delta_\al(X_{n,i}) =  (\E\xi^{4} - \E G^{4})  b_{i}^{3}(x)c_{i}(x) \mbox{ or }  (\E \xi^{4} - \E G^{4}) b_{i}(x)c_{i}^{3}(x),$$
and so
$$ c_{n}(\al, X_{n}) = \frac{1}{N}   (\E\xi^{4} - \E G^{4})  \sum_{i=1}^{n} b_{i}^{4}(x) \mbox{ or }   \frac{1}{N}  (\E \xi^{4} - \E G^{4}) \sum_{i=1}^{n} c_{i}^{4}(x).$$
\end{itemize}
We then denote $C_3 = \E\xi^{3} - \E G^{3}$ and $C_4 =\E\xi^{4} - \E G^{4}$.

So our computation reduces to the above expectation with ($C_{\alpha, \beta}$ and $C_\gamma$ are the constants in the expansion):

\begin{align*}
    f(t)=\frac{1}{2\pi}e^{-\frac{(t_1^2+t_2^2)}{2}}\left(1+\frac{1}{N}\sum_{|\alpha|=3}\sum_{|\beta|=3} C_{\alpha, \beta} \cdot c(\alpha,X_{n})c(\beta,X_{n})H_{\{\alpha, \beta\}}(t)+\frac{1}{N}\sum_{|\gamma|=4} C_\gamma\cdot  c(\gamma,X_{n})H_{\gamma}(t)\right).
\end{align*}

The first term of the Edgeworth expansion yields the gaussian expectation. Note that this indeed matches our simulation of $\frac{1}{\pi}$ with $[c_1M,c_2M] = [0,\sqrt{n}]$. \\
\begin{align}\int_{I_W} \int_{\R^2}  \Phi_\delta(t_1,t_2) \frac{1}{2\pi}e^{-\frac{(t_1^2+t_2^2)}{2}} = &(c_{2}-c_{1})M \cdot\int_{\R}|t_1|H_0(t_1)\phi(t_1)dt_1 \nonumber\\
&\times \int_{\R}\frac{1}{2\delta}\mathbbm{1}_{|t_2|<\delta}H_0(t_2)\phi(t_2)dt_2 = \frac{|I_{W}|}{\pi}.
\nonumber\end{align}

And, up to leading order, we have that $\E N_{I_W}(\xi)-\E N_{I_W}(G)$ equals:  $$\int_{I_W} \int_{\R^2}  \Phi_\delta(t_1,t_2)f(t_1,t_2)-\Phi_\delta(t_1,t_2)\frac{1}{2\pi}e^{-\frac{(t_1^2+t_2^2)}{2}} = \int_{I_W} \int_{\R^2}  \Phi_\delta(t_1,t_2)\frac{1}{2\pi}e^{-\frac{(t_1^2+t_2^2)}{2}}(\frac{1}{N}(\Gamma_{n,2}'(x)+\Gamma_{n,2}''(x))).
$$
We consider the terms separately.
\subsection{The term involving $\Gamma_{n,2}'(x)$} The sum over partitions in \eqref{exp-gamma} has 5 summands: $(4,0), (3,1), (2,2), (1,3), (0,4)$. Each of them is of the form: 
\begin{align*}
\label{factorout}    \frac{C_\gamma}{N}\int_{I_W} c(\gamma,x)dx\times \frac{1}{2\delta}\int_{\R^2}\Phi_\delta(t_1,t_2)H_{\gamma}(t_1,t_2)\phi(t_1)\phi(t_2)dt_1dt_2.
\end{align*}
Note again that, by similar symmetry considerations, the terms (3,1) and (1,3) are vanishing (i.e. odd function over a symmetric interval). For the rest,   set $\delta=n^{-5}$ and observe that the integral w.r.t.$(t_1, t_2)$ factorizes into a gaussian moment term and a term involving a Dirac delta in the $n\rightarrow\infty$ limit. Indeed,   we have:
\begin{align*}
  \frac{1}{2\delta}\int_{\R^2}\Phi_\delta(t_1,t_2)H_{\gamma}(t_1,t_2)\frac{1}{\sqrt{2\pi}^2}e^{-\frac{(t_1^2+t_2^2)}{2}}dt_1dt_2=\int_{\R}|t_1|H_{\gamma_1}(t_1)\phi(t_1)dt_1 \times \int_{\R}\frac{1}{2\delta}\mathbbm{1}_{|t_2|<\delta}H_{\gamma_2}(t_2)\phi(t_2)dt_2. 
\end{align*}
Straightforward computations yield
\begin{enumerate}
    \item For the $(4,0)$ term we have:
  $$\int_{\R}|t_1|H_4(t_1)\phi(t_1)dt_1=-\frac{\sqrt{2}}{\sqrt{\pi}},\quad \quad \int_{\R}\frac{1}{2\delta}\mathbbm{1}_{|t_2|<\delta}H_0(t_2)\phi(t_2)dt_2=\frac{1}{\sqrt{2\pi}};$$
    \item For the $(2,2)$ term we have:
    $$\int_{\R}|t_1|H_2(t_1)\phi(t_1)dt_1=\frac{\sqrt{2}}{\sqrt{\pi}},\quad \quad  \int_{\R}\frac{1}{2\delta}\mathbbm{1}_{|t_2|<\delta}H_2(t_2)\phi(t_2)dt_2=-\frac{1}{\sqrt{2\pi}};$$
    \item For the $(0,4)$ term we have:
    $$\int_{\R}|t_1|H_0(t_1)\phi(t_1)dt_1=\frac{\sqrt{2}}{\sqrt{\pi}},\quad \quad  \int_{\R}\frac{1}{2\delta}\mathbbm{1}_{|t_2|<\delta}H_4(t_2)\phi(t_2)dt_2=\frac{3}{\sqrt{2\pi}}.$$
    \end{enumerate}

To conclude the computation of the $\Gamma_{n,2}'$ we need only compute the integral with respect to $x$ in \eqref{factorout} below. Note that one has more explicitly, when averaging cumulants we are dealing with integrals of the form
\begin{align}
   \frac{1}{N} \int_{I_W}c(\gamma,x)dx=C_4 \int_{I_{W}}\sum_{i=0}^n \left(e^{-\frac{x^2}{2}}\frac{x^{i}}{(\sqrt{i!})}\right)^{4}\left(\frac{i-x}{x^2}\right)^{\gamma_2}.
\end{align}

\begin{lemma}\label{E-lemma} Assume that $x$ has order  $M$ from \eqref{eqn:I_{W}:gen}. For given integers $s,t\ge 0$, $s$ even and $t\ge 2$, there exists $\epsilon=\epsilon_{\delta_{\ast}, s,t}>0$ such that
$$\sum_{i=0}^ne^{-tx^2/2}\frac{x^{ti}}{(\sqrt{i!})^t}(\frac{i-x^2}{x})^s = C (t,s)x^{-c(t)}+O(e^{-\frac{t}{4}{n^{2\epsilon}}})$$
where
\begin{align}
    C (t,s)x^{-c(t)}=(2\pi)^{-\frac{t}{4}}\left(\frac{4}{t}\right)^{\frac{s+1}{2}} \Gamma\left(\frac{s+1}{2}\right) x^{-\frac{t-2}{2}}.
\end{align}
\end{lemma}
We will present a proof of this lemma in Appendix \ref{sect:E-lemma}. Assuming this result for a moment. Simply reading off the formula of the previous \hyperlink{lemma-formula}{Lemma 5.1} then integrating we have the following (using the assumption that $I_{W} = (c_{1}M, c_{2}M)$)

\begin{enumerate}
    \item Setting $t=4$ and $s=0$,  $\int_{I_W}c((4,0),x)dx=\frac{1}{2\sqrt{\pi}}\log(\frac{c_2}{c_1}) $
    \item Setting $t=4$ and $s=2$, $\int_{I_W}c((2,2),x)dx=\frac{1}{4\sqrt{\pi}}\log(\frac{c_2}{c_1}) $
     \item Setting $t=4$ and $s=4$, $\int_{I_W}c((0,4),x)dx= \frac{3}{8\sqrt{\pi}}\log(\frac{c_2}{c_1})$.
\end{enumerate}
So in total the $\Gamma_{n,2}'$ contribution after integration is
\begin{align*}
    \Gamma_{n,2}'&= \frac{1}{24}c_n((4,0),X)\,H_{(4,0)}(t)
   + \frac{1}{24}c_n((0,4),X)\,H_{(0,4)}(t)\\
  &\quad+ \frac{1}{6}c_n((3,1),X)\,H_{(3,1)}(t)
   + \frac{1}{6}c_n((1,3),X)\,H_{(1,3)}(t)
   + \frac{1}{4}c_n((2,2),X)\,H_{(2,2)}(t).\\
    &\to C_4\left(\frac{1}{24}\cdot\frac{-1}{\pi}\times \frac{1}{2\sqrt{\pi}}\log(\frac{c_2}{c_1}) +\frac{1}{4}\cdot\frac{-1}{\pi}\times  \frac{1}{4\sqrt{\pi}}\log(\frac{c_2}{c_1}) +\frac{1}{24}\cdot\frac{3}{\pi}\times \frac{3}{8\sqrt{\pi}}\log(\frac{c_2}{c_1})\right).
\end{align*}

\subsection{The term involving $\Gamma_{n,2}''(x)$} Next we analyze the sum  involving $\Gamma_{n,2}''$ in \ref{exp-gammaprime} which has terms of the form:
\begin{align}
  \label{factoroutprime}  \frac{1}{N}\int_{I_W}c(\alpha,x)c(\beta,x)dx\times \frac{1}{2\delta}\int_{\R^2}\Phi_\delta(t_1,t_2)H_{\alpha, \beta}(t_1,t_2)\frac{1}{2\pi}e^{-\frac{(t_1^2+t_2^2)}{2}}dt_1dt_2.
\end{align}

Analogous symmetry considerations of the integral w.r.t. $(t_1,t_2)$ allow us to conclude that there are 8 nonzero terms, namely:
\begin{enumerate}
    \item $(3,0)\times(3,0)$.
\item $(0,3)\times (0,3)$.
\item $(1,2)\times (1,2)$.
\item $(2,1)\times (2,1)$.
\item $(3,0)\times(1,2)$ where the integrals w.r.t. $x$ and $(t_1,t_2)$ reduce to the  $(2,1)\times (2,1)$ case.
\item $(1,2)\times(3,0)$ where the integrals w.r.t. $x$ and $(t_1,t_2)$ reduce to the  $(2,1)\times (2,1)$ case.
\item $(0,3)\times(2,1)$ where the integrals w.r.t. $x$ and $(t_1,t_2)$ reduce to the  $(1,2)\times (1,2)$ case.
\item $(2,1)\times(0,3)$ where the integrals w.r.t. $x$ and $(t_1,t_2)$ reduce to the  $(1,2)\times (1,2)$ case.

\end{enumerate}

Similarly, as before, the integral w.r.t. $(t_1,t_2)$ in \eqref{factoroutprime} factors neatly into the following easily computable terms. Recall that $H_{\alpha,\beta}=H_{\alpha}\times H_{\beta}$.

\begin{enumerate}
    \item For the $(3,0)\times (3,0)$ term we have:
    $$\int_{\R}|t_1|H_3^2(t_1)\phi(t_1)dt_1=\frac{18\sqrt{2}}{\sqrt{\pi}},\quad \quad \int_{\R}\frac{1}{2\delta}\mathbbm{1}_{|t_2|<\delta}H_0^2(t_2)\phi(t_2)dt_2=\frac{1}{\sqrt{2\pi}};$$
    \item For the $(0,3)\times (0,3)$ term we have:
    $$\int_{\R}|t_1|H_0^2(t_1)\phi(t_1)dt_1=\frac{\sqrt{2}}{\sqrt{\pi}},\quad \quad \int_{\R}\frac{1}{2\delta}\mathbbm{1}_{|t_2|<\delta}H_3^2(t_2)\phi(t_2)dt_2=0;$$
    \item For the $(1,2)\times (1,2)$ term we have:
    $$\int_{\R}|t_1|H_1^2(t_1)\phi(t_1)dt_1=2\frac{\sqrt{2}}{\sqrt{\pi}},\quad \quad \int_{\R}\frac{1}{2\delta}\mathbbm{1}_{|t_2|<\delta}H_2^2(t_2)\phi(t_2)dt_2=\frac{1}{\sqrt{2\pi}};$$
     \item For the $(2,1)\times (2,1)$ term we have:
    $$\int_{\R}|t_1|H_2^2(t_1)\phi(t_1)dt_1=5\frac{\sqrt{2}}{\sqrt{\pi}},\quad \quad  \int_{\R}\frac{1}{2\delta}\mathbbm{1}_{|t_2|<\delta}H_1^2(t_2)\phi(t_2)dt_2=0;$$
    \item For the $(3,0)\times (1,2)$ and $(1,2)\times (3,0)$ terms we have
    $$\int_{\R}|t_1|H_3(t_1)H_1(t_1)\phi(t_1)dt_1=\frac{2\sqrt{2}}{\sqrt{\pi}},\quad \quad \int_{\R}\frac{1}{2\delta}\mathbbm{1}_{|t_2|<\delta}H_0(t_2)H_2(t_2)\phi(t_2)dt_2=-\frac{1}{\sqrt{2\pi}};$$
    \item For the $(0,3)\times (2,1)$ and $(2,1)\times (0,3)$ terms we have:
    $$\int_{\R}|t_1|H_0(t_1)H_2(t_1)\phi(t_1)dt_1=\frac{\sqrt{2}}{\sqrt{\pi}},\quad \quad \int_{\R}\frac{1}{2\delta}\mathbbm{1}_{|t_2|<\delta}H_1(t_2)H_3(t_2)\phi(t_2)dt_2=0.$$
\end{enumerate}

To conclude the computation, we must deal with the integral over $x\in I_W$ in \eqref{factoroutprime}. Once again, we shall use the formula from \hyperlink{lemma-formula}{Lemma 5.1} as well as the observation that:

\begin{align*}
    \frac{1}{N}\int_{I_W}c(\alpha,x)c(\beta,x)dx=C_3^2\int_{I_{W}}[\sum_{i=0}^n \left(e^{-\frac{x^2}{2}}\frac{x^{i}}{(\sqrt{i!})}\right)^{3}\left(\frac{i-x}{x^2}\right)^{\alpha_2}][\sum_{i=0}^n \left(e^{-\frac{x^2}{2}}\frac{x^{i}}{(\sqrt{i!})}\right)^{3}\left(\frac{i-x}{x^2}\right)^{\beta_2}]dx.
\end{align*}

\begin{enumerate}
    \item Setting $t=3$ and $s=0$ and squaring,  $\int_{I_W}c((3,0),x)c((3,0),x)dx= \frac{\sqrt{2}}{3\sqrt{\pi}}  \log(\frac{c_2}{c_1})$;
    \item Setting $t=3$, $s=0$ and $t=3$, $s=2$, $\int_{I_W}c((3,0),x)c((1,2),x)dx=\frac{4\sqrt{2}}{9\sqrt{\pi}} \log(\frac{c_2}{c_1})$;
     \item Setting $t=3$ and $s=2$ and squaring, $\int_{I_W}c((1,2),x)c((1,2),x)dx=  \frac{8\sqrt{2}}{27\sqrt{ \pi}}\log(\frac{c_2}{c_1})$.
\end{enumerate}
Putting all together,   the  $\Gamma_{n,2}''$ contribution after integration with respect to $x\in I_{W}$ is
\begin{align*} 
&\frac{(\frac{1}{6}c_n((3,0),X)\,H_{(3,0)}(t)
   + \frac{1}{2}c_n((2,1),X)\,H_{(2,1)}(t)+ \frac{1}{2}c_n((1,2),X)\,H_{(1,2)}(t)
   + \frac{1}{6}c_n((0,3),X)\,H_{(0,3)}(t))^2}{2}\\
   & \to C_3^2\left(\frac{1}{72}\cdot\frac{18}{\pi}\times \frac{\sqrt{2}}{3\sqrt{\pi}} \log(\frac{c_2}{c_1}) +\frac{1}{12}\cdot(-\frac{2}{\pi})\times \frac{4\sqrt{2}}{9\sqrt{\pi}} \log\left( \frac{c_2}{c_1} \right)+\frac{1}{8}\cdot\frac{2}{\pi}\times \frac{8\sqrt{2}}{27\sqrt{ \pi}}\log(\frac{c_2}{c_1})\right).
\end{align*}
Integrating with respect to $t\in \R^{2}$ then leads to
\begin{align*}
& C_4\left(\frac{1}{24}\cdot\frac{-1}{\pi}\times \frac{1}{2\sqrt{\pi}}\log(\frac{c_2}{c_1}) +\frac{1}{4}\cdot\frac{-1}{\pi}\times  \frac{1}{4\sqrt{\pi}}\log(\frac{c_2}{c_1}) +\frac{1}{24}\cdot\frac{3}{\pi}\times \frac{3}{8\sqrt{\pi}}\log(\frac{c_2}{c_1})\right)\nonumber\\
    &+C_3^2\left(\frac{1}{72}\cdot\frac{18}{\pi}\times \frac{\sqrt{2}}{3\sqrt{\pi}} \log(\frac{c_2}{c_1}) +\frac{1}{12}\cdot(-\frac{2}{\pi})\times \frac{4\sqrt{2}}{9\sqrt{\pi}} \log\left( \frac{c_2}{c_1} \right)+\frac{1}{8}\cdot\frac{2}{\pi}\times \frac{8\sqrt{2}}{27\sqrt{ \pi}}\log(\frac{c_2}{c_1})\right).
\end{align*}

By \eqref{prop411}, the errors still vanish after integration. Therefore, we complete the proof of \eqref{eqn:expmain:1}.

\subsection{Proof of \eqref{eqn:expmain:2}} Let $M_{0} \to \infty$ with $n$, but $M_{0}$ grows slower than $\log \log n$ and $\ell =\log_{2} (n^{1/2}/2M_{0})$ is an integer. We decompose $[0,(1-c)\sqrt{n}]$ into
$$[0,(1-c)\sqrt{n}]= [0,M_{0}] \cup_{i=1}^{\ell} I_{i} \cup (n^{1/2}/2, (1-c)n^{1/2}] $$
where $I_{i}=(2^{i-1} M_{0},2^{i} M_{0}]$ for $1\le i\le \ell-1$.

Now by \eqref{eqn:exp:0} (or more precisely Lemma \ref{lemma:upperbd}) and \eqref{eqn:rho:W} we have, as $M_{0}\to \infty$
$$\E N_{[0,M_{0}],\bxi} \le M_{0}^{A_{0}}.$$ 
One the other hand \eqref{eqn:expmain:1} applied to $I_{i}, 1\le i\le \ell-1$ yields
$$\E N_{I_{i},\bxi}  = \E N_{I_{i},G}+C_{\xi} \log 2 +o(1).$$ 
Additionally, \eqref{eqn:expmain:1} also yields
 $$\E N_{ (n^{1/2}/2, (1-c)n^{1/2}] ,\bxi}  - \E N_{(n^{1/2}/2, (1-c)n^{1/2}] ,G}=C_{\xi} (\log [2(1-c)]) +o(1).$$
We thus obtain
$$\E N_{[0,(1-c)n^{1/2}],\bxi}  = \E N_{[0,(1-c)n^{1/2}],G} + C_{\xi} \ell \log 2+o(\ell)+o(M_0)=  \E N_{[0,(1-c)n^{1/2}],G} + (C_{\xi}/2+o(1))\log n.$$

\section{Completing the proof of Theorem \ref{thm:var:W}}\label{sect:var}

We will mainly focus on the first part.

\subsection{Proof of \eqref{eqn:varmain:1}} Here we follow the method of \cite{DNN} and \cite{BCP} and denote by $N_{I,\bxi}$ the number of real root of over interval $I$ with random i.i.d coefficients $\xi_{i}$.

Let $I_W=\cup I_s$ with $I_s=(c_{1}M+sN^\eps, c_{1}M+(s+1)N^\eps)$ be a decomposition of $I_W$ with a small constant $\eps$ to be chosen sufficiently small. The variance of the number of roots of a gaussian (or otherwise) polynomial decomposes into \textbf{off-diagonal} and \textbf{diagonal} terms:

\begin{align}
\var(N_G)& = \var(\sum N_{I_s ,G}) = \sum \var(N_{I_s ,G}) + 2\sum_{s<t} \cov(N_{I_s ,G},N_{I_t ,G}) \nonumber \\
&\label{decompose-var}=\sum_{|s-t| \leq 1} \cov(N_{I_s ,G},N_{I_t ,G}) + \sum_{|s-t| > 1} \cov(N_{I_s ,G},N_{I_t ,G}).
\end{align}
 
Recall from \eqref{eq:def:vn} of Section \ref{sect:kr} that:
\begin{align}\label{eq:def:vn2}
v_n(x,y,\bxi)&=\frac{1}{(2\delta)^2}\cov(P'_n(x) 1_{ |P_n(x)|<\delta }, P'_n(y) 1_{ |P_n(y)|<\delta }) \nonumber\\
&= \E \Psi_{\delta}\left (\frac{1}{\sqrt{N}}S_n(x,y,\bxi)\right)  - \E \Phi_\delta(x,\bxi) \E \Phi_\delta(y,\bxi).
\end{align}

We call the diagonal blocks as $D_{diag}$ and off-diagonal ones as $D_{\operatorname{off}}$. 

\subsection{Treatment for the off-diagonal term} We begin by controlling the \textbf{off-diagonal} term. First of all, we will show the Edgeworth error vanishes:

\begin{lemma}\label{lm:v1}  We have
	\begin{equation}\label{key}
	\sum_{|s-t| > 1} (\cov(N_{I_s ,\bxi},N_{I_t ,\bxi})- \cov(N_{I_s ,G},N_{I_t ,G})) = \sum_{|s-t| > 1} \int_{I_s}\int_{I_t} (v_n(x,y,\bxi)- v_n(x,y,G)) + R_{n,\eps}
	\end{equation} 
	where 
	\begin{equation}\label{key}
	\lim_{n} \frac{R_{n, \ep}}{N} = 0.\nonumber
	\end{equation}
\end{lemma}
In fact, we will show that $R_{N,\eps} = o(1)$. As a consequence, the comparison to the gaussian  variance for an arbitrary model with $\bxi$ coefficients reduces to 
\begin{align}\label{d-d}\var(N_{\bxi})- \var(N_G) =& \sum_{|s-t| \leq 1}(\cov(N_{I_s ,\bxi},N_{I_t ,\bxi})-\cov(N_{I_s ,G},N_{I_t ,G})) \\&\label{off-d}+ \sum_{|s-t| > 1} \int_{I_s}\int_{I_t} (v_n(x,y,\bxi)- v_n(x,y,G))+o(N).\end{align}

\begin{proof}(of Lemma \ref{lm:v1}) We adapt the proof of \cite[Lemma 4.2]{BCP} to our setting. We will let $\delta = N^{-\theta}$ in this proof, where $\theta$ is sufficiently large, such as $\theta>5\delta_{\ast}^{-1}$. Let 
$$\delta_{[a, b], \bxi}:= \min_{x\in [a, b]} \{|P_n(a, \bxi)|, |P_n(b, \bxi)|, |P_n(x, \bxi)|+|P_n'(x, \bxi)|\}.$$

By Kac-Rice formula, for any interval $[a, b]$, the number of zeros of $P_n(x, \bxi)$ in the interval $[a, b]$ is given by 
 \begin{equation}\label{KacRice}
 N_{[a, b], \bxi} = \int_{a}^{b} |P_n'(x, \bxi)|\textbf{1}_{|P_n(x, \bxi)|\le \delta} \frac{dx}{2\delta} \quad \text{ if } \delta\le \delta_{[a, b], \bxi}.
 \end{equation}

To prove Lemma \ref{lm:v1}, it suffices to show that for any $(s,t)\in D_{\operatorname{off}}$ 
\begin{equation}\label{eq:v1:1}
\E N_{I_s,\bxi} N_{I_t,\bxi} = \int_{I_s\times I_t} \E \Phi_\delta (x, \bxi) \Phi_\delta(y, \bxi) dx dy + O(\ep_{s, t}) 
\end{equation}
and
\begin{equation}\label{eq:v1:2}
\E N_{I_s,\bxi} \E N_{I_t,\bxi} = \int_{I_s\times I_t} \E \Phi_\delta (x, \bxi) \E \Phi_\delta(y, \bxi) dx dy + O(\ep_{s, t}) 
\end{equation}
where 
\begin{equation}\label{key}
\sum_{(s, t)\in \mathcal D_{\operatorname{off}}} \ep_{s, t} = o(N).\nonumber
\end{equation}

Since the proof of \eqref{eq:v1:1} and \eqref{eq:v1:2} are similar, we shall now only prove \eqref{eq:v1:1}. 

By \eqref{KacRice}, 
\begin{equation} 
\E N_{I_s,\bxi} N_{I_t,\bxi} \textbf{1}_{\delta\le \min\{\delta_{I_s, \bxi}, \delta_{I_t, \bxi}\}} = \int_{I_s\times I_t} \E \Phi_\delta (x, \bxi) \Phi_\delta(y, \bxi)\textbf{1}_{\delta\le \min\{\delta_{I_s, \bxi}, \delta_{I_t, \bxi}\}}  dx dy. \nonumber
\end{equation}
Thus, by setting
\begin{equation}\label{key}
\ep'_{s, t} = \E N_{I_s,\bxi} N_{I_t,\bxi} \textbf{1}_{\delta > \min\{\delta_{I_s, \bxi}, \delta_{I_t, \bxi}\}} 
\end{equation}
and 
\begin{equation}\label{key}
\ep''_{s, t} = \int_{I_s\times I_t} \E \Phi_\delta (x, \bxi) \Phi_\delta(y, \bxi)\textbf{1}_{\delta > \min\{\delta_{I_s, \bxi}, \delta_{I_t, \bxi}\}}  dx dy,
\end{equation}
we are left to show that 
\begin{equation}\label{eq:v1:ep'}
\sum_{(s, t)\in \mathcal D_{\operatorname{off}}} \ep'_{s, t} = o(1)
\end{equation}
and
\begin{equation}\label{eq:v1:ep''}
\sum_{(s, t)\in \mathcal D_{\operatorname{off}}} \ep''_{s, t} = o(1).
\end{equation}

For \eqref{eq:v1:ep'}, using the fact that the number of real roots inside $[c_1\sqrt{n},c_2\sqrt{n}]$ is at most $n$ deterministically, we get that
\begin{eqnarray}\label{eq:v1:ep':5}
\ep'_{s, t} &\ll& n^{2}\P\left (\delta > \min\{\delta_{I_s, \bxi}, \delta_{I_t, \bxi}\}\right )\le n^{2}\P\left (\delta >  \delta_{I_s, \bxi} \right ) + n^{2}\P\left (\delta >  \delta_{I_t, \bxi} \right ).
\end{eqnarray}
Let $a, b$ be the endpoints of $I_s$. We have
\begin{eqnarray} \label{eq:v1:ep':2}  
\P\left (\delta > \delta_{I_s, \bxi} \right )\le  \P\left (|P_n(a)|<\delta\right ) + \P\left (|P_n(b)|<\delta\right ) + \P\left (\min_{x\in I_s} |P_n(x)|+|P_n'(x)|<\delta\right) \le N^{-\theta},
\end{eqnarray}
where we used Lemma \ref{lemma:smallball:KR} in the last estimate.

Similarly for $I_t$. Thus, from \eqref{eq:v1:ep':5}, we have $\ep'_{s, t}\ll n^{2} N^{-\Theta} \le n^{2} n^{\sigma_{\ast} \theta} \le n^{-3}$, which gives \eqref{eq:v1:ep'}.

For \eqref{eq:v1:ep''}, we argue similarly using the observation in \cite[Inequality (4.2)]{BCP} that, deterministically, 
\begin{equation}\label{eqn:key}
\int_{I_s} \Phi_\delta(x, \bxi) dt \le 1 + N_{I_s,\bxi}\le n+1 \quad\text{and}\quad \int_{I_t} \Phi_\delta(x, \bxi) dt \le 1 + N_{I_t,\bxi} \le n+1.
\end{equation}

\end{proof} 

As with the expectation, we proceed by Edgeworth expansion.

\begin{lemma}\label{lemma-offd} For every constant $\eps>0$ we have 
$$\lim_{n} \frac{1}{N} \int_{D_{\operatorname{off}}} \big(v_n(x,y,\bxi) -v_n(x,y,G)\big) dx dy =  r_\eps$$
with $|r_\eps| = O( N^{-1/2} + e^{-\Theta(N^{\eps/2})})$.
\end{lemma}

\begin{proof} Firstly, notice that the diagonal blocks $D_{\operatorname{off}}$ are made up by $I_s \times I_t$ with $|s-t|>1$. Following the formula in RHS \eqref{eq:def:vn2}, by Proposition \ref{prop:EW:delta},  
\begin{align*}
&\int_{I_s} \int_{I_t}\E \Psi_{\delta}\left (S_n(x,y,\bxi)\right)-\E \Psi_{\delta}\left (S_n(x,y,G)\right)dxdy\\
&= \int_{I_s} \int_{I_t}\int_{\R^4}\Phi_\delta (t_1,t_2,t_3,t_4)\prod_{i=1}^4\phi(t_i)(\frac{1}{N}\Gamma_{n,2}'+\frac{1}{N}\Gamma_{n,2}'')dxdy + O(|I_{s}||I_{t}|/N^{3/2}).
\end{align*}
Note that when we sum over $(s,t) \in D_{off}$, the error bound becomes $O(|I_{W}|^{2}/N^{3/2}) = O(N^{1/2})=o(N)$, and hence we can omit them from now on.

Here we  note that $\Gamma_{n,2}'$ (resp. $\Gamma_{n,2}''$) is decomposed into  terms $\Gamma_{n,2}'(x)+\Gamma_{n,2}'(y)+\Gamma_{n,2}'(x,y)$ depending  on whether the partitions involve $x$ only, $y$ only, or both. The other term to be analyzed by \eqref{eq:def:vn2} crucially lacks mixed terms of the form $\Gamma_{n,2}'(x,y)$ or $\Gamma_{n,2}''(x,y)$ at the order of $\frac{1}{N}$. Indeed, by the second conclusion of Proposition \ref{prop:EW:delta} we have that
\begin{align*}\E \Phi_\delta(x,\bxi) \E \Phi_\delta(y,\bxi) &=\int_{\R^2}\Phi_\delta (t_1,t_2)\prod_{i=1}^2\phi(t_i)(1+\frac{1}{N}\Gamma_{n,2}'(x)+\frac{1}{N}\Gamma_{n,2}''(x) + O(\frac{1}{N^{3/2}}))\nonumber\\
&\times \int_{\R^2}\Phi_\delta (t_3,t_4)\prod_{i=1}^2\phi(t_i)(1+\frac{1}{N}\Gamma_{n,2}'(y)+\frac{1}{N}\Gamma_{n,2}''(y)+ O(\frac{1}{N^{3/2}}))\nonumber\\
&=\int_{\R^4}\Phi_\delta (t_1,t_2,t_3,t_4)\prod_{i=1}^4\phi(t_i)\left(1+\frac{1}{N}\Gamma_{n,2}'(x)+\frac{1}{N}\Gamma_{n,2}''(x)+\frac{1}{N}\Gamma_{n,2}'(y)+\frac{1}{N}\Gamma_{n,2}''(y)\right) \nonumber \\
& + \CR_{n}(x,y),\end{align*}
where by \eqref{eqn:key} 
$$\sum_{(s,t) \in D_{off}} \int_{I_{s} \times I_{t}} \CR_{n}(x,y) =O(N^{1-\eps} \times \E N_{I_{W}}/N^{3/2}) = O(N^{1/2}).$$

Ignoring the $\CR_{n}$ terms, the main contribution of this second term is given by
\begin{align}
    \int_{I_s} \int_{I_t} &\E \Phi_\delta(x,\bxi) \E \Phi_\delta(y,\bxi)- \E \Phi_\delta(x,G) \E \Phi_\delta(y,G)\nonumber\\
    &\approx\int_{I_s} \int_{I_t}\int_{\R^4}\Phi_\delta (t_1,t_2,t_3,t_4)\prod_{i=1}^4\phi(t_i)(\frac{1}{N}\Gamma_{n,2}'(x)+\frac{1}{N}\Gamma_{n,2}''(x)+\frac{1}{N}\Gamma_{n,2}'(y)+\frac{1}{N}\Gamma_{n,2}''(y)).\end{align}

So after canceling ``univariate" terms the whole thing reduces to: \begin{align}\label{vn-final}
\int_{I_s} \int_{I_t}v_n(x,y,\bxi)- v_n(x,y,G)=  \int_{I_s} \int_{I_t}\int_{\R^4}\Phi_\delta (t_1,t_2,t_3,t_4)\prod_{i=1}^4\phi(t_i)(\frac{1}{N}\Gamma_{n,2}'(x,y)+\frac{1}{N}\Gamma_{n,2}''(x,y)).\end{align}

Next we proceed to analyze these terms in \eqref{vn-final} more carefully. Recall, from expectation, that:

\begin{align}
    \Gamma_{n,2}'(x,y,t)&=\frac{1}{24}\sum_{|\gamma|=4}c_n(x,y,\gamma)H_{\gamma}(t),\\
     \Gamma_{n,2}''(x,y,t)&=\frac{1}{72}\sum_{|\alpha|=3}\sum_{|\beta|=3}c_n(x,y,\alpha)c_n(x,y,\beta)H_{\{\alpha,\beta\}}(t).
\end{align}
Similar to the previous section, we have the following equation:
$$\int_{I_s} \int_{I_t}v_n(x,y,\bxi)- v_n(x,y,G)=\sum_{}\int_{\R^4}\Phi_\delta(t)H^{*}(t)\mathbf{\phi}(t)dt \cdot \int_{I_s} \int_{I_t} c_n(x,y,*) dxdy.
$$ Here $c_n(x,y,*)$ and $H^*$ stand for all the linear combination of the cumulants and Hermite polynomials respectively in $\Gamma_{n,2}$.

As before, the Hermite terms contribute upon integration w.r.t. $t=(t_1,t_2,t_3,t_4)$ via terms of the form

\begin{align}\label{Hermite}
    \int_{\R^4}\Phi_\delta(t)H_{\gamma}(t)\mathbf{\phi}(t)dt&=\int_{\R}|t_1|H_{\gamma_1}(t_1)\phi(t_1)dt_1 \cdot  \int_{\R}|t_3|H_{\gamma_3}(t_3)\phi(t_3)dt_3 \nonumber\\
    &\times \lim_{\delta\rightarrow0}\int_{\R}\frac{1}{2\delta}\mathbbm{1}_{|t_2|<\delta}H_{\gamma_2}(t_2)\phi(t_2)dt_2\cdot \lim_{\delta\rightarrow0}\int_{\R}\frac{1}{2\delta}\mathbbm{1}_{|t_4|<\delta}H_{\gamma_2}(t_4)\phi(t_4)dt_4.
\end{align}

Here $\phi$  is the standard univariate/multivariate gaussian weight and the formula extends to the setting of $\Gamma''$ simply by replacing $\gamma_i\rightarrow \alpha_i+\beta_i$ in the Hermite weights.  \\

As with expectation, the nonzero terms in the $t$-integrals amount to order of constant contributions so our main focus is on the order of the integrals of $c_n(x,y,\gamma)$ and $c_n(x,y,\alpha)c_n(x,y,\beta)$ w.r.t. $(x,y)$.

\begin{enumerate}
    \item $\Gamma'$: $|\gamma|=4$, no restrictions.  Thus we can write a generic term as $$c_n(x,y,\gamma)=\sum_{i=1}^nb_i(x)^{\gamma_{1}}c_i(x)^{\gamma_2}b_i(y)^{\gamma_3}c_i(y)^{\gamma_4}.$$ 
Here $\gamma_1+\gamma_2>0$ and $\gamma_3+\gamma_4>0$.\\

    \item $\Gamma''$: $|\alpha|=3$. The usual parity considerations from the Hermite integrals in \eqref{Hermite} mean that products of "simple univariates" corresponding to terms of the form $c_n(x, \alpha)c_n(y,\beta)$ will vanish upon integration w.r.t. $t$. Indeed for any combination of $\alpha,\beta$ from $(3,0), (2,1), (1,2), (0,3)$ at least one of the factors in \eqref{Hermite} vanishes. \\

\end{enumerate}

The following lemma provides an exponential bound on the terms of the form $c_n(x,y,\gamma)$ ($|\gamma|=4$) and $c_n(x,y,\alpha)$ ($|\alpha|=3$). Unlike with expectaion and Lemma \ref{E-lemma} we just need the vanishing bound and not an exact formula. \\

\begin{lemma}\label{lemma-c} Assume that $x\in I_{s}, y\in I_{t}$ for $|s-t|>1$. Assume that $\al_{1},\dots, \al_{4}$ are non-negative integers such that $\al_{1}+\al_{2}>0$ and $\al_{3}+\al_{4}>0$. Then we have
\begin{align}
   \sum_{i=1}^nb_i(x)^{\alpha_1}c_i(x)^{\alpha_2}b_i(y)^{\alpha_3}c_i(y)^{\alpha_4}=O(e^{-\Theta(n^{\eps/2})}).
\end{align}
\end{lemma}

\begin{proof} This follows from Claim \ref{claim:Weyl:ix}. If $|x-y| \ge N^{\eps}$ then $[x^{2}-N^{\eps/2} x, x^{2}+ N^{\eps/2}x]$ is disjoint from $[y^{2}-N^{\eps/2} y, y^{2}+ N^{\eps/2}y]$.
\end{proof}

Now we finish the proof of the Lemma \ref{lemma-offd} case. For the $|\gamma|=4$ terms we can use Lemma \ref{lemma-c} directly so integration w.r.t. $(x,y)$ at most adds $O(n^c)$ factors and we have a contribution

$$\int_{I_s}\int_{I_t}c_n(x,y,\gamma)dxdy =\int_{I_s}\int_{I_t}O(e^{-\Theta(N^{\eps/2})}) =O(e^{-\Theta(N^{\eps/2})}).$$

 For terms of the form  $c_n(x,y,\alpha)c_n(x,y,\beta)$ with $|\alpha|=|\beta|=3$, we apply Lemma \ref{lemma-c} to each of the factors to obtain a contribution once again:
 $$\int_{I_s}\int_{I_t}c_n(x,y,\alpha)c_n(x,y,\beta) = \int_{I_s}\int_{I_t}(O(e^{-\Theta(N^{\eps/2})}))^2 = O(e^{-\Theta(N^{\eps/2})}).$$

 For terms of the form  $c_n(x,\alpha)c_n(x,y,\beta)$ or $c_n(x,y, \alpha)c_n(y,\beta)$ with $|\alpha|=|\beta|=3$, we apply Lemma \ref{lemma-c} to, say, $c_n(x,\alpha)$ which is polynomial in $x$ and thus its contribution w.r.t. is at most polynomial also.  Lemma \ref{lemma-c} then takes care of the other factor making it  exponentially small. 

$$\int_{I_s}\int_{I_t}c_n(x,\alpha)c_n(x,y,\beta) = \int_{I_s}\int_{I_t}O(n)O(e^{-\Theta(N^{\eps/2})})= O(e^{-\Theta(N^{\eps/2})}).$$

Since each combination vanish to exponentially small, while the Hermite polynomial parts is of constant order, we finished the proof.

\end{proof}

\subsection{Treatment for the diagonal term}\label{sub:diag} We then control the \textbf{diagonal} term in (\ref{d-d}). The following lemma shows that it has small contribution as well. 

\begin{lemma}\label{off}
    There exists a small constant $\delta$ such that,
    $$\sum_{|s-t| \leq 1}(\cov(N_{I_s ,\bxi},N_{I_t ,\bxi})-\cov(N_{I_s ,G},N_{I_t ,G})) \ll N^{1/2-\delta}
    $$
\end{lemma}

In order to proceed the proof, we further slice each $I_i$ into smaller intervals with length 1 and re-index those into $J_1,...$ ($J_1$ would be the first slice in $I_1$). We call such a pair of intervals \textbf{close} if they lie in $I_i$, $I_j$ with $|i-j| \leq 1$. Notice that there are $3N^{1/2-\eps}$ pairs of diagonal $I_i$, $I_j$, so there will be $(N^{\eps})^2(3N^{1/2-\eps}) = O(N^{1/2+\eps})$ pairs of close $J_k$, $J_p$. We shall use the following lemma:
\begin{lemma}\label{off2}There exists a constant $c_0$ such that for all $k, p$ that makes $J_k,J_p$ close,
	\begin{equation}\label{eq:localuniv:1}
	\E N_{J_k ,\bxi} N_{J_p ,\bxi} - \E N_{J_k ,G}  N_{J_p ,G} \ll N^{-2c_0},
	\end{equation}
	\begin{equation}\label{eq:localuniv:2}
	\E N_{J_k ,\bxi} - \E N_{J_k ,G} \ll N^{-2c_0},
	\end{equation}
	and
	\begin{equation}\label{eq:localuniv:3}
	\E N_{J_k ,\bxi}\ll 1, \E N_{J_k ,G}\ll 1.
	\end{equation}	
\end{lemma}
This result is a consequence of Theorem \ref{theorem:universality} applied to the Weyl case.  A proof of this result will be given in Section \ref{section:NgV} for the reader's convenience, by applying \cite{ONgV}.

 \begin{proof}[Proof of Lemma \ref{off}]
     Assume lemma \ref{off2}, we have for all close index pairs $k, p$, 
	\begin{equation}\label{key}
	\E N_{J_k ,\bxi}\cdot \E N_{J_p ,\bxi} -\E N_{J_k ,G} \cdot \E N_{J_p ,G}\ll N^{-2c_0}\nonumber
	\end{equation}
where we used the triangle inequality, \eqref{eq:localuniv:2}, and \eqref{eq:localuniv:3}.
Combining this with \eqref{eq:localuniv:1}, we obtain 
\begin{equation}\label{eq:cov:compare} 
\cov(N_{J_k ,\bxi},N_{J_p ,\bxi})-\cov(N_{J_k ,G},N_{J_p ,G}) \ll n^{-2c_0} \quad \text{for all close } k, p,
\end{equation}
and in particular when $k=p$, we have the variance bounded, too.
Combining this bound with counting, we obtain
\begin{equation}\label{eq:v2:last}
\sum_{|s-t| \leq 1}(\cov(N_{I_s ,\bxi},N_{I_t ,\bxi})-\cov(N_{I_s ,G},N_{I_t ,G})) \ll n^{-2c_0}\#\{\text{Close }(k, p)\} = n^{1/2+\eps-2c_0} =O(N^{1/2 -c_{0}}).
\end{equation} 
provided that $\eps<c_{0}$.
 \end{proof}

\subsection{Proof of \eqref{eqn:varmain:2}} Let $\eps>0$ be arbitrary, we'll show that 
$$|\mathrm{Var}\, N_{[0,(1-c)\sqrt{n}],\boldsymbol{\xi}}
    -\mathrm{Var}\, N_{[0, (1-c)\sqrt{n}],G}| \le \eps \sqrt{n}.$$
 Let $c'=c'(\eps)$ be chosen sufficiently small, we will decompose $[0, (1-c)\sqrt{n}]$ into 
 $$[0, (1-c)\sqrt{n}] = \cup_{i=0}^{\ell} [2^{i}n^{\sigma_{\ast}}, 2^{i+1} n^{\sigma_{\ast}})  \cup [c'\sqrt{n} , (1-c)\sqrt{n}]  =: \cup_{i=0}^{\ell} I_{i} \cup J,$$
  where $2^{\ell+1} n^{\sigma_{\ast}} = c' \sqrt{n}$. 
 For convenience let
$$X_{\xi} :=N_{J,\xi},\quad \quad \quad  \mbox{} Y_{i,\xi} := N_{I_{i},\xi},$$ 
and 
$$Y_\xi := \sum_{i=1}^\ell Y_{i,\xi}.$$
Thus 
$$N_{[0,(1-c)\sqrt{n}],\boldsymbol{\xi}} = X_\xi + Y_{0,\xi} + Y_\xi.$$
By \eqref{eqn:varmain:1},
$$\Var(X_{\xi}) = \Var(X_{G})+ o(\sqrt{n})$$
and so by \cite[Lemma 4]{DV} for the gaussian case (i.e. Theorem \ref{thm:variance:G})
\begin{equation}\label{eqn:X_xi}
\Var(X_{\xi}) = (1-c-c')C_{W}+o(1))\sqrt{n}.
\end{equation}

For $1\le i\le \ell$, again by \eqref{eqn:varmain:1} and \cite[Lemma 4]{DV} for the gaussian case,  
$$\Var(Y_{i,\xi}) = \Var(Y_{i,G}) + o(|I_{i}|) = (C_{W}+o(1)) 2^{i} n^{\sigma_{\ast}}.$$ 
Thus for each $1\le i,j\le \ell+1$
$$|\E((Y_{i,\xi}-\E Y_{i,\xi})(Y_{j,\xi}-\E Y_{j,\xi}))| \le \sqrt{ \var(Y_{i,\xi}) \var (Y_{j,\xi})} \le (C_{W}+o(1)) 2^{i/2} 2^{j/2} n^{\sigma_{\ast}}.$$
As such, summing over all $1\le i \le j\le \ell$, as $\sum_{i=0}^{\infty} 2^{-i/2} \le 4$
$$\sum_{1\le i \le j \le \ell+1}|\E((Y_{i,\xi}-\E Y_{i,\xi})(Y_{j,\xi}-\E Y_{j,\xi}))| \le 4 (C_{W}+o(1))\sum_{j=1}^{\ell} 2^{j} n^{\sigma_{\ast}} = 8(C_{W}+o(1))c' \sqrt{n}.$$
Thus  
$$|\Var(Y_\xi) - \sum_{i=1}^\ell \Var(Y_{i,\xi})| \le \sum_{1\le i \le j \le \ell+1}|\E((Y_{i,\xi}-\E Y_{i,\xi})(Y_{j,\xi}-\E Y_{j,\xi}))|\le 8(C_{W}+o(1))c' \sqrt{n}.$$
We thus obtain that 
\begin{equation}\label{eqn:Y_xi}
\Var(Y_\xi)  \le 10(C_{W}+o(1))c' \sqrt{n}.
\end{equation}
For $Y_{0,\xi}$,  \eqref{eqn:varmain:1} does not apply, but instead we can apply \cite[Subsection 12.1]{TV} to obtain
\begin{equation}\label{eqn:Y_0xi}
\Var(Y_{0,\xi}) = O(n^{\sigma}),
\end{equation}
where $\sigma>0$ is small if $\sigma_{\ast}$ is sufficiently small.

Using this fact, together with the bounds from \eqref{eqn:X_xi} and \eqref{eqn:Y_xi} we obtain
\begin{equation}
|\var(X_\xi+Y_{0,\xi} + Y_\xi) - \var(X_\xi)-\var(Y_{0,\xi}) -\var(Y_\xi) | \le \sqrt{ 10 c'}(C_{W}+o(1))\sqrt{n}.
\end{equation}
By triangle inequality,
$$|\var(X_\xi+Y_{0,\xi} + Y_\xi) -C_W(1-c)\sqrt{n}| \le \sqrt{ 10 c'}(C_{W}+o(1))\sqrt{n} + (c'+o(1))\sqrt{n} \le \eps \sqrt{n},$$
provided that $c'$ is sufficiently small.

 \section{Further discussion}\label{section:discussion}

In line with Question~\ref{conj:exp} and Conjecture~\ref{conj:var}, perhaps the most intriguing next direction is to investigate the fluctuations of the number of real zeros $N_{\mathbb{R},\bxi}(F_{n})$ for random polynomials with general coefficients.

There has been considerable recent progress concerning the Central Limit Theorem (CLT)---and even certain non-CLT behaviors---for gaussian polynomial models; see, for instance, \cite{AnL, ADL, AL, DV, GW, NS-complex, ONgV} and the references therein.  
The techniques employed vary from model to model, but one of the most powerful and unifying tools is the Wiener chaos decomposition, which expresses a functional of gaussian variables as a sum of orthogonal components with respect to the gaussian measure. It is natural to conjecture that a CLT-type fluctuation should persist for general coefficient distributions.

\begin{conjecture}[Central Limit Theorem for general coefficients]\label{conj:CLT}
For all random polynomial models considered in Section \ref{sect:intro}, under the assumption that \( \xi \) is subgaussian with mean zero and variance one, we have
\[
    \frac{N_{\mathbb{R},\bxi} - \mathbb{E} N_{\mathbb{R},\bxi}}{\sqrt{\mathrm{Var}\, N_{\mathbb{R},\bxi}}}
    \xrightarrow{d} \BN(0,1)
    \qquad \text{as } n \to \infty.
\]
\end{conjecture}

For the remainder of this section we discuss the Weyl model in more detail and outline a potential approach toward proving Conjecture~\ref{conj:CLT} for $I=[c_1 \sqrt{n}, c_2\sqrt{n}]$.  
Recall that when the coefficients \( \xi_i \) are i.i.d.\ complex gaussian, the Weyl polynomial ensemble corresponds to the truncation of the gaussian Entire Function $P_{\infty}(x)
    = \sum_{i=0}^{\infty} \xi_i \frac{x^i}{\sqrt{i!}}$, whose zeros can be shown to be invariant under the isometries of the complex plane~\cite{HKPV}.

The fluctuation of the number of complex zeros of \( P_{\infty} \) was established by Nazarov and Sodin~\cite{NS-complex}, while the fluctuation of the number of real zeros for real gaussian coefficients was obtained by Do and Vu~\cite[Corollary~1]{DV}.  
In the same work, the authors extended their result to the polynomial ensemble \( P_n \), thereby establishing Conjecture~\ref{conj:CLT} for the Weyl model with gaussian coefficients.

Roughly speaking, the proofs in \cite{DV} rely on a version of the \emph{method of moment}, showing that all cumulants of order at least three vanish asymptotically.  
This reduction is achieved through a detailed analysis of correlation functions.  
However, these arguments are highly specific to the gaussian setting.  

In view of our current development, we outline below a framework for extending such fluctuation results to general coefficients using higher-moment expansions.

Let
\[
    \wb{M} = \mathbb{E} N_{I_{W},\boldsymbol{\xi}},
    \qquad
    M' = \frac{\wb{M}}{|I_{W}|}.
\]
Since both \( M \) and \( |I_W| \) are of order \( N \asymp \Theta(\sqrt{n}) \), we have \( M' = \Theta(1) \).  
Then for each fixed integer \( k \),
\begin{align*}
\lim_{n \to \infty}
\mathbb{E}\!\left[
    \frac{N_n(\boldsymbol{\xi}) - M}{\sqrt{N}}
\right]^{k}
&=
\lim_{n \to \infty}
\frac{1}{N^{k/2}}
\mathbb{E}\!\left[
    \left(
        \frac{1}{2\delta}
        \int_{I} P_n'(x,\boldsymbol{\xi}) \,
        \mathbf{1}_{\{|P_n(x,\boldsymbol{\xi})| < \delta\}} \, dx
        - M
    \right)^{\!k}
\right] \\
&=
\lim_{n \to \infty}
\frac{1}{N^{k/2}}
\mathbb{E}\!\left[
    \int_{I_W}
    \bigg(
        \frac{1}{2\delta} P_n'(x,\boldsymbol{\xi})
        \mathbf{1}_{\{|P_n(x,\boldsymbol{\xi})| < \delta\}} - M'
    \bigg) dx
\right]^{\!k} \\
&=
\lim_{n \to \infty}
\frac{1}{N^{k/2}}
\mathbb{E}\!\left[
    \int_{I_W^k}
    \prod_{r=1}^{k}
    \big(
        \Phi_{\delta}(S_n(x_r,\boldsymbol{\xi}))
        - \mathbb{E}\Phi_{\delta}(S_n(x_r,\boldsymbol{\xi}))
    \big)
    \, dx_1 \cdots dx_k
\right],
\end{align*}
where \( \Phi_{\delta} \) is a smoothed indicator of small values of \( P_n(x,\boldsymbol{\xi}) \), and \( S_n(x,\boldsymbol{\xi}) \) represents the standardized polynomial process.

\medskip

As in Section \ref{sect:var}, for a small parameter \( \varepsilon > 0 \), we partition the interval \( I \) into subintervals of length \( N^{\varepsilon} \):
\begin{equation}\label{def:Ik}
    I_k := [k n^{\varepsilon}, (k+1) n^{\varepsilon}) \subset I,
    \qquad k = 1, \dots, k_N = n^{1/2-\varepsilon}.
\end{equation}
We can then rewrite the above expression as
\begin{equation}\label{eqn:sum:CLT}
    \lim_{n \to \infty}
    \frac{1}{N^{k/2}}
    \sum_{(I_{i_1}, \dots, I_{i_k})}
    S_{i_1,\dots,i_k},
\end{equation}
where, for each tuple \( (I_{i_1}, \dots, I_{i_k}) \in (I_1, \dots, I_{k_N})^k \),
\[
    S_{i_1,\dots,i_k}
    :=
    \mathbb{E}\!\left[
        \int_{I_{i_1} \times \cdots \times I_{i_k}}
        \prod_{r=1}^{k}
        \big(
            \Phi_{\delta}(S_n(x_r,\boldsymbol{\xi}))
            - \mathbb{E}\Phi_{\delta}(S_n(x_r,\boldsymbol{\xi}))
        \big)
        dx_1 \cdots dx_k
    \right].
\]

\medskip

When the intervals \( I_{i_1}, \dots, I_{i_k} \) are \( N^{\varepsilon} \)-separated (that is, \( \mathrm{dist}(I_i, I_j) \ge N^{\varepsilon} \) for all \( i \neq j \)), it seems plausible to extend the techniques of Section~\ref{section:LCD:W} (more precisely Subsection \ref{subsection:W:4}) to handle this regime, leading to an Edgeworth-type expansion.

The main difficulty arises when some of the points \(x_1, \dots, x_k\) (or the intervals \(I_{1}, \dots, I_{k}\)) are \emph{not} \(N^{\varepsilon}\)-separated.  
In this situation, two or more points may lie close to each other, rendering the techniques from Section~\ref{section:LCD:W} no longer effective.  
We emphasize that such scenarios cannot be treated as negligible error terms, since it appears that the main contributions to~\eqref{eqn:sum:CLT} come from those sets \(S_{i_1, \dots, i_k}\) in which there are exactly \(k/2\) pairs satisfying \(i_{l} = i_{m}\).  
In connection with the present note, this includes developing more precise estimates for the diagonal terms (from Subsection~\ref{sub:diag}) without relying on the replacement methods of~\cite{ONgV}.  

The essential challenge, therefore, is to establish a corresponding Edgeworth expansion in this near-collision regime (among other related cases).  
This requires analyzing the decay of the characteristic function in this regime, as well as extending Theorem~\ref{thm:EW:linear} to the case where \(\sigma \to 0\) as \(n \to \infty\).  
While technically demanding, we expect this obstacle to be quantitative rather than conceptual.  
We hope to provide a detailed treatment of these issues in the near future.

 {\bf Acknowledgements.} The authors are grateful to O. Nguyen for many helpful comments.

\appendix

\section{Proof of Lemma \ref{off2}}\label{section:NgV}

In order to prove this result, we will need the following ingredients\footnote{We cite here a precise statement from~\cite{ONgV}, although only a special case of it will be used in our argument.}.

\begin{condition}\label{cond:c1}
    
Two sequences of real random variables
\[
(\xi_1,\dots,\xi_n)
\quad\text{and}\quad
(\tilde\xi_1,\dots,\tilde\xi_n)
\]
are said to satisfy this condition if there exist constants
\[
N_0\in\mathbb{N},\quad \tau>0,\quad 0<\varepsilon<1
\]
such that:
\begin{enumerate}[(i)]
  \item \textbf{Uniformly bounded $(2+\varepsilon)$ central moments:}\\
    The variables $\{\xi_i\}$ (and likewise $\{\tilde\xi_i\}$) are independent, satisfy
    \[
      \E\bigl[(\xi_i-\E\xi_i)^2\bigr]=1,
      \qquad
      \E\bigl|\xi_i-\E\xi_i\bigr|^{2+\varepsilon}\le\tau,
      \quad
      1\le i\le n,
    \]
    and similarly for each $\tilde\xi_i$.
  \item \textbf{Matching moments up to second order (with finitely many exceptions):}\\
    \begin{itemize}
      \item For all indices $i\ge N_0$, the first two moments agree exactly:
      \[
        \E[\xi_i]=\E[\tilde\xi_i],
        \quad
        \E[\xi_i^2]=\E[\tilde\xi_i^2]=1.
      \]
      \item For the finitely many exceptions $1\le i< N_0$, the means remain close:
      \[
        \bigl|\E[\xi_i]-\E[\tilde\xi_i]\bigr|\le\tau.
      \]
    \end{itemize}
\end{enumerate}
\end{condition}
We then have the following theorem on the universality of pair correlation of real zeros of Weyl polynomials:
\begin{theorem}\cite[Theorem 5.2]{TV}(see also \cite{ONgV}, Theorem 2.6, Theorem 5.1)\label{univ}
Let $(\xi_j)_{j=0}^n$ be independent real random variables satisfying \ref{cond:c1}, and let $(\tilde\xi_j)_{j=0}^n$ be i.i.d.\ standard real gaussians.  For Weyl polynomials
\[
  P(z)\;=\;\sum_{j=0}^n \xi_j\,\frac{e^{-\frac{z^2}{2}}z^j}{\sqrt{j!}},
  \qquad
  \tilde P(z)\;=\;\sum_{j=0}^n \tilde\xi_j\,\frac{e^{-\frac{z^2}{2}}z^j}{\sqrt{j!}},
\]
and write $\{\zeta_i\}$, $\{\tilde\zeta_i\}$ for their real zero‐sets. 

Let $k \in \Z_{+}, \eps>0, C>0$ be given. Let $x_{1},\dots, x_{k}$ be quantities depending on $n$ with $n^{\eps} \le |x_{i}|\le \sqrt{n}$. Let $F: \R^{k}\to \R$ be any smooth function supported on $[-C,C]^{k}$ such that
\[
  ||\nabla^a F||_\infty\le C
  \text{ for all }0\le a\le 5k+1.
\]
Then one has
\[
  \Bigl|
    \E\Bigl[\sum_{(i_{1},\dots, i_{k})} F(\zeta_{i_{1}},\dots, \zeta_{i_{k}})\Bigr]
    -
    \E\Bigl[\sum_{(i_{1},\dots, i_{k})} F(\tilde\zeta_{i_{1}},\dots, \tilde\zeta_{i_{k}})\Bigr]
  \Bigr|
  \;\le\;
  O(n^{-c_{}}),
\]
where $c_{}$ depends only on $k$.
\end{theorem}
As a consequence, for the first intensity, by lower bound and upper bound the indicator function by smooth function, it follows that  
\begin{corollary} \label{univ-co}
    Under the same hypotheses, let $B \subset [z_{0}-1, z_0+1]$ where $n^{\eps} \le z_{0} \le \sqrt{n}$. Then
\[
  \E N_{B,\bxi}
  \;=\;
   \E N_{B,G}+
  O\bigl(n^{-c_{}}\bigr).
\]
\end{corollary}

We will also need the following lemma.

\begin{lemma}\label{lemma:upperbd} Let $A>0$ be given. Assume that $|z_{0}| \le \sqrt{n}$ and $|z_{0}|$ is sufficiently large. Then for any $M \ge |z_{0}|$, there exists $K$ depending on $A$ such that for any $M \ge |z_{0}|$ we have
$$\P(N_{B(z_{0},1), \xi} \ge M^{2}) \le \frac{K}{M^{A}}.$$
As a consequence, there exists a constant $A_{0}$ such that for $M_{0} \le \sqrt{n}$ and $M_{0}\to \infty$ with $n$, 
$$\E(N_{[0,M_{0}], \xi}) \le M_{0}^{A_{0}}.$$
\end{lemma}

\begin{proof}(of Lemma \ref{lemma:upperbd}) This result follows from \cite{ONgV}. More specifically, the R.H.S bound is obtained by combining \cite[Lemma 12.1]{ONgV} (including its proof) and \cite[Lemma 9.2]{ONgV} and Jensen inequality for the number of complex zeros of $F_{n}(z) = P_{n}(z)/e^{|z_{0}|^{2}/2} e^{(z-z_{0})\wb{z}_{0}}$ inside $B(z_{0},1)$.
\end{proof}

\begin{proof}[Proof of lemma \ref{off2}]

Let $x_k, x_p$ be the midpoint of $J_k$, $J_p$, respectively. Notice that \eqref{eq:localuniv:3} and and \eqref{eq:localuniv:2} automatically follow from Corollary \ref{univ-co}. We now focus on \eqref{eq:localuniv:1}.
    
    Let $\gamma = n^{-s}$ for $s = c/100$ and $c$ be the constant in Theorem \ref{univ}.

	We approximate the indicator function on the interval $[-1, 1]$ by a smooth function $\phi$ satisfying 
	$$\textbf{1}_{[-1+\gamma, 1 - \gamma]}\le \phi\le \textbf{1}_{[-1, 1]}$$
	and 
	$$||{\triangledown^a \phi}||_\infty\ll \gamma^{-a}, \quad\forall 0\le a\le 8.$$
	Let 
	$$F(x, y): = \phi(x-x_k)\phi(y-x_p).$$
	Let 
	$$M_k(\bxi): = \sum_{\mbox{ real zeros } \zeta_{i}} \phi\left (\zeta_i (\bxi) -x_k\right),\quad M_p (\bxi) : = \sum_{\mbox{ real zeros } \zeta_{i}}\phi\left (\zeta_i(\bxi) -x_p\right).$$
	Denote by $M_k(G)$ and $M_p(G)$ the corresponding terms for the gaussian case, i.e., with $\zeta_i(G)$ in place of $\zeta_i(\bxi)$. 
	Apply Theorem \ref{univ} to the function $\gamma ^{8}F$, we obtain
	\begin{eqnarray}\nonumber
	\left |\E\sum F\left (\zeta_i(\bxi), \zeta_j(\bxi)\right) 
	-\E\sum F\left ( \zeta_i(G),   \zeta_j(G)\right) \right |\le C'n^{-c/2}n^{8c/100} \leq C'n^{-c/3},
	\end{eqnarray}
	and so
	\begin{eqnarray}\label{eq:mmmm}
	\left |\E M_k(\bxi) M_p(\bxi) -\E  M_k(G)  M_p(G) \right |\le C'n^{-c/3}.
	\end{eqnarray}
	We shall show that
	\begin{equation}\label{eq:nnmm}
	\E N_{J_k ,\bxi} N_{J_p ,\bxi} - \E M_k(\bxi) M_p(\bxi) = O\left (n^{-s/10}\right ).
	\end{equation}
	The same argument applied to the gaussian case will show that 
	\begin{equation}\label{eq:nnmm:g}
	\E N_{J_k ,G} N_{J_p ,G} - \E M_k(G) M_p(G) = O\left (n^{-s/10}\right ).
	\end{equation}
	Combining \eqref{eq:mmmm}, \eqref{eq:nnmm}, and \eqref{eq:nnmm:g}, we obtain \eqref{eq:localuniv:1} as desired (by choosing the $c_0$ in \eqref{eq:localuniv:1} to be $s/10$). 
	
	To prove \eqref{eq:nnmm}, by Holder's inequality, we have
	\begin{equation}\label{eq:ninjminj}
	( \E N_{J_k ,\bxi} N_{J_p ,\bxi} - \E M_k(\bxi) N_{J_p ,\bxi})^{2} \ll \E (N_{J_k ,\bxi} - M_k(\bxi))^{2} \E N_{J_p ,\bxi}^{2}.
	\end{equation}
	Let 	$N_{\gamma}(\bxi)$ be the number of roots of $P_n(\cdot, \bxi)$ in the union of the intervals $[x_k+1-\gamma, x_k+1]$, $[x_k-1, x_k-1-\gamma]$, $[x_p+1-\gamma, x_p+1]$, and $[x_p-1, x_p-1-\gamma]$. We observe that 
	$$ |N_{J_k ,\bxi} - M_k(\bxi)| \le N_{\gamma}(\bxi).$$ 
 In what follows we will argue as in the proof of Lemma \ref{lemma:upperbd}. By \cite[Equation (50)]{ONgV}, there exists an $x\in J_k$ such that
	\begin{equation}\label{key}
	\P\left (\log |F_n(x, \bxi)|\le -n^{s/10}\right )\ll n^{-100}.\nonumber
	\end{equation}
	By \cite[Equation (48)]{ONgV}, 
	\begin{equation}\label{key}
	\P\left (\log \max_{z\in B(x, 100)}|F_n(z, \bxi)|\ge n^{s/10}\right )\ll n^{-100}.\nonumber
	\end{equation}
	By Jensen's inequality (see, for example, \cite[Equation (8)]{ONgV}), under the event that $\log |F_n(x, \bxi)|\ge -n^{s/10}$ and $\log \max_{z\in B(x, 100)}|F_n(z, \bxi)|\le n^{s/10}$, we have $N_{J_k ,\bxi} \le n^{s/10}$. Thus, 
	\begin{equation}\label{eq:ngamma:1}
	\P\left (N_{J_k ,\bxi}\ge n^{s/10}\right )\ll n^{-100}.
	\end{equation}
	And by \cite[Lemma 8.5]{ONgV} (alternatively, we can also use our small ball estimates developed in Section \ref{sect:sbp} to give a similar estimate), 
	\begin{equation}\label{eq:n:gamma:3s2}
	\P\left (N_{\gamma}(\bxi)\ge 2\right )\ll n^{-3s/2}.
	\end{equation}
	When $N_{\gamma}(\bxi)< 2$, we have $N_{\gamma}(\bxi)^{2} = N_{\gamma}(\bxi)$. Thus,
	\begin{eqnarray}\label{eq:n:gamma}
	\E N_{\gamma}(\bxi)^{2} &\le& \E N_{\gamma}(\bxi)\textbf{1}_{N_{\gamma}(\bxi) < 2} + \E N_{\gamma}(\bxi)^{2}\textbf{1}_{2\le N_{\gamma}(\bxi)\le n^{s/10}} + \E N_{\gamma}(\xi)^{2}\textbf{1}_{n^{s/10}\le N_{\gamma}(\bxi)\le n}\nonumber\\
	&\ll& n^{-3s/2} +  \E N_{\gamma}(\bxi)^{2}\textbf{1}_{2\le N_{\gamma}(\bxi) \le n^{s/10}} + \E N_{\gamma}(\bxi)^{2}\textbf{1}_{n^{s/10}\le N_{\gamma}(\bxi)\le n}\quad \text{by \ref{univ-co}}\nonumber\\
&\ll& n^{-s/2} +n^{-s}+ n^{-98} \ll n^{-s/2} \quad\text{by \eqref{eq:ngamma:1} and \eqref{eq:n:gamma:3s2}},
	\end{eqnarray}
 	 provided that $c$ (and $s$) is small.
	  
	Similarly, 
	\begin{eqnarray}\label{eq:n:gamma:p}
	\E N_{J_p ,\bxi}^2 &\le& \E N_{J_p ,\bxi}^2\textbf{1}_{N_{J_p ,\bxi}\le n^{s/10}} + \E N_{J_p ,\bxi}^2\textbf{1}_{n^{s/10}\le N_{J_p ,\bxi}\le n}\ll n^{s/5}+ n^{-98} \ll n^{s/5}.
	\end{eqnarray}
	
	Plugging \eqref{eq:n:gamma} and \eqref{eq:n:gamma:p} into \eqref{eq:ninjminj}, we get
	\begin{equation} 
	( \E N_{J_k ,\bxi} N_{J_p ,\bxi} - \E M_k(\bxi) N_{J_p ,\bxi})^{2} \ll  n^{-s/2}\E N_{J_p ,\bxi}^2\ll n^{-s/10}.\nonumber
	\end{equation}
	Similarly, 
	\begin{equation} 
	( \E M_k(\bxi) N_{J_p ,\bxi}- \E M_k(\bxi) M_p(\bxi))^{2} \ll n^{-s/2}.\nonumber
	\end{equation}
	Combining these two inequalities gives \eqref{eq:nnmm} and completes the proof.

\end{proof}

\section{Proof of Lemma \ref{E-lemma}}\label{sect:E-lemma}
By Claim \ref{claim:Weyl:ix}, we focus on the region where $i \approx x^2$, and localize the sum:
\[
i = x^2 + k, \quad k \in [-n^{\epsilon}x, n^{\epsilon}x]
\]
The sum becomes:
\[
\sum_{k = -n^\epsilon x}^{n^\epsilon x} e^{-tx^2/2} \cdot \frac{x^{t(x^2 + k)}}{(\sqrt{(x^2 + k)!})^t} (\frac{k}{x})^s.
\]

We apply Stirling’s approximation to expand the factorial:
\begin{align*}
e^{-tx^2/2} \cdot \frac{x^{ti}}{(\sqrt{i!})^t}
&= \exp\left(-\frac{t}{2}x^2 + ti\log x - \frac{t}{2}\log(i!) \right) \\
&= \exp\left(-\frac{t}{2}x^2 + ti\log x - \frac{t}{2}i\log i +\frac{t}{2}i - \frac{t}{4}\log(2\pi i) + O(1/i)\right).
\end{align*}

Substitute $i = x^2 + k$:
\begin{align*}
&= \exp\Big(-\frac{t}{2}x^2 + t(x^2 + k)\log x - \frac{t}{2}(x^2 + k)\log(x^2 + k) \\
&\quad + \frac{t}{2}(x^2 + k) - \frac{t}{4}\log(2\pi(x^2 + k)) + O(1/x^2)\Big).
\end{align*}

Expand logarithms:
\begin{align*}
\log\left(\frac{x}{\sqrt{x^2 + k}}\right) &= -\frac{k}{2x^2} + \frac{k^2}{4x^4} + \cdots \\
\log(x^2 + k) &= \log x^2 + \frac{k}{x^2} - \frac{k^2}{2x^4} + \cdots.
\end{align*}

Using these, the expression simplifies to:
\[
\exp\left( -\frac{t}{4}\frac{k^2}{x^2} - \frac{t}{4}\log(2\pi x^2)  + O\left(\frac{k^3}{x^4}+\frac{k}{x^2} + \frac{1}{x^2} \right) \right).
\]

Let $e(k,x)$ denote the cumulative error term:
\[
e(k,x) = O\left(\frac{k^3}{x^4}+\frac{k}{x^2} + \frac{1}{x^2}\right).
\]
where the first error comes from the Taylor expansion of the first log, second error comes from the Taylor expansion of the second log, and the last error comes from the Stirling approximation.\\

We approximate the sum by an integral:
\[
\sum_{k=-n^\epsilon x}^{n^\epsilon x} \exp\left(-\frac{t}{4}\frac{k^2}{x^2} - \frac{t}{4}\log(2\pi x^2) + e(k,x) \right)(\frac{k}{x})^s \approx (2\pi x^2)^{-\frac{t}{4}} \int_{-n^\epsilon x}^{n^\epsilon x} \exp\left(-\frac{t}{4}\frac{k^2}{x^2} + e(k,x) \right)(\frac{k}{x})^s \, dk.
\]
Notice that the error is 

\[(2\pi x^2)^{-\frac{t}{4}}(\frac{f(n^\epsilon x)+f(-n^\epsilon x)}{2}
\;+\;
\frac{B_{2}}{2!}\bigl(f'(n^\epsilon x)-f'(-n^\epsilon x)\bigr)
\;+\;
\frac{B_{4}}{4!}\bigl(f^{(3)}(n^\epsilon x)-f^{(3)}(-n^\epsilon x)\bigr)
\;+\;\cdots),\]
where $B_i$ is the $i$-th Bernoulli number 
by Euler-Maclaurin formula. This error bound is $O(e^{-\frac{t}{4}{n^{2\epsilon}}})$.

Change variables $L = \frac{k}{x}$, giving:
\[
(2\pi x^2)^{-\frac{t}{4}} \int_{-n^\epsilon x}^{n^\epsilon x} \exp\left(-\frac{t}{4}\frac{k^2}{x^2} + e(k,x) \right)(\frac{k}{x})^s \, dk
= (2\pi x^2)^{-\frac{t}{4}} x \int_{-n^\epsilon}^{n^\epsilon} \exp(-\frac{t}{4}L^2 + e(L,x))L^s \, dL.
\]

This is integrated by approximating \[
(2\pi x^2)^{-\frac{t}{4}} x \int_{-n^\epsilon}^{n^\epsilon} \exp(-\frac{t}{4}L^2 + e(L,x))L^s \, dL \approx (2\pi x^2)^{-\frac{t}{4}} x \int_{-\infty}^{\infty} \exp(-\frac{t}{4}L^2)L^s(1 + e(L,x)) \, dL \]\[ = (2\pi)^{-\frac{t}{4}}(\frac{4}{t})^{\frac{s+1}{2}} \Gamma(\frac{s+1}{2}) x^{-\frac{t-2}{2}}.
\] This error bound is also $O(e^{-\frac{t}{4}{n^{2\epsilon}}})$.


\end{document}